\documentclass[12pt]{article}
\usepackage{amsmath, amsthm, amsfonts, amssymb, color}
\usepackage[utf8]{inputenc}
\usepackage{enumitem}
\usepackage{hyperref}
\usepackage{leftindex}
\usepackage[backend=biber,style=alphabetic,abbreviate=false,giveninits=true,sorting=nty,maxbibnames=99]{biblatex}

\def\doublespace{\baselineskip=14pt}

\newcommand{\E}{\mathcal E}
\newcommand{\R}{\mathbb R}
\newcommand{\N}{\mathbb N}
\newcommand{\Z}{\mathbb Z}

\newcommand{\dom}{\mathcal D}

\newcommand{\univ}{\mathbf e}
\newcommand{\eins}{\boldsymbol 1}

\newcommand{\im}{\operatorname{Im}}

\newcommand{\de}{\mathop{}\!\textnormal{d}}
\newcommand{\grad}{\mathop{}\!\nabla}

\newcommand{\cov}[1]{\mathop\textnormal{Cov}\!_{#1}}

\newcommand{\bracket}[1]{<\!{#1}\!>}

\newcommand{\supp}[1]{\operatorname{supp}[{#1}]}

\newcommand{\euc}{\textnormal{euc}}

\numberwithin{equation}{section}

\newtheorem{thm}{Theorem}[section]
\newtheorem{prop}[thm]{Proposition}
\newtheorem{corollary}[thm]{Corollary}
\newtheorem{lemma}[thm]{Lemma}

\theoremstyle{definition}
\newtheorem{condition}[thm]{Condition}
\newtheorem{remark}[thm]{Remark}
\newtheorem{exa}[thm]{Example}

\newtheorem{test}{Lemma}

\newenvironment{thm(ii)}[1][]
{\enumerate[leftmargin=21pt, itemsep=2pt, parsep=0pt, label=\rm (\roman*), #1]}
{\endenumerate}

\newenvironment{thm(vi)}[1][]
{\enumerate[leftmargin=25pt, itemsep=2pt, parsep=0pt, label=\rm (\roman*), #1]}
{\endenumerate}

\title{Mosco convergence of gradient forms\\ with non-convex potentials  II}

\author{Martin Grothaus\footnote{University of Kaiserslautern-Landau, \textit{grothaus@mathematik.uni-kl.de}}\, and Simon Wittmann\footnote{The Hong Kong Polytechnic university, \textit{simon.wittmann@polyu.edu.hk}}
	\footnote{Simon Wittmann, who was a scientific assistant in the Department of Mathematics at the Technical University of Kaiserslautern at the time this research was conducted, received funding the by the DFG (GR 1809/14-1).}}

\begin{document}
	
	\maketitle
	\thispagestyle{empty}

	\begin{abstract}
		This article provides a scaling limit for a family of skew interacting Brownian motions in the context of mesoscopic interface models. Let $d\in\N$ and height levels $y_1,\dots,y_M\in\R$ be fixed. For each $N\in\N$ we consider a  $k_N$-dimensional, skew reflecting distorted Brownian motion $(X^{N,i}_t)_{i=1,\dots,k_N}$, $t\geq 0$,
		and investigate the scaling limits for $N\to\infty$.
		Apart from a positive symmetric operator $A_N:\R^{k_N}\to\R^{k_N}$ and a bounded continuous function $f:\R\to\R$, the drift includes skew reflections at height levels 
		$\tilde y_j:=N^{1-\frac{d}{2}}y_j$ with intensities $\beta_j/N^d$ for $j=1,\dots,M$.
		The corresponding SDE is given by
		\begin{multline}\label{eqref:zerozero}
			\de X^{N,i}_t=-\big(A_N X^{N}_t\big)_i\de t-\frac{1}{2}N^{-\tfrac{d}{2}-1}\,f\big(N^{\frac{d}{2}-1}X^{N,i}_t\big)\de t
			\\+\sum_{j=1}^M\tfrac{1-e^{-\beta_j/N^d}}{1+e^{-\beta_j/N^d}}\de  l_t^{N,i,  \tilde y_j}
			+\de B_t^{N,i},
		\end{multline}
		where ${(B_t^{N,i})}_{t\geq 0}$, $i=1,\dots, k_N$, are independent Brownian motions
		and $ l_t^{N,i, \tilde y_j}$ denotes the local time of  ${(X^{N,i}_t)}_{t\geq 0}$ at 
		$\tilde y_j$.
		Due to these, the invariant measure of \eqref{eqref:zerozero} is not log-concave.
		The process $(X^{N}_t)_{t\geq 0}$ with components $(X^{N,i}_t)_{t\geq 0}$
		serves as a microscopic model for a $(d+1)$-dimensional dynamic interface over a bounded open domain $D\subset \R^d$, exposed to external forces.
		We prove the weak convergence of the equilibrium laws of 
		\begin{equation*}
			u_t^N=\Lambda_N\circ X^{N}_{N^2t},\quad t\geq 0,
		\end{equation*}
		for $N\to\infty$, choosing suitable injective, linear maps $\Lambda_N:\R^{k_N}\to \{h\,|\,h:D\to\R\}$, called the height maps.
		The scaling limit is a distorted Ornstein-Uhlenbeck process whose state space is the Hilbert space $H=L^2(D,\de z)$.
		We characterize a class of height maps, which contains the most relevant examples from literature,  such that the scaling limit of the dynamic is not influenced by the particular choice of ${(\Lambda_N)}_{N\in\N}$ within that class.
	\end{abstract}

	\tableofcontents
	
	\doublespace
	\section{Introduction}
	
	The reaction-diffusion stochastic partial differential equation of the form
\begin{align}\label{eqn:inispde}
	\frac{\partial u}{\partial t}=-A u-\frac{1}{2}g'(u)+\dot W,	\quad t\geq 0,
\end{align}
over the spatial domain $D=(0,1)$ with Dirichlet boundary conditions
has been analysed in \cite{bounebache14}.
There, the equation is suggested as a model for the random fluctuations of an effective physical interface at a mesoscopic level.
Above, $\dot W(t,z)$, $t\geq 0$, $z\in D$, denotes space-time white-noise and $-A$ is the second-order differential operator $\frac{1}{2}\frac{\partial^2}{\partial z^2}$.
\eqref{eqn:inispde} is an informal equation, since
$g:\R\to\R$ is merely a function of bounded variation and not necessarily differentiable or continuous. However,
$g'(u)$ can be made sense of as a continuous additive functional of the solution $u$.
 This article covers some relevant topics for scaling limits related to \eqref{eqn:inispde}.
In a more general fashion, we consider a spatial domain $D$ which is a bounded, open domain of $\R^d$ with $d\in\N$ and $A$ denotes any self-adjoint operator on $H:=L^2(D,\de z)$ such that $\frac{1}{2}A^{-1}$ is the covariance operator of a Gaussian measure on $H$. 
For such $g$, $A$ and domain $D$ we consider \eqref{eqn:inispde} as an informal equation modelling a mesoscopic dynamic interface.
As a standard example in this text, we take $D=(0,1)^2$ and set $A=-\Delta+\Delta^2$ with operator domain $H^{4,2}(D)\cap H_0^{3,2}(D)$, resulting in the Gaussian measure for a semiflexible polymer model, such as considered in \cite{cipriani20}. 
For $N\in\N$ we construct a system of $k_N$ interacting skew Brownian motions. Their equilibrium laws converge weakly to the stationary solution of the reaction-diffusion equation \eqref{eqn:inispde} under a scaling limit for $N\to\infty$. 
The identification of a diffusion process on $\R^{k_N}$ with a process on $H$ is done via suitable  height maps $\Lambda_N:\R^{k_N}\to H$, $N\in\N$. Different choices for the sequence ${(\Lambda_N)}_{N\in\N}$, such as piecewise constant interpolation or piecewise linear, continuous interpolation might serve for that purpose. We show the invariance of the scaling limit for a general class of height maps which includes the two mentioned types of interpolations.
 
The approximation of an infinite dimensional random variable or an infinite dimensional diffusion process by the corresponding discretized objects is an essential aspect of an interface model in statistical mechanics.
It captures the transition from a microscopic perspective to a macro- or mesoscopic point of view.
In this article,  a discretization of $D$ is considered. We denote by $\overline D$ the topological closure of $D$.
For each $N\in\N$ a finite set of grid points $G_N$ is selected 
as a subset of the lattice with $G_N\subseteq N\overline D\cap\Z^d$
and a bijection $I_N:$ $\{1,\dots, k_N\}$ $\to$ $G_N$, $k_N=|G_N|$, is fixed.
By interpolating between the nodes $\frac{1}{N}G_N$ $:=\{z\in \overline D$ $|$ $Nz\in G_N\}$ we can
identify a finite family of real random variables $X^{N,i}$,  or a family of real-valued diffusion processes
${(X_t^{N,i})}_{t\geq 0}$,
over the index set $i\in\{1,\dots k_N\}$ with an $H$-valued random variable,  or a diffusion on $H$.
Like this, a stochastic model for the phenomenons of a $(d+1)$-dimensional physical interface can be defined. The interface's height 
at a node $z$ such that $Nz=I_N(i)$ for some $i\in\{1,\dots, k_N\}$ is given by $X^{N,i}$, respectively 
${(X_t^{N,i})}_{t\geq 0}$.
For each $N\in\N$ the interface at this scale is thoroughly determined by its height values on the set of nodes $z\in\frac{1}{N}G_N$. 

In this frame, the main interests of this survey are:
\begin{itemize}
	
	\item Can we approximate a diffusion, informally given by \eqref{eqn:inispde}, with a sequence of process 
	$\{X_t^{N,i}\in\R$ $|$ $i=1,\dots,k_N$, $t\geq 0$ $\}$, which we identify with a sequence of diffusion processes on $H$ via the height maps ${(\Lambda_N})_N$, if we let $N\to\infty$?
	
	\item Does a change of the interpolation method affect the asymptotic behaviour for $N\to\infty$ of this stochastic interface model?
\end{itemize}
The relevant asymptotic statement, to which both questions refer, is the weak measure convergence of equilibrium fluctuations. 
In the article, for technical reasons, the second problem has to be covered first.
In the first part of this survey, we address the imminent question, to what extent the height map influences the scaling limit. We give an answer under the assumption that 
\begin{equation}
	\Lambda_Nx(z):=N^{\frac{d}{2}-1}\sum_{i=1}^{k_N}x_i\,\Xi(Nz- I_N(i)),\quad  x\in\R^{k_N},\,z\in D,
\end{equation}
where $\Xi:\R^d\to[0,1]$ is independent of $N$ and $I_N:\{1,\dots,k_N\}\to N\overline D\cap\Z^d$ is a bijection identifying the index of a microscopic component with a lattice point within $N\overline D$. As it turns out, 
the choice of the function $\Xi$ is irrelevant to the scaling limit, as long as it meets some basic and plausible assumptions. This result is obtained for a static random interface at first and then generalized to the dynamic case.
The second part of the article is dedicated to characterize the asymptotic distribution of $(u_t^N)_{t\geq 0}$ for $N\to\infty$. The main result proves the weak convergence of equilibrium laws.
It is based on the recent progress in \cite{me21} on the topic of Mosco-Kuwae-Shioya convergence of gradient forms with non-convex potentials.


The two most common interpolation methods to encounter in this context are either a piecewise constant interpolation 
from  $\R^{k_N}$ into $H$ or a piecewise linear interpolation from $\R^{k_N}$ into the continuous functions $C(\overline D)\subset H$. Usually in concrete examples (e.g.~in \cite{Zambotti04}, \cite{zambotti05}, \cite{Dembo2005}, \cite{bounebache14}, \cite{cipriani20}), one of those two methods is fixed. However, it would be desirable to be able to transfer a convergence result, which is proven for one method, to the other without doing extra calculations.
In this article, we give the abstract arguments why this is possible.
The interpolation methods taken into account are of the following general type.
Let $N\in\N$. Considering a static model at first, the starting point is a random vector with values in $\R^{k_N}$. So, we want to define the height map $\Lambda_N$ as a continuous linear operator from $\R^{k_N}$ into $H$. 
To this end, each index $i=1,\dots, k_N$ is spatially associated with the node $\frac{1}{N}I_N(i)\in\overline D$.
We want to assume spatial homogeneity, 
i.e.~the image set of the unit vectors $\{\univ_i$  $|$ $i=1,\dots k_N\}$ under $\Lambda_N$ is a set of shifted versions of one and the same archetype function $\Xi:\R^d\to[0,1]$, rescaled by a factor $c_N$.
This motivates the definition
\begin{equation}\label{eqn:heightDef0}
	\Lambda_Nx(z):=c_N\sum_{i=1}^{k_N}x_i\,\Xi(Nz- I_N(i)),\quad z\in D,\,x\in\R^{k_N}.
\end{equation}
The factor $c_N$ rescales the interface's amplitude, or height.
In this text, the choice is always $c_N:=N^{\frac{d}{2}-1}$.
This is merely a convention. However, the choice seems natural. 
As an example, we consider
the one-dimensional $\nabla \varphi$ interface model with $D=(0,1)$ and $G_N=\{1,\dots,N\}$.
If we set $X^{N,0}:=0$
and define a family of real random variables $\{X^{N,i}\}$ over the index $i\in G_N$ such that the law of 
\begin{equation*}
	{\big(X^{N,i}-X^{N,i-1}\big)}_{i=1,\dots, N+1}\in\R^{N+1}
\end{equation*}
is the $(N+1)$-dimensional standard normal distribution, then the scaling limit under $\Lambda_N$ is the standard Brownian motion on $[0,1]$. More precisely, let 
\begin{equation*}
Y^N:=X^{N,1}\univ_1+\dots+X^{N,N}\univ_N,
\end{equation*}
and 
\begin{align*}
	\Lambda_Nx(z)&:=\frac{1}{\sqrt{N}}\sum_{i=1}^{N}x_i\,\eins_{[-1,0)}(Nz- i)\\
	&=\frac{1}{\sqrt{N}}\sum_{i=1}^{N}x_i\,\eins_{[\frac{i-1}{N},\frac{i}{N})}(z),\quad z\in (0,1),\,x\in\R^{N}.
\end{align*}
	Then, the law of ${(\Lambda_N\circ Y^N)}_{N\in\N}$ converges towards the law of a Brownian motion on $[0,1]$, in the sense of weak measure convergence on $H=L^2((0,1),\de z)$. 
 The particular choice of $\Xi$ in \eqref{eqn:heightDef0} is irrelevant to the question of existence of a scaling limit in the sense of weak measure convergence on $H$.
This invariance principle is proven in Theorem \ref{thm:HeightMaps} for the general setting in which $D$ is a bounded open domain of $\R^d$ and $\Xi:\R^d\to[0,1]$ meets the properties specified in Condition \ref{condi:Xi}.
 A similar statement holds for the dynamic case and is stated in Theorem \ref{thm:Q1dyn}.

To motivate the first question posed above, we recall the stochastic partial differential equation, discussed in \cite{bounebache14}. 
Let $g:\R\to\R$ be a function of bounded variation and
$\mu^{\textnormal{BB}}$ denote the law of a Brownian bridge between $0$ and $0$ on the interval $[0,1]$.
In \cite{bounebache14}, the reader finds a definition of the weak solution for
\begin{align}\label{eqn:bouneHeat}
	&\frac{\partial u}{\partial t}=\frac{1}{2}\frac{\partial^2 u}{\partial z^2}-\frac{1}{2}\int_\R \frac{\partial }{\partial z}l_{t,z}^a \,g(\de a)+\dot W,\nonumber\\
	&u(t,0)=u(t,1)=0,\nonumber\\
	&u(0,z)=u_0(z),\quad z\in[0,1],
\end{align}
where ${(l_{t,z}^a)}_{t\geq 0,z\in[0,1]}$ is the family of local times at $a\in\R$ accumulated over $[0,z]$ by the process ${(u(t,r))}_{t\geq 0,r\in[0,1]}$ and $\dot W(t,z)$, $t\geq 0$, $z\in [0,1]$, denotes space-time white-noise.
Then, the authors construct a weak solution via Dirichlet form techniques. 
The invariant measure $m$ of \eqref{eqn:bouneHeat} is the probability on $H$ whose Radon-Nikodym derivative w.r.t.~$\mu^{\textnormal{BB}}$ is proportional to the density $\exp(-\int_{(0,1)}g\circ h\de z)$, $h\in H$.
The Dirichlet form of the Markov process from \cite{bounebache14} solving \eqref{eqn:bouneHeat} is 
given by
\begin{equation*}
	\E(u,v)=\frac{1}{2}\int_H\langle\grad u,\grad v\rangle_H\de m,\quad u,v\in\dom(\E),
\end{equation*}
on $L^2(H,m)$. Here, $\grad$ denotes the weak gradient in the Gaussian Sobolev space \sloppy $W^{1,2}(H,\mu^{\textnormal{BB}})$, which coincides with the domain $\dom(\E)$,
as the density $\de m/\de \mu^{\textnormal{BB}}$ is bounded from below and above by positive constants.
The stationary law of the diffusion process associated to $\E$ can be identified with 
the scaling limit of a system of interacting skew Brownian motions. As usual, we use the index $N$ to indicate the level of scaling.
For each $N\in\N$ we define a probability measure on $\R^N$ by
\begin{gather}\label{eqn:heree}
	\de m_{A_N}(x):=\frac{1}{Z_N}\exp\Big(-x^\textnormal{T}A_Nx-\frac{1}{N}\sum_{i=1}^N g\big(N^{-\frac{1}{2}}x_i\big)\Big)\de x,\\
	Z_N:=\int_{\R^N}\exp\Big(-x^\textnormal{T}A_Nx-\frac{1}{N}\sum_{i=1}^N g\big(N^{-\frac{1}{2}}x_i\big)\Big)\de x,\nonumber
\end{gather}
which has the symmetric, positive operator $A_N:\R^N\to\R^N$ as a parameter.
The Gaussian measure on $\R^N$ with density proportional to $\exp(-x^\textnormal{T}A_Nx)$ is denoted by $\mu_{A_N}$.
The stationary law of the Markov process on $\R^N$ associated to the Dirichlet form
\begin{equation}\label{eqn:aprE}
	\E^N(u,v)=\frac{1}{2}\sum_{i=1}^N\int_{\R^N}\frac{\partial u}{\partial x_i}\frac{\partial v}{\partial x_i}\de m_{A_N}(x),
	\quad u,v\in\dom(\E^N),
\end{equation}
on $L^2(\R^N,m_{A_N})$ is the natural candidate to yield a finite-dimensional approximation for the solution of \eqref{eqn:bouneHeat}.
The Domain $\dom(\E^N)$ of \eqref{eqn:aprE} coincides with the Gaussian Sobolev space $W^{1,2}(\R^N,\mu_{A_N})$.
We denote the components of that process by ${(X^{N,i}_t)}_{t\geq 0}$, $i=1,\dots, N$.
Given a decomposition
\begin{equation*}
	g(y)=g_0(y)+\sum_{j=1}^M\beta_j\eins_{(-\infty,y_j]}(y),\quad y\in\R,
\end{equation*}
for some $M\in\N$, $\beta_j,y_j\in \R$ for $j=1,\dots,M$ and $g_0\in C_b^1(\R)$, 
the processes ${(X^{N,\,\cdot\,}_t)}_{t\geq 0}$ satisfy  the system of stochastic differential equations

\begin{multline}\label{eqn:inifini}
	\de X^{N,i}_t=-\big(A_N X^{N,\,\cdot\,}_t\big)_i\de t-\frac{1}{2}N^{-\frac{3}{2}}\,g_0'\big(X^{N,i}_t/\sqrt{N}\big)\de t
	\\+\sum_{j=1}^M\frac{1-e^{-\beta_j/N}}{1+e^{-\beta_j/N}}\de l_t^{N,i,\tilde y_j}
	+\de B_t^{N,i},\quad t\geq 0,\,i=1,\dots, N,
\end{multline}
in the weak sense, where ${(B_t^{N,i})}_{t\geq 0}$, denote independent Brownian motions
and $l_t^{N,i,\tilde y_j}$ is the local time of ${(X^{N,i}_t)}_{t\geq 0}$ at $\tilde y_j:=\sqrt Ny_j$.

A result related to the mesoscopic scaling limit of \eqref{eqn:inifini} is provided in \cite{bounebache14}.
It refers to the subsequence $N=2^l$, $l\in\N$, and a particular choice for  $A_N$.
Let 
\begin{equation*}
	T_N:H\ni h\mapsto N\sum_{i=1}^N\Big\langle\eins_{\big[\tfrac{i-1}{N},\tfrac{i}{N}\big)},h\Big\rangle_H\eins_{\big[\tfrac{i-1}{N},
		\tfrac{i}{N}\big)}
\end{equation*}
be the orthogonal projection and 
\begin{equation*}
	\Lambda_N:\R^{N}\ni x\mapsto N^{-\frac{1}{2}}\sum_{i=1}^{N}x_i\eins_{\big[\tfrac{i-1}{N},
		\tfrac{i}{N}\big)},
\end{equation*}
the piecewise constant interpolating height map.
If $A_N:\R^{N}\to\R^{N}$ in \eqref{eqn:heree} is the positive, symmetric operator such that
\begin{equation}\label{eqn:proA}
	\mu^{\textnormal{BB}}\circ \big(\Lambda_N^{-1}\circ T_N\big)^{-1}\propto \exp(-x^\textnormal{T}A_Nx)\de x,
\end{equation}
then the laws of 
\begin{equation*}
u^N_t:=\Lambda_N X^{N,\,\cdot\,}_{N^2t}=
N^{-\frac{1}{2}}\sum_{i=1}^{N}X^{N,i}_{N^2t}\eins_{\big[\tfrac{i-1}{N},
	\tfrac{i}{N}\big)},\quad t\geq 0,
\end{equation*}
 converges  weakly for $N=2^l\overset{l}{\to}\infty$, as a sequence of measures  on $C([0,\infty),H^{-1,2}(0,1))$.
Its limit is a diffusion process with state space $C([0,1])$ and a stationary solution of \eqref{eqn:bouneHeat}.
This statement corresponds to \cite[Lem.~5.8]{bounebache14}.
The authors argue via Mosco convergence of Dirichlet forms (\cite[Thm.~5.6]{bounebache14})
in the framework of Kuwae-Shioya (see \cite{mosco94}, \cite{kuwae03}).

The portrayed example of a scaling limit is significantly extended in this article.
The spatial dimension of the mesoscopic interface model is now given by $d\in\N$, replacing 
$(0,1)$ with a general open, bounded domain $D\subset \R^d$. 
We derive the scaling limit for $N\to\infty$ of a $k_N$-sized family of interacting skew Brownian motions 
${(X^{N,i}_t)}_{t\geq 0}$, $i=1,\dots, N$.
Analogous as in the example above, we consider the stationary process ${(X^{N,\,\cdot\,}_t)}_{t\geq 0}$
associated to the gradient-type Dirichlet form on $L^2(\R^{k_N},m_N)$, where now 
\begin{align}\label{eqn:typeM}
	\de m_{A_N}(x):=\frac{1}{Z_N}\exp\Big(-x^\textnormal{T}A_Nx-\frac{1}{N^d}\sum_{i=1}^{k_N} g\big(N^{\frac{d}{2}-1}x_i\big)\Big)\de x,\nonumber\\
	Z_N:=\int_{\R^{k_N}}\exp\Big(-x^\textnormal{T}A_Nx-\frac{1}{N^d}\sum_{i=1}^{k_N} g\big(N^{\frac{d}{2}-1}x_i\big)\Big)\de x,
\end{align}
with a symmetric positive operator $A_N:\R^{k_N}\to\R^{k_N}$ and a function of bounded variation $g:\R\to\R$.
Each index $i\in\{1,\dots, k_N\}$ corresponds to a grid point $p_i\in  N\overline D\cap\Z^d$.
For a height map of the type \eqref{eqn:heightDef0}, the weak measure convergence of the law of the $H$-valued process
\begin{align*}
	u^N_t&:=\Lambda_N X^{N,\,\cdot\,}_{N^2t}\\&=
	N^{\frac{d}{2}-1}\sum_{i=1}^{k_N}X^{N,i}_{N^2t}\,\Xi(N\,\cdot\,- p_i),\quad t\geq 0,
\end{align*}
for $N\to\infty$ is shown. This holds indifferently from the choice of $\Xi$.
Moreover, in our model, the sequence ${(A_N)}_{N\in\N}$ is not pre-defined through the limiting Gaussian, as it is in \eqref{eqn:proA}. The only assumption is the existence of a Gaussian measure $\mu$ such that $\mu_{A_N}\circ\Lambda_N^{-1}\Rightarrow\mu$ in the sense of weak measure convergence on $H$, 
where as above $\mu_N$ denotes the Gaussian on $\R^N$ such that $\mu_N\propto\exp(-x^\textnormal{T}A_Nx)\de x$. 
A suitable state space for the convergence of the equilibrium laws of ${(u^N_t)}_{t\geq 0}$ for $N\to\infty$ is the dual $H_0'$ of a densely included separable Hilbert space $H_0$ in the Gelfand triple $H_0\subset H\subset H_0'$, assuming the inclusion is a Hilbert-Schmidt operator.
For the limiting process an infinite-dimensional version of the Fukushima decomposition is available due to \cite{Ro2zhu}.

	The outline of this article is as follows.
Section \ref{sec:hvar} deals with a static interface model over a bounded, open domain $D\subset \R^d$, $d\in\N$.
 We fix Borel probability measures, $m_N$ on $\R^{k_N}$ for $N\in\N$, and $m$ on $H=L^2(D,\de z)$. 
The measure $m_N$ represents the law of a $|k_N|$-sized family of real random variables.
The asymptotic element $m$ represents the law of a static random interface over $D$ in the $(d+1)$-dimensional space, which takes values in $H$.
Theorem \ref{thm:HeightMaps} provides a perturbation result for the height map \eqref{eqn:heightDef0}
regarding the asymptotic behaviour of the sequence ${(m_N)}_{N\in\N}$.
The validity of $m_N\circ\Lambda_N^{-1}\Rightarrow m$ (in the sense of weak measure convergence on $H$)
is not affected by a swap of $\Xi$ in \eqref{eqn:heightDef0},
as long as the function $\Xi$ meets certain criteria, specified in Condition \ref{condi:Xi}.
Lemma \ref{lem:ellipt}, Proposition \ref{prop:oder} and Theorem \ref{thm:HeightMaps} are taken from 
\cite{Wit23}. Their proofs have been slightly revised for better readability.
The examples in Section \ref{sec:hvar} 
treat the wetting model of \cite{zambotti05} and the polymer model in \cite{cipriani20}.
In each of those two cases the asymptotic equivalence 
of piecewise linear, continuously interpolating height maps and piecewise constant interpolating height maps is established.

Section \ref{sec:tightness} starts in the setting of the preceding section. A dynamic analogue of Theorem \ref{thm:HeightMaps} is derived, together with a tightness result in the context of Dirichlet forms. Then, the Fukushima decomposition for the diffusion processes which are relevant to the topic of microscopic interface models is stated. For the mesoscopic scaling limit of such processes an infinite-dimensional version of the Fukushima decomposition by Röckner-Zhu-Zhu is provided. At first, we assume that a symmetric, regular, local and conservative Dirichlet form $\E^N$ on $L^2(\R^{k_N},m_N)$ with Carré-du-Champs operator $\Gamma_N$ is given. For the tightness result, we require  each linear functional $l:\R^{k_N}\to\R$ to be a member of $\dom(\E)$ and
uniform bounds for $\Gamma_N(l,l)$ in terms of the Lipschitz constant of $l$ (see \eqref{eqn:Gamma}), independent from $N$. We consider a separable Hilbert space $H_0$, densely included in $H$, such that the inclusion $H_0\hookrightarrow H$ has a finite Hilbert Schmidt norm. Via the Gelfand triple
\begin{equation*}
	H_0\subset H=H'\subset H_0'
\end{equation*}
 we obtain a suitable state space $H_0'$ for the tightness result. We look at the image measure $\mathbb P_N$ of the equilibrium law of the diffusion process ${(X^N_t)}_{t\geq 0}$ associated to $\E^N$ under the time-scaled, $H$-valued process
 	\begin{equation*}
 	u^N_t:=\Lambda_N \circ X^{N}_{N^2t},\quad t\geq 0.
 \end{equation*}
 Proposition \ref{prp:tightness} states the tightness of ${(\mathbb P_N)}_{N\in\N}$ as a family of probabilities on $C([0,\infty),H_0')$, provided the tightness of the invariant measures.
  We consider another height map $\widetilde \Lambda_N$, which is equivalent to $\Lambda_N$ in the sense that it is also built on Condition \ref{condi:Xi}, 
 and define $\widetilde{\mathbb P_N}$ as the equilibrium law under 
  	\begin{equation*}
 	u^N_t:=\widetilde \Lambda_N \circ X^{N}_{N^2t},\quad t\geq 0.
 \end{equation*}
 Then,  the weak measure convergence of  ${(\widetilde{\mathbb P_N})}_{N\in\N}$ 
 and that of ${({\mathbb P_N})}_{N\in\N}$ on $C([0,\infty),H_0')$ are two equivalent statements.
This is shown in Corollary \ref{thm:Q1dyn}, a dynamic version of Theorem \ref{thm:HeightMaps}.
 The Fukushima decomposition for ${(X^N_t)}_{t\geq 0}$ is stated for the case, where $m_N=m_{A_N}$ is of the type \eqref{eqn:typeM} and $\E^N$ is the standard gradient form with that reference measure.
 The process is a system of skew interacting Brownian motions. The infinite-dimensional analogue is a Röckner-Zhu-Zhu decomposition for the diffusion associated to the  gradient from on $H$ with a perturbed Gaussian as a reference measure. The perturbing density is
proportional to $\exp(-\int_{D}g\circ h\de z)$, $h\in H$. It is the potential of the non-linear drift term in the reaction-diffusion equation \eqref{eqn:inispde}. Section \ref{sec:tightness} closes with an observation preparing the subsequent and final Section \ref{sec:Mokoko}.
The weak convergence of the equilibrium laws ${(\mathbb P_N)}_{N\in\N}$ is equivalent to the Mosco-Kuwae-Shioya convergence of the image forms of $N^2\E^N$ under the map $\Lambda_N$.
By virtue of the above mentioned Corollary \ref{thm:Q1dyn}, we may w.l.o.g.~assume $\Xi=\eins_{[-1,0)^d}$. In this case, the image form of $N^2\E^N$ under $\Lambda_N$ is again a standard gradient form. Its state space is $\im (\Lambda_N)\subset H$ and the reference measure is $m_N\circ\Lambda_N^{-1}$.
 
Section \ref{sec:Mokoko} deals with Mosco-Kuwae-Shioya convergence of gradient forms on $H$. To provide the desired result for the perturbed Gaussian models considered above, we apply the  more abstract criterion for Mosco convergence, developed in \cite{me21}. The main results are stated in Theorem \ref{thm:maiMO} and Corollary \ref{cor:mai}.
	
	The most important results of this article are:
\begin{itemize}
	\item Let $m_N$ on $\R^{k_N}$, $N\in\N$, and $m$ on $H=L^2(D,\de z)$ be Borel probability measures.
	We give a family $\mathcal A$ of functions $\R^d\to [0,1]$ from which to choose $\Xi$ 
	in \eqref{eqn:heightDef0}. Then, either the asymptotic statement of
	\begin{equation}\label{eqn:wmcHeightE}
		\lim_{N\to\infty}\int_{\R^{k_N}}f(\Lambda_Nx)m_N(\de x)=\int_H f \de m\quad\text{for all }f\in C_b(H)
	\end{equation}
	 holds true for every choice $\Xi\in\mathcal A$, or there is no  $\Xi\in\mathcal A$ at all to choose in \eqref{eqn:heightDef0} in order to obtain \eqref{eqn:wmcHeightE}. 
	 The family $\mathcal A$ is general enough to accommodate the most relevant choices for $\Xi$.
	 So, as a consequence of this result (see Theorem \ref{thm:HeightMaps}), it is irrelevant to the validity of \eqref{eqn:wmcHeightE}, whether one chooses piecewise linear, continuously interpolating height maps, or piecewise constant interpolating height maps (see Examples \ref{exa:arche}, \ref{exa:wettingM}, \ref{exa:sfPoly} and \ref{exa:ZambCip}).
	The family $\mathcal A$ comprises all functions specified in Condition \ref{condi:Xi}.
	
	\item Let $\mu$ be a non-degenerate centred Gaussian measure on $H$ and $\mu_N$ be a non-degenerate centred Gaussian measure on $\R^{k_N}$ such that $\mu_N\circ\Lambda_N^{-1}$ converges to $\mu$ for $N\to\infty$, in the sense of weak measure convergence on $H$. Further, given functions of bounded variation $g,g_N:\R\to\R$, $N\in\N$, we set
	$$\rho:H\ni h\mapsto\frac{1}{Z}\exp\Big(\int_D g\circ  h\de z\Big),\quad 
	Z:=\int_H\exp\Big(\int_D g\circ  h\de z\Big)\de\mu(h)$$
	and
	\begin{align*}
	\rho_N:\R^{k_N}\ni x\mapsto\frac{1}{Z_N}\exp\Big(\frac{1}{N^d}\sum_{i=1}^{k_N} g_N\big(N^{\frac{d}{2}-1}x_i\big)\Big),\\
	Z_N:=\int_{\R^{k_N}}\exp\Big(\frac{1}{N^d}\sum_{i=1}^{k_N} g_N\big(N^{\frac{d}{2}-1}x_i\big)\Big)\de \mu_N(x).
	\end{align*}
    We assume that ${(g_N)}_{N\in\N}$ is an approximation for $g$ in a generalized sense of local uniform convergence (as defined in Section \ref{sec:Mokoko}).
	Now, we consider the distorted Ornstein-Uhlenbeck process ${(X_t)}_{t\geq 0}$ on $H$ whose Dirichlet form is the standard gradient form with reference measure $\rho\mu$ (as defined  in Remark \ref{rem:remmi}) on the one hand. On the other hand, we look at the
	skew reflecting distorted Brownian motion ${(X_t^N)}_{t\geq 0}$ on $\R^{k_N}$ whose Dirichlet form is the Euclidean standard gradient form with reference measure $\rho_N\mu_N$ (as defined in \ref{exa:moveto}). 
	Then, the equilibrium laws of the latter process under the scaling
	\begin{equation*}
		\Lambda_N\circ X^N_{N^2t},\quad t\geq 0,
	\end{equation*}
	weakly converge towards the equilibrium law of ${(X_t)}_{t\geq 0}$ as $N\to\infty$ (see Corollary \ref{cor:mai}).
\end{itemize}
	
	\section{The static model and the height map $\mathbf\Lambda_N$}\label{sec:hvar}
	Let $d\in\N$ and $D$ be a bounded, open domain of $\R^d$ with topological closure $\overline D$. $\mathcal B_b(D)$ denotes the bounded, measurable functions from $D$ to $\R$. We assume that for $N\in\N$ we are given a set of  
grid points $G_N\subset N\overline D\cap\Z^d$ and an injective linear map $\Lambda_N$ from $\R^{|G_N|}$ to $\mathcal B_b(D)$.
Intuitively, $\Lambda_N$, which is called the height map,
associates a microscopic state $x\in\R^{|G_N|}$ with a physical interface in the $(d+1)$-dimensional space, the graph of $\Lambda_Nx$. 
It is defined over the domain $D$. In this text,  a microscopic state is usually denoted with the variable $x$, while the variable $z$ is usually used to represent a point in the domain of the interface, i.e.~a point in $D$. We denote the Lebesgue measures in the corresponding dimension by $\de x$, respectively $\de z$.
Moreover, we set $H:=L^2(D,\de z)$ with its inner product $\langle\cdot,\cdot\rangle$. $C_b(H)$ denotes the space of bounded, continuous functions from $H$ to $\R$.

 In this section, we fix Borel probability measures, $m_N$ on $\R^{|G_N|}$ for  $N\in\N$, and $m$ on $H$. Then we
 discuss an asymptotic stability property for the sequence ${(m_N)}_{N\in\N}$.
 We ask whether the weak measure convergence,
\begin{equation}\label{eqn:wmcHeight}
	\lim_{N\to\infty}\int_{\R^{|G_N|}}f(\Lambda_Nx)m_N(\de x)=\int_H f \de m\quad\text{for }f\in C_b(H),
\end{equation}
is preserved under a slight modification of the height map $\Lambda_N$, while ${(m_N)}_{N\in\N}$ and $m$ are fixed.
This question is tackled under the additional assumption that the height map $\Lambda_N$ is of a particular type.
We assume that a bijection $\{1,\dots,|G_N|\}$ $\to$ $G_N$, $i\mapsto p_i$, is given such that 
\begin{equation}\label{eqn:heightDef}
	\Lambda_N x(z)=N^{\frac{d}{2}-1}\sum_{i=1}^{|G_N|}x_i\,\Xi(Nz- p_i),\quad x\in \R^{|G_N|},\, z\in D,
\end{equation}
for some measurable function $\Xi:\R^d\to[0,1]$, which is independent of the parameter $N$.
The scaling factor of $N^{\frac{d}{2}-1}$ is merely a convention, as any re-scaling of the height maps \eqref{eqn:heightDef} could simply be absorbed by considering the accordingly re-defined measures ${(m_N)}_{N\in\N}$ in \eqref{eqn:wmcHeight}. 

As an answer to the stability question raised above, we can explicitly give a set $\mathcal A$ of functions $\R^d\to [0,1]$ such that, either \eqref{eqn:wmcHeight} holds for every choice $\Xi\in\mathcal A$ in \eqref{eqn:heightDef}, or there is no  $\Xi\in\mathcal A$ at all to choose in \eqref{eqn:heightDef} in order to obtain \eqref{eqn:wmcHeight}. As an consequence of this section's main result, Theorem \ref{thm:HeightMaps}, it is irrelevant to the validity of \eqref{eqn:wmcHeight} whether one chooses linearly interpolating height maps, or piecewise constant interpolating height maps (see Examples \ref{exa:arche}, \ref{exa:wettingM}, \ref{exa:sfPoly} and \ref{exa:ZambCip}).
$\mathcal A$ is the set of all functions specified by the following condition.

\begin{condition}\label{condi:Xi}
	\begin{thm(vi)}
		$\Xi:\R^d\to[0,1]$ is measurable with the following properties:
		\item $\Xi$ vanishes outside the centred $d$-cube with side-length $2$, i.e.~$\supp\Xi\subseteq[-1,1]^d$.
		\item The integral $\int_{\R^d}\Xi(z)\de z$ equals $1$.
		\item The family $\{\Xi(\,\cdot\,-\,\mathbf k)$ $|$ $\mathbf k\in\Z^d\}$, containing all shifts of $\Xi$ by lattice points, form a partition of unity, i.e.~
		\begin{equation*}
			\sum_{\mathbf k\in\Z^d}\Xi(z-\mathbf k)=1,\quad z\in\R^d.
		\end{equation*}
		\item The strict inequality $\displaystyle\int_{\R^d}\Xi(z)^2\de z>\frac{1}{2}$ holds true.
	\end{thm(vi)}
\end{condition}

One obvious choice for $\Xi$ to meet the criteria of Condition \ref{condi:Xi} is the indicator function $\eins_{[-\frac{1}{2},\frac{1}{2})^d}$ of a semi-open cube with side-length $1$.	Alternatively, the tent function can provide a legitimate choice for $\Xi$. We discuss it for the cases $d=1$ and $d=2$ in the proceeding examples.
For an arbitrary index set $I$ and functions $f_i:\R^d\to\R$, $i\in I$, let
\begin{gather*}
	(f_i\land f_j)(z):=\min(\{f_i(z),f_j(z)\}),\\
	\Big(\bigwedge_{i\in I}f_i\Big)(z):=\inf(\{f_i(z)\,|\,i\in I\})
\end{gather*}
for $z\in\R^d$.
The positive part $z\mapsto \max(\{0,f(z)\})$ of a function $f:\R^d\to\R$ is denoted by $f^+$.
Moreover, $\pi_i:\R^d\to\R$ denotes the projection onto the $i$-th coordinate for $i\in\{1,\dots,d\}$.

\begin{exa}\label{exa:arche}\textsc{The tent function}
	\begin{thm(ii)}
		\item Let $d=1$ and \lq$\textnormal{id}$\rq\, be the identity on $\R$. The continuous, piecewise linear function
		\begin{equation}\label{eqn:1dtent}
			\Xi:=\big((1+\textnormal{id})\land(1-\textnormal{id})\big)^+
		\end{equation} 
		meets the criteria of Condition \ref{condi:Xi}, as easily verified.
		To check the fourth criterion, for example,  we calculate
		\begin{equation*}
			\int_{\R}\Xi(z)^2\de z=\int_{-1}^0(1+z)^2\de z+\int_0 ^1 (1-z)^2\de z=\frac{1}{3}+\frac{1}{3}=\frac{2}{3}>\frac{1}{2}.
		\end{equation*}
		\item  Now we consider the case $d=2$. The continuous function
		\begin{equation}\label{eqn:2dtent}
			\Xi:=\Big[\bigwedge_{i,j\in\{1,2\}}
			\big((1+\pi_i-\pi_j)\land(1+\pi_i)\land (1-\pi_i)\big)\Big]^+
		\end{equation}
		meets the criteria of Condition \ref{condi:Xi}. $\Xi$ is the function which results from taking the minimum over a family of six affine linear functions,
		\begin{align*}
			\Big\{\quad &1+\pi_1-\pi_2,\quad 1+\pi_1,\quad 1-\pi_1,\\
			&1+\pi_2-\pi_1,\quad 1+\pi_2,\quad 1-\pi_2\quad\Big\},
		\end{align*}
		and then take its positive part. Simply comparing the values of those six affine linear functions, we get
		\begin{equation}\label{eqn:tentCase}
			\Xi(z)=
			\begin{cases}
				1-z_1&\text{if }0\leq z_2<z_1<1,\\
				1-z_2&\text{if }0\leq z_1\leq z_2<1,\\
				1+z_1-z_2&\text{if }-1\leq z_2-1<z_1<0,\\
				1+z_1&\text{if }-1\leq z_1\leq z_2<0,\\
				1+z_2&\text{if }-1\leq z_2<z_1<0,\\
				1+z_2-z_1&\text{if }-1\leq z_1-1\leq z_2<0,\\
				0&\text{else.}
			\end{cases}
		\end{equation}	
		The fourth criterion of Condition \ref{condi:Xi} can be checked with elementary calculations. It holds
		\begin{align*}
			\int_{\{0\leq\pi_2<\pi_1<1\}}(1-\pi_1)^2(z)\de z=\int_0^1\int_0^{z_1}(1-z_1)^2\de z_2\de z_1=\int_0^1z_1(1-z_1)^2\de z_1=\frac{1}{12}
		\end{align*}
		and
		\begin{align*}
			\int_{\{-1\leq\pi_1-1\leq\pi_2<0\}}(1+\pi_2-\pi_1)^2(z)\de z&=\int_0^1\int_{z_1-1}^0(1+z_2-z_1)^2\de z_2\de z_1\\
			&=\int_0^1-\frac{3}{2}z_1^3+\frac{3}{2}z_1^2-\frac{3}{2}z_1+\frac{5}{6}\de z_1=\frac{5}{24}.
		\end{align*}
		For reasons of symmetry it holds
		\begin{multline*}
			\frac{1}{12}=\int_{\{0\leq\pi_2<\pi_1<1\}}(1-\pi_1)^2(z)\de z=\int_{\{0\leq\pi_1\leq\pi_2<1\}}(1-\pi_2)^2(z)\de z\\
			=\int_{\{-1\leq\pi_2<\pi_1<0\}}(1+\pi_2)^2(z)\de z=\int_{\{-1\leq\pi_1\leq\pi_2<0\}}(1+\pi_1)^2(z)\de z
		\end{multline*}
		and also
		\begin{equation*}
			\frac{5}{24}=\int_{\{-1\leq\pi_1-1\leq\pi_2<0\}}(1+\pi_2-\pi_1)^2(z)\de z=
			\int_{\{-1\leq\pi_2-1<\pi_1<0\}}(1+\pi_1-\pi_2)^2(z)\de z.
		\end{equation*}
	    Hence, from \eqref{eqn:tentCase} it follows
		\begin{equation*}
			\int_{\R^d}\Xi(z)^2\de z=4\cdot\frac{1}{12}+2\cdot\frac{5}{24}=\frac{3}{4}>\frac{1}{2}
		\end{equation*}
		as desired. A more extensive analysis of the tent function can be found in \cite{me21} and \cite{Wit23}. In both of these references, a general dimension $d\in\N$ is considered, complementing the definitions of \eqref{eqn:2dtent} and \eqref{eqn:1dtent} with 
		\begin{equation*}
			\Xi:=\Big[\bigwedge_{i,j\in\{1,\dots,d\}}\big((1+\pi_i-\pi_j)\land(1+\pi_i)\land (1-\pi_i)\big)\Big]^+
		\end{equation*}
		for a case $d>2$.
		For the $d$-dimensional tent function, where $d\in\N$, items (i) to (iii) of Condition \ref{condi:Xi} are proven in \cite[Sect.~2.1]{me21}, respectively \cite[Sect.~3.1.1]{Wit23}.
	\end{thm(ii)}
\end{exa}

Before turning back to the general case, another remark about the interpolating property of the tent function is inserted, which is exploited later on, in Examples \ref{exa:wettingM} \& \ref{exa:sfPoly}. For the abstract results of this section, Remark \ref{rem:Xi} does not play any role. The reader who is just interested in the main results of 
Section \ref{sec:hvar}, Proposition \ref{prop:oder} and Theorem \ref{thm:HeightMaps}, may skip it.

\begin{remark}\label{rem:Xi}
	    \begin{thm(ii)}
	    	\item We consider the one-dimensional tent function of Example \ref{exa:arche} (i). 
	    	For any given sequence of values
	    	${(\lambda_{ k})}_{ k\in\Z}\in \R^{\Z}$, the function
	    	\begin{equation}\label{eqn:interpol1d}
	    		\R\ni z\mapsto \sum_{k\in\Z}\lambda_{ k}\,\Xi(z- k)
	    	\end{equation}
	    	is the unique element of the space of piecewise linear, continuous functions
	    	\begin{align*}
	    		\Big\{ f\in C(\R)\,\Big|\,\text{for each }k\in\Z
	    		\text{ there is an affine linear function}&\\a_{k}:\R\to\R
	    		\text{ with }a_{k}(z)=f(z)\text{ for }z\in [k,k+1)&\Big\}
	    	\end{align*}
	    	which interpolates the sample $\{( k,\lambda_ k)$ $|$ $ k\in\Z\}$.
	    	So, \eqref{eqn:interpol1d} has the alternative representation 
	    	\begin{equation}\label{eqn:display1d}
	    		\sum_{k\in\Z}\lambda_{ k}\,\Xi(z- k)=\sum_{k\in\Z}\eins_{[k,k+1)}(z)\Big(\lambda_k+(z-k)(\lambda_{k+1}-\lambda_k)\Big).
	    	\end{equation}

	    \item For the two-dimensional case, an analogous representation as in \eqref{eqn:display1d} can be derived, which displays the connection between the tent function and piecewise linear interpolation.
	    The semi-open unit square $[0,1)^2$ is the disjoint union of two triangles $T_{1, 0}$ and $T_{2, 0}$, where
	    \begin{equation*}
	    	T_{1,0}:=\{z\in\R^2\,|\,0\leq z_2<z_1<1\}\quad\text{and}\quad
	    	T_{2,0}:=\{z\in \R^2\,|\, 0\leq z_1\leq z_2<1\}.
	    \end{equation*}
	    Shifting those triangles results in a partition of $\R^2$ through the family of disjoint sets
	    \begin{equation*}
	    	T_{i,\mathbf k}:=\{\mathbf k+z\,|\,z\in T_{i,0}\},\quad i=1,2,\quad \mathbf k\in\Z^2.
	    \end{equation*}
	    The representation of the two-dimensional tent function from \eqref{eqn:tentCase} shows that
	    the function $\Xi$ considered in Example \ref{exa:arche} (ii) is a member of the linear space
	    \begin{align*}
	    	Q:=\Big\{ f\in C(\R^2)\,\Big|\,\text{for each }i\in\{1,2\}\text{ and }\mathbf k\in\Z^2
	    	\text{ there is an affine linear function}&\\a_{i,\mathbf k}:\R^2\to\R
	    	\text{ with }a_{i,\mathbf k}(z)=f(z)\text{ for }z\in T_{i,\mathbf k}&\Big\}.
	    \end{align*}
	    The same applies to all shifted functions $\Xi(\,\cdot\, -\mathbf k)$, $\mathbf k\in\Z^2$, of course.
	    Let ${(\lambda_{\mathbf k})}_{\mathbf k\in\Z^2}\in \R^{\Z^2}$.
		Since $\Xi(\mathbf l -\mathbf k)=\delta_{\mathbf l,\mathbf k}$ for $\mathbf l$, $\mathbf k$ $\in\Z^2$, 
		the function
		\begin{equation}\label{eqn:interpo2d}
			\R^2\ni z\mapsto \sum_{\mathbf k\in\Z^2}\lambda_{\mathbf k}\,\Xi(z-\mathbf k)
		\end{equation}
		is an element in $Q$ which interpolates the sample $\{(\mathbf k,\lambda_\mathbf k)$ $|$ $\mathbf k\in\Z^2\}$.
		It is the unique element of $Q$ with that property, as every affine linear function is uniquely determined by its values on the corner points of a triangle in the two-dimensional plane. 
		Denoting the unit vectors of $\R^2$ by $\univ_1$, $\univ_2$, the corner points of 
		 $T_{1,\mathbf k}$ are $\mathbf k$, $\mathbf k+\univ_1$,  $\mathbf k+\univ_1+\univ_2$,
		while the corner points of $T_{2,\mathbf k}$ are $\mathbf k$, $\mathbf k+\univ_2$,  $\mathbf k+\univ_1+\univ_2$, for every $\mathbf k\in\Z^2$.  This argumentation shows that the map of \eqref{eqn:interpo2d} has the following representation, which reflects its interpolating property:
		\begin{align}\label{eqn:display2d}
			&\sum_{\mathbf k\in\Z^2}\lambda_{\mathbf k}\,\Xi(z-\mathbf k)\nonumber\\
			&=\sum_{\mathbf k\in\Z^2}\eins_{T_{1,\mathbf k}}(z)\Big(\lambda_\mathbf k+
			\pi_1(z-\mathbf k)(\lambda_{\mathbf k+\univ_1}-\lambda_\mathbf k)
			+\pi_2(z-\mathbf k)(\lambda_{\mathbf k+\univ_1+\univ_2}-\lambda_{\mathbf k+\univ_1})\Big)\nonumber\\
			&\qquad+\eins_{T_{2,\mathbf k}}(z)\Big(\lambda_\mathbf k+
			\pi_2(z-\mathbf k)(\lambda_{\mathbf k+\univ_2}-\lambda_\mathbf k)
			+\pi_1(z-\mathbf k)(\lambda_{\mathbf k+\univ_1+\univ_2}-\lambda_{\mathbf k+\univ_2})\Big).
		\end{align}
		This implies the interpolating property  of the map \eqref{eqn:heightDef} if $d=2$ and $\Xi$ is chosen as the tent function \eqref{eqn:2dtent}.
	\end{thm(ii)}
\end{remark}

On the Hilbert space $l^2(\Z^d)$ of all square-summable, multi-indexed sequences ${(\lambda_{\mathbf k})}_{\mathbf k\in\Z^d}$ $\in\R^{\Z^d}$ the infinite (multi-indexed) matrix
\begin{equation*}
	A_{\mathbf k,\mathbf l}:=\int_{\R^d}\Xi(z-\mathbf k)\,\Xi(z-\mathbf l)\de z, \quad \mathbf k,\mathbf l\in\Z^d,
\end{equation*}
corresponds to a symmetric linear operator $A$. The lower estimate in Proposition \ref{prop:oder} below is based on the existence of an upper bound for the operator norm of $A^{-1}$. 
Item (iv) of Condition \ref{condi:Xi} is a sufficient criterion for $A^{-1}$ to be a bounded linear operator on $l^2(\Z^d)$, as the next lemma shows.

\begin{lemma}\label{lem:ellipt}
	Let ${(\lambda_{\mathbf k})}_{\mathbf k\in\Z^d}\in \R^{\Z^d}$ such that $\{\mathbf k\in\Z^d$ $|$ $\lambda_\mathbf k\neq 0\}$ is a finite set and $\Xi$ as in Condition \ref{condi:Xi}.
	It holds
	\begin{equation*}
		\sum_{\mathbf k\in\Z^d}\sum_{\mathbf l\in\Z^d}\lambda_{\mathbf k}A_{\mathbf k,\mathbf l}\lambda_{\mathbf l}\geq c\sum_{\mathbf k\in \Z^d}\lambda_{\mathbf k}^2,\qquad\text{where}\quad c:=2\int_{\R^d}\Xi(z)^2\de z-1.
	\end{equation*}
\end{lemma}
\begin{proof}
	Since $A_{\mathbf k,\mathbf k}=\frac{c+1}{2}$ for $\mathbf k\in\Z^d$, 
	\begin{align}\label{eqn:basiC}
		\sum_{\mathbf k\in\Z^d}\sum_{\mathbf l\in\Z^d}\lambda_{\mathbf k}A_{\mathbf k,\mathbf l}\lambda_{\mathbf l}&=\sum_{\mathbf k\in \Z^d}A_{\mathbf k,\mathbf k}\lambda_{\mathbf k}^2+\sum_{\substack{\mathbf k,\mathbf l\in\Z^d\\\mathbf l\neq \mathbf k}}A_{\mathbf k,\mathbf l}\lambda_{\mathbf k}\lambda_{\mathbf l}\nonumber\\
		&=\frac{1}{2}\sum_{\mathbf k\in \Z^d}\lambda_{\mathbf k}^2+\sum_{\substack{\mathbf k,\mathbf l\in\Z^d\\\mathbf l\neq \mathbf k}}A_{\mathbf k,\mathbf l}\lambda_{\mathbf k}\lambda_{\mathbf l}+\frac{c}{2}\sum_{\mathbf k\in \Z^d}\lambda_{\mathbf k}^2.
	\end{align}
	Moreover, since $\int_{\R^d}\Xi(z)\de z=1$ and $\sum_{\mathbf l\in\Z^d}\Xi(z-\mathbf l)=1$ for $z\in\R^d$,  
	\begin{align}\label{eqn:thisEq}
		1&=2\int_{\R^d}\Xi(z-\mathbf k)\de z-1=2\sum_{\mathbf l\in\Z^d}A_{\mathbf k,\mathbf l}-1\nonumber\\
		&=2A_{\mathbf k,\mathbf k}-1+2\sum_{\substack{\mathbf l\in\Z^d\\\mathbf l\neq \mathbf k}}A_{\mathbf k,\mathbf l}
		=c+2\sum_{\substack{\mathbf l\in\Z^d\\\mathbf l\neq \mathbf k}}A_{\mathbf k,\mathbf l},\quad\mathbf k\in\Z^d.
	\end{align}
	Multiplying \eqref{eqn:thisEq} by $\lambda_\mathbf k^2/2$ for each $\mathbf k\in\Z^2$ and then summing up over $\mathbf k\in\Z^d$ leads to
	\begin{equation}\label{eqn:plugIn}
		\frac{1}{2}\sum_{\mathbf k\in \Z^d}\lambda_{\mathbf k}^2=\frac{c}{2}\sum_{\mathbf k\in \Z^d}\lambda_{\mathbf k}^2+\sum_{\substack{\mathbf k,\mathbf l\in\Z^d\\\mathbf l\neq \mathbf k}}A_{\mathbf k,\mathbf l}\lambda_{\mathbf k}^2.
	\end{equation}
    Rewriting the right-hand side of \eqref{eqn:basiC} under use of \eqref{eqn:plugIn}, one obtains
	\begin{align*}
		\sum_{\mathbf k\in\Z^d}\sum_{\mathbf l\in\Z^d}\lambda_{\mathbf k}A_{\mathbf k,\mathbf l}\lambda_{\mathbf l}
		&=\sum_{\substack{\mathbf k,\mathbf l\in\Z^d\\\mathbf l\neq \mathbf k}}A_{\mathbf k,\mathbf l}\big(\lambda_\mathbf k^2+\lambda_{\mathbf k}\lambda_{\mathbf l}\big)+c\sum_{\mathbf k\in \Z^d}\lambda_{\mathbf k}^2\\
		&=\frac{1}{2}\sum_{\mathbf k,\mathbf l\in \Z^d}A_{\mathbf k,\mathbf l}\big(\lambda_\mathbf k+\lambda_\mathbf l\big)^2+ c\sum_{\mathbf k\in \Z^d}\lambda_{\mathbf k}^2\geq c\sum_{\mathbf k\in \Z^d}\lambda_{\mathbf k}^2
	\end{align*}
	as desired.
\end{proof}

We make a further assumption on the height maps $\Lambda_N$, $N\in\N$, concerning the sets of Grid points $G_N$. We assume there is a sequence of increasing sets $D_1\subseteq D_2\subseteq\dots$ contained in $\overline D$ such that 
\begin{equation*}
	\de z\Big(\overline D\setminus\bigcup_{N\in\N}D_N\Big)=0
\end{equation*}
and for $N\in\N$ we have
\begin{align}\label{eqn:GEen}
	\text{ $D_N$ is open w.r.t.~the trace topology of $\overline D$,}\nonumber\\
	G_N=ND_N\cap\Z^d.
\end{align}
To shorten notation, let $k_N:=|G_N|$ and 
\begin{equation*}
\xi_N^p(z):=\Xi(Nz-p),\quad z\in\R^d,\,p\in G_N.
\end{equation*}
Then, according to \eqref{eqn:heightDef}
there is a bijection $\{1,\dots,k_N\}$ $\to$ $G_N$, $i\mapsto p_i$, such that, setting $\xi_{N,i}:=\xi_N^{p_i}$ for $i\in{1,\dots,k_N}$,
\begin{equation}\label{eqn:heightDef2}
	\Lambda_N x(z)=N^{\frac{d}{2}-1}\sum_{i=1}^{k_N}x_i\,\xi_{N,i}(z),\quad x\in \R^{k_N},\, z\in D.
\end{equation}

\begin{exa}\label{exa:wettingM}\textsc{Pining Model}
	
		In the one-dimensional case with $D=(0,1)$, we can select $D_N=(0,1]$ for $N\in\N$, and 
	   $\Xi$ as in Example \ref{exa:arche} (i). 
	   So,
		$G_N=\{1,\dots, N\}$ and the natural ordering leads to the definition $p_i:=i$, $i=1,\dots,N$.
		Choosing this particular set-up in \eqref{eqn:heightDef2} yields
		\begin{equation*}
			\Lambda_N x(z)=\frac{1}{\sqrt N}\sum_{i=1}^N x_i\max\Big(\Big\{0,\min\big(\big\{1-i+Nz,1+i-Nz\big\}\big)\Big\}\Big),
		\end{equation*}
		$z\in(0,1)$, $x\in\R^N$.
		As a consequence of Remark \ref{rem:Xi} (i), $\Lambda_Nx(\cdot)$
		 restricts to an affine linear function on each interval 
		$(\frac{i-1}{N},\frac{i}{N})$, $i=1,\dots, N$, and interpolates the values $\frac{x_1}{\sqrt N},\dots,\frac{x_N}{\sqrt N}$ over the grid points $\frac{1}{N},$ $\frac{2}{N}$, $\dots, 1$. Setting $x_0:=0$, a different way to represent this particular height map is 
		\begin{equation*}
				\Lambda_N x(z)=	\frac{1}{\sqrt N}\sum_{i=0}^{N-1} \Big(x_{i}+(Nz-i)(x_{i+1}-x_{i})\Big)
						\eins_{[\frac{i}{N},\frac{i+1}{N})}(z),\quad z\in(0,1),\,x\in\R^N,
		\end{equation*} 
		as follows from \eqref{eqn:display1d} after re-scaling. We cite a result from \cite{zambotti05}.
		Let $V:\R\to\R\cup\{\infty\}$ such that $\exp(-V(\cdot))$ is continuous, $V(0)<\infty$ and
		\begin{equation*}
			\int_{\R}e^{-V(s)}\de s=:\kappa\in\R
             \quad\text{as well as}\quad
			\frac{1}{\kappa}\int_\R s^2 e^{-V(s)}\de s:=\sigma^2\in \R.
		\end{equation*}
		For $x\in\R^N$, $N\in\N$, setting $x_0:=0$, $x_{N+1}:=0$, let 
		\begin{equation*}
			H_N(x):=\sum_{i=0}^NV(\sigma x_{i+1}-\sigma x_i).
		\end{equation*}
		Now, for $\beta\in[0,\infty)$ we define a probability measure on $\R^N$ (with support on $[0,\infty)^N$) as
	\begin{align*}
		m_N^\beta(\de x):=\frac{1}{Z_{\beta,N}}\exp(-H_N(x))
		\prod_{i=1}^N\big(\eins_{[0,\infty)}(x_i)\de x_i+\beta\delta_0(\de x_i)\big)\\
		\text{with}\quad Z_{\beta,N}:=\int_{[0,\infty)^N}\exp({-H_N(x)})\prod_{i=1}^N\big(\eins_{[0,\infty)}(x_i)\de x_i+\beta\delta_0(\de x_i)\big)
	\end{align*}
	and let
	\begin{equation*}
		\beta_c:=\frac{\kappa}{1+\sum_{N=1}^\infty Z_{0,N}}.
	\end{equation*}
	Due to \cite[Thm.~1]{zambotti05} the sequence ${(m_N^\beta\circ\Lambda_N)}_{N\in\N}$ has a limit $m^\beta$ regarding weak measure convergence on $C([0,1])$ (equipped with the supremum-norm) for every $\beta\in[0,\infty)$. Hence, in particular, this is true regarding weak measure convergence on $H=L^2((0,1),\de z)$. The model undergoes a phase transition as the limit $m^\beta$ depends on the pinning strength $\beta$ as follows:
	\begin{itemize}
		\item If $0\leq\beta<\beta_c$, then $m^\beta$ is the law of the three-dimensional Bessel bridge between $0$ and $0$  on the time interval $[0,1]$.
		\item In the critical case $m^{\beta_c}$ is the law of the one-dimensional Bessel bridge between $0$ and $0$ on the time interval $[0,1]$.
		\item If $\beta_c<\beta$, then $m^\beta$ assigns the value $1$ to the singleton of the constant zero-function.  on $[0,1]$.
	\end{itemize}
\end{exa}

\begin{exa}\label{exa:sfPoly}\textsc{Semiflexible Polymers and $\Delta\varphi$-model}
	
	In the two-dimensional case with $D=(0,1)^2$, we can 
	choose 
	\begin{equation*}
	D_N=\big(\frac{2}{N},\frac{N-2}{N}\big)\times \big(\frac{2}{N},\frac{N-2}{N}\big),
	\end{equation*}
	w.l.o.g.~assuming $N\geq 6$.
	As usual, let $G_N:=ND_N\cap\Z^2$. 
	We cite a result on weak measure convergence from \cite{cipriani20}.
	The method, used there, interpolates a state
	${(\lambda_{\mathbf k})}_{\mathbf k\in G_N}\in\R^{G_N}$ by identifying it with the  interface $S({(\lambda_{\mathbf k})}_{\mathbf k\in G_N},\,\cdot\,):D\to \R$,
	\begin{multline}\label{eqn:ciprInt}
		S\big({(\lambda_{\mathbf k})}_{\mathbf k\in G_N},z\big)\\
		:=\begin{cases}
		\lambda_{[Nz]}+
		\pi_1(Nz-[Nz])(\lambda_{\mathbf [Nz]+\univ_1}-\lambda_{[Nz]})
		+\pi_2(Nz-[Nz])(\lambda_{[Nz]+\univ_1+\univ_2}-\lambda_{[Nz]+\univ_1})
		&\\
		\lambda_{[Nz]}+
		\pi_2(Nz-[Nz])(\lambda_{[Nz]+\univ_2}-\lambda_{[Nz]})
		+\pi_1(Nz-[Nz])(\lambda_{[Nz]+\univ_1+\univ_2}-\lambda_{[Nz]+\univ_2})
		\end{cases}\\
		\begin{aligned}
		\text{if}\quad \pi_2( Nz -[Nz])&<\pi_1( Nz -[Nz])\text{ in the first case,}\\
		\pi_1( Nz -[Nz])&\leq\pi_2( Nz -[Nz])\text{ in the second case}
		\end{aligned}
	\end{multline}
	(cf.~\cite[Eq.~(2.6)]{cipriani20}), where $[Nz]$ denotes the element in $\Z^2$ with $Nz-[Nz]\in[0,1)^2$ for $z\in\Z^2$.
	This is just a differently scaled version of \eqref{eqn:display2d}, as we show now.
	For $T_{1,\mathbf k}$, $T_{2,\mathbf k}$
	as in Remark \ref{rem:Xi} (ii), $\mathbf k\in\Z^2$ and $z\in \R^2$ it holds
	\begin{equation*}
		\eins_{T_{1,\mathbf k}}(Nz)=
		\begin{cases}
			1&\text{if}\quad \mathbf k=[Nz]\quad\land \quad \pi_2( Nz -\mathbf k)<\pi_1( Nz -\mathbf k)\\
			0&\text{else,}
		\end{cases}
	\end{equation*}
	while 
	\begin{equation*}
		\eins_{T_{2,\mathbf k}}(Nz)=
		\begin{cases}
			1&\text{if}\quad\mathbf k=[Nz]\quad\land \quad \pi_1( Nz -\mathbf k)\leq\pi_2( Nz -\mathbf k)\\
			0&\text{else.}
		\end{cases}
	\end{equation*}
	By Remark \ref{rem:Xi} (ii), with $\Xi$ as in Example \ref{exa:arche} (ii), it holds
	\begin{align*}
		&S\big({(\lambda_{\mathbf k})}_{\mathbf k\in G_N},z\big)\\
		&=\sum_{\mathbf k\in L_N}\eins_{T_{1,\mathbf k}}(Nz)\Big(\lambda_\mathbf k+\pi_1(Nz-\mathbf k)(\lambda_{\mathbf k+\univ_1}-\lambda_\mathbf k)
		+\pi_2(Nz-\mathbf k)(\lambda_{\mathbf k+\univ_1+\univ_2}-\lambda_{\mathbf k+\univ_1})\Big)\\
		&+\sum_{\mathbf k\in G_N}\eins_{T_{2,\mathbf k}}(Nz)
		\Big(\lambda_\mathbf k+
		\pi_2(Nz-\mathbf k)(\lambda_{\mathbf k+\univ_2}-\lambda_\mathbf k)
		+\pi_1(Nz-\mathbf k)(\lambda_{\mathbf k+\univ_1+\univ_2}-\lambda_{\mathbf k+\univ_2})\Big)\\
		&=\sum_{\mathbf k\in G_N}\lambda_{\mathbf k}\,\Xi(Nz-\mathbf k).
	\end{align*}
	Consequently, if we fix an arbitrary bijection $I_N:\{1,\dots,k_N\}\to G_N$, where $k_N:=|G_N|$,
	and initialize the height map from \eqref{eqn:heightDef2} by choosing $\Xi$ from Example \ref{exa:arche} (ii),
	then $\Lambda_N$ and the interpolation \eqref{eqn:ciprInt} from \cite{cipriani20} are related as follows:
	\begin{gather*}
	\Lambda_Nx(z)=S\big({(\widetilde x_{\mathbf k})}_{\mathbf k\in L_N},z\big),\quad z\in D,\\
	\text{where}\quad\widetilde x_\mathbf k:=x_{I_N^{-1}(\mathbf k)},\quad \mathbf k\in G_N,\,x\in\R^{k_N}.
	\end{gather*}
	For $x\in\R^{k_N}$ we extend ${(\widetilde x_\mathbf k)}_\mathbf k$
	to an element of $\R^{Z^2}$ by setting $\widetilde x_\mathbf k=0$ for $\mathbf k\in Z^2\setminus G_N$.
	For a scaling $\alpha:\N\to (0,\infty)$ we define
	\begin{equation*}
		H_N^\alpha(x):=\sum_{\mathbf k\in Z^2}\Big[\frac{1}{16}
		\sum_{\substack {\mathbf l\in\Z^2:\\|\mathbf l-\mathbf k|_\textnormal{euc}=1}} (\widetilde x_\mathbf k-\widetilde x_\mathbf l)^2+\alpha(N)\Big(\frac{1}{2d}\sum_{\substack {\mathbf l\in\Z^2:\\|\mathbf l-\mathbf k|_\textnormal{euc}=1}} (\widetilde x_\mathbf k-\widetilde  x_\mathbf l)\Big)^2\,\Big]
	\end{equation*}
	for $N\in\N$ and $x\in\R^{k_N}$.
	Let $m^\alpha_N$ be the Gaussian measure on $\R^{k_N}$ with density 
	\begin{equation*}
	\R^{k_N}\ni x\mapsto \frac{1}{Z_N} \exp({- 4H^\alpha_N(x)}),\quad\text{where } Z_N:=\int_{\R^{k_N}}\exp({-4H^\alpha_N(x)})\de x.
	\end{equation*}
	By virtue of \cite[Thm.~2.1 (1) \& (2)]{cipriani20}, the sequence ${(m^\alpha_N\circ\Lambda_N)}_{N\in\N}$ has a limit $m^\alpha$ regarding weak measure convergence on $C(\overline D)$ (equipped with the supremum-norm) in the following cases, depending on the asymptotic behaviour of ${(\alpha(N))}_{N\in\N}$.
	\begin{itemize}
		\item If $\alpha(N)\gg N^2$, then the limit $m^\alpha$ is the law of the continuous Gaussian process on $\overline D$ whose covariance operator on $H$ is the
		inverse of the operator
		\begin{equation*}
			\Delta^2\quad\text{with domain }H^{4,2}(D)\cap H_0^{3,2}(D).
		\end{equation*}
		\item If $\alpha(N)\sim 2 N^2$, then the limit $m^\alpha$ is the law of the continuous Gaussian process on $\overline D$ whose covariance operator on $H$ is the
		inverse of the operator
		\begin{equation*}
			-\Delta+\Delta^2\quad\text{with domain }H^{4,2}(D)\cap H_0^{3,2}(D).
		\end{equation*}
	\end{itemize}
	 In particular, this is true regarding weak measure convergence on $H=L^2(D,\de z)$.
\end{exa}

\begin{prop}\label{prop:oder} It holds
	\begin{equation*}\label{eqn:tightOne:5}
		cN^{-2}|x|^2_\textnormal{euc}\leq {|\Lambda_Nx|}^2_H\leq 3^dN^{-2}|x|^2_\textnormal{euc}
	\end{equation*}
	for $x\in \R^{k_N}$, where $0<c\leq 1$ is a constant, independent from $N$ and $d$.
\end{prop}
\begin{proof}
	    The claimed inequality holds for all $x\in\R^{k_N}$ and some constant $c$, if and only if
	    \begin{equation}\label{eqn:lambdaEq}
	    	\frac{c}{N^d}\sum_{p\in G_N}y_{p}^2\leq
	    	\Big|\sum_{p\in G_N}y_{p}\,\xi_{N}^p\Big|^2_H\leq \frac{3^d}{N^d}\sum_{p\in G_N}y_{p}^2\quad\text{for all }y\in\R^{G_n}.
	    \end{equation}
	    
	    We choose $c:=2\int_{\R^d}\Xi(z)^2\de z-1$ and prove \eqref{eqn:lambdaEq}. Note that $0<c\leq 1$, since $\Xi$ meets Condition \ref{condi:Xi}. Let $y\in \R^{G_N}$ and $\lambda\in\R^{\Z^d}$ such that
	    $\lambda_\mathbf k=y_p$, if $\mathbf k=p$ for some $p\in G_N$, and $\lambda_\mathbf k=0$ else.
		Due to Lemma \ref{lem:ellipt}, it holds
	\begin{align*}
		c\sum_{p\in G_N}y_p^2&=c\sum_{\mathbf k\in\Z^d}\lambda_\mathbf k^2
		\leq \int_{\R^d}\Big(\sum_{\mathbf k\in \Z^d}\lambda_{\mathbf k}\Xi(z-\mathbf k)\Big)^2\de z\\
		&=N^d\int_{\R^d}\Big(\sum_{\mathbf k\in\Z^d}\lambda_{\mathbf k}\Xi(Nz-\mathbf k)\Big)^2\de z
		= N^d\int_{\R^d}\Big(\sum_{p\in G_N}y_p\,\xi_N^{p}(Nz-p)\Big)^2\de z\\
		&= N^d\,\Big|\sum_{p\in G_N}y_p\,\xi_N^{p}\Big|^2_H.
	\end{align*}
	this proves the first inequality of \eqref{eqn:lambdaEq}. On the other hand, 
	\begin{equation}\label{eqn:struggles}
		\big|{\langle y_p\,\xi_N^{p},y_q\,\xi_N^{q}\rangle}_H\big|
		 \leq \frac{y_p^2+y_q^2}{2}\int_{\R^d}\Xi (Nz-p)\Xi (Nz-q)\de z=\frac{y_p^2+y_q^2}{2N^d}
	\end{equation}
	for $y\in\R^{G_N}$ and $p,q\in G_N$. On top of that, $\max_{i=1,\dots ,d}|p_i-q_i|\leq 1$ is a necessary condition for ${\langle \xi_N^{p},\xi_N^{q}\rangle}_H\neq 0$, due to Condition \ref{condi:Xi} (i).
	Now, the missing upper bound of \eqref{eqn:lambdaEq} is evident from
	\begin{equation*}
		\Big|\sum_{p\in G_N}y_p\,\xi_{N}^p\Big|^2_H
		\leq\sum_{\substack{p,q\in G_N:\\\max_{i}|p_i-q_i|\leq 1}}\frac{y_p^2}{2N^d}
		+\sum_{\substack{p,q\in G_N:\\\max_{i}|p_i-q_i|\leq 1}}\frac{y_q^2}{2N^d}
		=\frac{3^d}{N^d}\sum_{p\in G_N}y_p^2,
	\end{equation*}
	where we first used \eqref{eqn:struggles} and then $\#\{q\in G_N$ $|$  
	$\max_{i=1,\dots ,d}|p_i-q_i|\leq 1\}= 3^d$ for $p\in G_N$.
\end{proof}

We now compare the validity of \eqref{eqn:wmcHeight} for
two different choices, $\Xi$ and $\widetilde\Xi$, assuming that each of them meets Condition \ref{condi:Xi}. 
So, for $p\in G_N$ (as in \eqref{eqn:GEen}) we have
\begin{equation*}
	\xi_N^p(z):=\Xi(Nz-p),\quad z\in\R^d,
\end{equation*}
on the one hand, and
\begin{equation*}
\widetilde \xi_N^p(z):=\widetilde \Xi(Nz-p),\quad z\in\R^d,
\end{equation*}
on the other. This leads to two different height maps,
\begin{equation*}
	\Lambda_Nx(z)=N^{\frac{d}{2}-1}\sum_{i=1}^{k_N}x_i\xi_{N,i}(z),\quad x\in\R^{k_N},\,z\in D,
\end{equation*}
with $\xi_{N,i}:=\xi_{N}^{p_i}$, as well as
\begin{equation*}
	\widetilde \Lambda_Nx(z)=N^{\frac{d}{2}-1}\sum_{i=1}^{k_N} x_i\widetilde \xi_{N,i}(z),\quad x\in\R^{k_N},\,z\in D,
\end{equation*}
with $\widetilde \xi_{N,i}:=\widetilde \xi_{N}^{p_i}$. It turns out, these two resulting height maps are asymptotically equivalent in the sense specified by the subsequent theorem.

\begin{thm}\label{thm:HeightMaps}
	\item Let $m_N$ be a Borel probability measure on $\R^{k_N}$ for $N\in\N$ and $m$ be a Borel probability measure on $H$. Regarding the weak convergence of measures on $H$ it holds
	\begin{equation*}
		m_N\circ\Lambda_N^{-1}\Rightarrow m\quad\textnormal{is equivalent to}\quad  m_N\circ\widetilde\Lambda_N^{-1}\Rightarrow m.
	\end{equation*}
\end{thm}
\begin{proof}
	At first, we prove the convergence
	\begin{equation}\label{eqn:tightOne:6}
		\lim_{N\to\infty}\int_D\Big|\sum_{p\in G_N}\varphi\big(\frac{p}{N}\big)\xi_N^p(z)-\varphi(z)\Big|^2\de z=0
	\end{equation}
	for $\varphi\in C(\overline D)$.
	To verify \eqref{eqn:tightOne:6}, consider that
	\begin{align*}
	\supp{\Xi(N\,\cdot\,-\mathbf k)}&\subseteq \frac{1}{N}\Big([\mathbf k_1-1,\mathbf k_1+1]\times\dots\times [\mathbf k_d-1,\mathbf k_d+1]\Big)\\
	&\subseteq \Big\{z\in\R^d\,\Big|\,\big|z-\tfrac{\mathbf k}{N}\big|_\euc\leq{\tfrac{\sqrt d}{N}}\Big\}
	\end{align*}
	for $\mathbf k\in\Z^d$. So, with \eqref{eqn:GEen} and Condition \ref{condi:Xi} (iii), we have
	\begin{equation*}
		\limsup_{N\to\infty}\int_{D}\Big|\sum_{\mathbf k\in \Z^d\setminus G_N}\Xi(Nz-\mathbf k)\Big|^2\de z\leq\limsup_{N\to\infty}\de z\big(D\cap B^\textnormal{ euc}_{{\sqrt d/N}}(\overline D\setminus D_N)\big)=0,
	\end{equation*}
	where $B^\textnormal{ euc}_{{\sqrt d/N}}(\overline D\setminus D_N):=\{z'\in\R^d$ $|$ $|z'-z|_\euc<{\sqrt d/N}$ for some $z\in \overline D\setminus D_N\}$.
	Then, the estimate
	\begin{align*}
		&\Big(\int_D\Big|\sum_{p\in G_N}\varphi\big(\frac{p}{N}\big)\xi_N^p(z)-\varphi(z)\Big|^2\de z\Big)^{\frac{1}{2}}\\&=
		\Big(\int_D\Big|\sum_{p\in G_N}\varphi\big(\frac{p}{N}\big)\xi_N^p(z)-\varphi(z)\sum_{\mathbf k\in\Z^d}\Xi(Nz-\mathbf k)\Big|^2\de z\Big)^{\frac{1}{2}}
		\\&\leq \Big(\int_D\Big|\sum_{p\in G_N}\Big(\varphi\big(\frac{p}{N}\big)-\varphi(z)\Big)\xi_N^p(z)\Big|^2\de z\Big)^{\frac{1}{2}}+\sup_{z\in \overline D}|\varphi(z)|\Big(\int_D\Big|\sum_{\mathbf k\in\Z^d\setminus G_N}\Xi(Nz-\mathbf k)\Big|^2\de z\Big)^{\frac{1}{2}}
	\end{align*}
	leads to the desired equality in \eqref{eqn:tightOne:6}, because
	\begin{multline*}
		\Big(\int_D\Big|\sum_{p\in G_N}\Big(\varphi\big(\frac{p}{N}\big)-\varphi(z)\Big)\xi_N^p(z)\Big|^2\de z\Big)^{\frac{1}{2}}\\\leq
		\sup_{\substack{z,z'\in\overline D\\|z-z'|_\euc\leq {\sqrt d/N}}}|\varphi(z)-\varphi(z')|
		\Big(\int_D\sum_{p\in G_N}\xi_N^p(z)^2\de z\Big)^{\frac{1}{2}}\\\leq
		\sqrt{\de z(D)}\sup_{\substack{z,z'\in\overline D\\|z-z'|_\euc\leq {\sqrt d/N}}}|\varphi(z)-\varphi(z')|.
	\end{multline*}
	In the last inequality, $\xi_N^p(z)^2\leq \xi_N^p(z)$ and $\sum_{p\in G_N}\xi_N^P(z)\leq 1$ is used.
    Now, since \eqref{eqn:tightOne:6} is proven, 
    \begin{equation}\label{eqn:tightOne:tilde}
	\lim_{N\to\infty}\int_D\Big|\sum_{p\in G_N}\varphi\big(\frac{p}{N}\big)\widetilde\xi_N^p(z)-\varphi(z)\Big|^2\de z=0
    \end{equation}
	holds likewise, of course.
	
	The rest of this proof is dedicated to show
	\begin{equation*}
	m_N\circ\Lambda_N^{-1}\Rightarrow m\quad\textnormal{implies}\quad  m_N\circ\widetilde\Lambda_N^{-1}\Rightarrow m
    \end{equation*}
    regarding weak measure convergence on $H$.
    If so, then the other direction holds automatically.
    So, we assume $m_N\circ\Lambda_N^{-1}\Rightarrow m$ from here on. First, the tightness of ${(m_N\circ\widetilde\Lambda_N^{-1})}_{N\in\N}$ is shown.
    For this purpose,  our first claim is that for each compact set $K_1\subset H$ there exists a compact set $K_2\subset H$ such that
	\begin{equation}\label{eqn:tightOne:4}
		\bigcup_{N\in\N}\widetilde\Lambda_N\circ\Lambda_N^{-1}(K_1)\subseteq K_2.
	\end{equation}
	To prove \eqref{eqn:tightOne:4}, we show that the set on the left-hand side is totally bounded. 
	It follows from Proposition \ref{prop:oder} that $\widetilde\Lambda_N\circ\Lambda_N^{-1}$ is a continuous linear operator from $\im(\Lambda_N)$ to $\im(\widetilde\Lambda_N)$ for $N\in\N$ with
	\begin{equation}\label{eqn:tightOne:opnormEst}
		\big|\widetilde\Lambda_N\circ\Lambda_N^{-1}h\big|_H\leq 3^{\frac{d}{2}}c^{-\frac{1}{2}}\,|h|_H,\quad h\in\im(\Lambda_N),
	\end{equation}
	where $c:=$ $2\int_{\R^d}\Xi(z)^2\de z- 1$ $\in(0,\infty)$ is the constant from that proposition.
	In the following, the open $\delta$-ball $\{h'\in H$ $|$ $|h-h'|_H<\delta\}$ 
	with centre $h\in H$  is denoted by $B(h,\delta)$ for $\delta\in(0,\infty)$.
	
	 Let $\varepsilon>0$ be fixed.
	Since $K_1$ is totally bounded and $C(\overline D)$ is dense in $H$, there exist $M\in\N$ and $\varphi_1,\dots,\varphi_M\in C(\overline D)$ such that
	\begin{equation}\label{eqn:K1}
		K_1\subseteq \bigcup_{l=1}^M  B(\varphi_l,3^{-\frac{d}{2}-1}c^{\frac{1}{2}}\varepsilon).
	\end{equation}
	Due to \eqref{eqn:tightOne:6} and \eqref{eqn:tightOne:tilde}, there exists $N_\varepsilon\in\N$ and $x_{N,l}\in\R^{k_N}$ for $l=1,\dots,M$, $N> N_\varepsilon$, such that
	\begin{equation*}
		\widetilde\Lambda_Nx_{N,l}\in B(\varphi_l,\varepsilon/3)\quad\textnormal{and}\quad\Lambda_Nx_{N,l}\in B(\varphi_l,3^{-\frac{d}{2}-1}c^{\frac{1}{2}}\varepsilon),\quad l=1,\dots,M,\,N> N_\varepsilon.
	\end{equation*}
	For each $N\in\N$ the set $\widetilde\Lambda_N\circ\Lambda_N^{-1}(K_1)$ is a compact and therefore there exists $k\in\N$ and $\eta_1,\dots,\eta_k\in C(\overline D)$ such that
	\begin{equation*}
		\bigcup_{N=1}^{N_\varepsilon}\widetilde\Lambda_N\circ\Lambda_N^{-1}(K_1)\subseteq\bigcup_{l=1}^k B(\eta_l,\varepsilon).
	\end{equation*}
	Next, we show 
	\begin{equation}\label{eqn:neps+1}
		\bigcup_{N= N_\varepsilon+1}^\infty\widetilde\Lambda_N\circ\Lambda_N^{-1}(K_1)\subseteq \bigcup_{l=1}^M  
		B(\varphi_l,\varepsilon),
	\end{equation}
	which would prove the total boundedness of $\bigcup_{N\in\N}\widetilde\Lambda_N\circ\Lambda_N^{-1}(K_1)$, as the choice  of $\varepsilon$ is arbitrary.
	To prove \eqref{eqn:neps+1}, we assume $h\in \widetilde\Lambda_N\circ\Lambda_N^{-1}(K_1)$ for some $N>N_\varepsilon$. Using \eqref{eqn:K1},
	there exists $l\in\{1,\dots,M\}$ and $y\in \Lambda_N^{-1}(B(\varphi_l,3^{-\frac{d}{2}-1}c^{\frac{1}{2}}\varepsilon))$
	with $\widetilde\Lambda_Ny=h$.
    Now, \eqref{eqn:tightOne:opnormEst} and the choices of $x_{N,l}$ yield
	\begin{align*}
		|\widetilde\Lambda_Ny-\varphi_l|_H&=\big|\widetilde\Lambda_N\circ\Lambda_N^{-1}(\Lambda_Ny-\Lambda_Nx_{N,l})+\widetilde\Lambda_Nx_{N,l}-\varphi_l\big|_H
		\\&\leq3^{\frac{d}{2}}c^{-\frac{1}{2}}\,\big|\Lambda_Ny-\Lambda_Nx_{N,l}\big|_H+\big|\widetilde\Lambda_Nx_{N,l}-\varphi_l\big|_H
		\\&\leq3^{\frac{d}{2}}c^{-\frac{1}{2}}\Big(\big|\Lambda_Ny-\varphi_l\big|_H
		+\big|\varphi_l-\Lambda_Nx_{N,l}\big|_H
		\Big)+\big|\widetilde\Lambda_Nx_{N,l}-\varphi_l\big|_H
		\\&\leq3^{\frac{d}{2}}c^{-\frac{1}{2}}\big(3^{-\frac{d}{2}-1}c^{\frac{1}{2}}\varepsilon+3^{-\frac{d}{2}-1}c^{\frac{1}{2}}\varepsilon\big)+\varepsilon/3=\varepsilon.
	\end{align*}
	This proves \eqref{eqn:neps+1}. The proof for the existence of a compact set $K_2\subset H$ satisfying \eqref{eqn:tightOne:4} for a fixed compact set $K_1$ is complete.

	Now, the tightness of ${(m_N\circ\widetilde\Lambda_N^{-1})}_N$ is an immediate consequence of the tightness of ${(m_N\circ\Lambda_N^{-1})}_N$. Indeed, for $0<\delta<1$ and a compact set $K_\delta\subset H$ such that
	\begin{equation*}
		\inf_{N\in\N}m_N(\Lambda_N^{-1}(K_\delta))\geq 1-\delta,
	\end{equation*}
	we can choose a compact set $K_\delta'\subset H$ with $\Lambda_N^{-1}(K_\delta)\subseteq\widetilde\Lambda_N^{-1} (K_\delta')$ for all $N\in\N$ due to \eqref{eqn:tightOne:4}. In particular,
	\begin{equation*}
		\inf_{N\in\N}m_N\circ\widetilde\Lambda_N^{-1}( K_\delta')
		\geq\inf_{N\in\N}m_N\circ\Lambda_N^{-1}(K_\delta)\geq 1-\delta.
	\end{equation*}
	
	Identifying the accumulation points of ${(m_N\circ\widetilde\Lambda_N^{-1})}_N$ is the only thing left to do in this proof. We make use of \cite[Theorem 4.5 of Chapter 3]{kurtz86}.
	It is sufficient to show
	\begin{equation}\label{eqn:tightOne:7}
		\lim_{N\to\infty}\int_{\R^{k_N}}F(\widetilde\Lambda_Nx)\de m_N(x)=
		\int_HF(h)\de m(h)
	\end{equation}
	 for any $F:H\to\R$ of the form $F(h)= g(\langle f_1,h\rangle_H,\dots,\langle f_q,h\rangle_H)$, where $q\in\N$, $g:\R^m\to\R$ is a bounded, Lipschitz continuous function   and $f_1,\dots,f_q:$ $D\to\R$ are continuous with compact support contained in $D$. The reason is that the family of all functions $F:H\to\R$ of this form separates the points of $H$.
	Let $F$, $g$ and $f_1,\dots,f_q$, $q\in\N$, be chosen in this manner. We have
	\begin{multline*}
		\limsup_{N\to\infty}\Big|\int_{\R^{k_N}}F(\widetilde\Lambda_Nx)\de m_N(x)
		-\int_HF(h)\de m(h)\Big|
		\\\leq
		\limsup_{N\to\infty}\Big|\int_{\R^{k_N}}F(\widetilde\Lambda_Nx)\de m_N(x)
		-\int_{\R^{k_N}}F(\Lambda_Nx)\de m_N(x)\Big|\\
		+\limsup_{N\to\infty}\Big|\int_{\R^{k_N}}F(\Lambda_Nx)\de m_N(x)
		-\int_HF(h)\de m(h)\Big|.
	\end{multline*}
	Due to the weak measure convergence $m_N\circ\Lambda_N^{-1}\Rightarrow m$, the second term of the right-hand side vanishes. Now, we address the first summand. Let $\varepsilon>0$. 
	There exists a compact set  $K_\varepsilon\subset H$ with 
	\begin{equation*}
		m_N\circ\Lambda_N^{-1}(K_\varepsilon)\geq 1-
		\frac{\varepsilon}{4\sup_{h\in h}|F(h)|}
	\end{equation*}
	for $N\in\N$. Let $L$ denote the Lipschitz constant of $g$.
	We have
	\begin{align*}
		&\Big|\int_{\R^{k_N}}F(\widetilde\Lambda_Nx)\de m_N(x)
		-\int_{\R^{k_N}}F(\Lambda_Nx)\de m_N(x)\Big|\\&\leq
		\int_{\Lambda_N^{-1}(K_\varepsilon)}\big|F(\widetilde\Lambda_Nx)
		-F(\Lambda_Nx)\big|\de m_N(x)+2\|F\|_\infty\int_{\R^{k_N}\setminus\Lambda_N^{-1}(K_\varepsilon) }\de m_N
		\\&\leq
		L\int_{\Lambda_N^{-1}(K_\varepsilon)}\Big(\sum_{k=1}^q\big|\langle f_k,\widetilde\Lambda_Nx-\Lambda_Nx\rangle_H\big|^2\Big)^{1/2}\de m_N(x)+\frac{\varepsilon}{2}.
    \end{align*}
    So, \eqref{eqn:tightOne:7} follows, if there exists $M_\varepsilon\in\N$ such that
    \begin{equation}\label{eqn:tightOne:7,5}
    \big|\langle f_k,\widetilde\Lambda_Nx-\Lambda_Nx\rangle_H\big|\leq\frac{\varepsilon}{2L\sqrt{q}}\quad\text{for }x\in \Lambda_N^{-1}(K_\varepsilon),\, k=1,\dots,q\text{ and }N\geq M_\varepsilon.
    \end{equation}
    
    We define
    \begin{equation*}
    	\omega_{N}(z):=\max_{k=1,\dots, q}\sup_{\substack{z'\in\overline D\\|z-z'|_\euc\leq \sqrt d/N}}|f_k(z')-f_k(z)|, \quad z\in\overline D.
    \end{equation*}
    By virtue of \eqref{eqn:tightOne:4} there exists a compact set $K'_\varepsilon$
    such that 
    \begin{equation}\label{eqn:contain}
    	\bigcup_{N\in\N}\widetilde\Lambda_N\circ\Lambda_N^{-1}(K_\varepsilon)\subseteq K'_\varepsilon.
    \end{equation}
    Then, we choose $M_\varepsilon^{(1)},M_\varepsilon^{(2)}\in\N$ such that
    \begin{equation}\label{eqn:2ndLast}
     \omega_{M_\varepsilon^{(1)}}(z)\leq\varepsilon\big(4L \max_{h\in K_\varepsilon\cup  K'_\varepsilon}|h|_H\big)^{-1}\Big(q\int_D\de z\Big)^{-\frac{1}{2}},\quad z\in D
    \end{equation}
    and
    \begin{equation*}
    	\bigcup_{k=1}^q\Big\{z\in\R^d\,\Big|\,|z-z'|_\euc\leq{\sqrt d/M_\varepsilon^{(2)}}\text{ for some }z'\in\supp {f_k}\Big\}
    	\subset D.
    \end{equation*}
    We set $M_\varepsilon:=\max(\{M_\varepsilon^{(1)},M_\varepsilon^{(2)}\})$.
    Let $k\in\{1,\dots,q\}$, $x\in\Lambda_N^{-1}(K_\varepsilon)$ and $N\geq M_\varepsilon$.
    \begin{multline}\label{eqn:sunday}
    	\big|\langle f_k,\widetilde\Lambda_Nx-\Lambda_Nx\rangle_H\big|\\
    	\leq \Big|\langle f_k,\widetilde\Lambda_Nx\rangle_H -N^{-\frac{d}{2}-1} \sum_{i=1}^{k_N}f_k\big(\frac{p_i}{N}\big)x_{i}\Big|
    	+\Big|N^{-\frac{d}{2}-1} \sum_{i=1}^{k_N}f_k\big(\frac{p_i}{N}\big)x_{i}-\langle f_k,\Lambda_Nx\rangle_H\Big|.
    \end{multline}
    Next, since $\int_{\R^d}\widetilde\xi_{N,i}\de z=N^{-d}$ and
    \begin{equation*}
    \supp{\widetilde\xi_{N,i}}\subseteq\Big\{z\in\R^d\,\Big|\,\big|z-\frac{p_i}{N}\big|_\euc\leq{\sqrt d/M_\varepsilon}
    \Big\}\subset D 
    \end{equation*}
    for $i\in\{1,\dots,k_N\}$ with $p_i/N\in\supp{f_k}$,  it holds
    \begin{multline*}
    	\Big|\langle f_k,\widetilde\Lambda_Nx\rangle_H -N^{-\frac{d}{2}-1} \sum_{i=1}^{k_N}f_k\big(\frac{p_i}{N}\big)x_{i}\Big|
    	=\Big|\langle f_k,\widetilde\Lambda_Nx\rangle_H -N^{\frac{d}{2}-1} \sum_{i=1}^{k_N}f_k\big(\frac{p_i}{N}\big)x_{i}\int_D\widetilde\xi_{N,i}(z)\de z\Big|\\
    	= \Big|N^{\frac{d}{2}-1}\sum_{i=1}^{k_N}\int_D \Big(f_k(z)-f_k\big(\frac{p_i}{N}\big)\Big)\,x_i\widetilde\xi_{N,i}(z)\de z\Big|
    	\leq \int_D \omega_N(z)\big|\widetilde\Lambda_Nx(z)\big|\de z\\
    	\leq \varepsilon \big(4L\max_{h\in K_\varepsilon\cup  K'_\varepsilon}|h|_H\big)^{-1}\Big(q\int_D\de z\Big)^{-\frac{1}{2}}\int_D \big|\widetilde\Lambda_Nx(z)\big|\de z
    	\leq \frac{\varepsilon}{4L\sqrt q}.
    \end{multline*}
    In the second to last inequality, we used \eqref{eqn:2ndLast}, while the last inequality follows from \eqref{eqn:contain}, $x\in \Lambda_N^{-1}(K_\varepsilon)$ and H\"older's inequality 
    $\int_D \big|\widetilde\Lambda_Nx(z)\big|\de z\leq|\widetilde\Lambda_Nx(z)|_H |\eins(\cdot)|_H$.
    With the analogous argumentation
    \begin{equation*}
    	\Big|\langle f_k,\Lambda_Nx\rangle_H -N^{-\frac{d}{2}-1} \sum_{i=1}^{k_N}f_k\big(\frac{p_i}{N}\big)x_{i}\Big|
    	\leq \Big(\int_D w_N(z)^2\de z\Big)^{\frac{1}{2}}\max_{h\in K_\varepsilon}|h|_H\leq \frac{\varepsilon}{4L\sqrt q}.
    \end{equation*}
    Hence \eqref{eqn:tightOne:7,5} holds and thereby \eqref{eqn:tightOne:7}.
     This concludes the proof. 
\end{proof}

\begin{exa}\label{exa:ZambCip}
	\begin{thm(ii)}
		\item 
		We consider the case of Example \ref{exa:wettingM}.
		and choose $\Xi=\eins_{[-1,0)}$ instead of the one-dimensional tent function and re-define $\Lambda_N$ according to \eqref{eqn:heightDef2} for $N\in\N$, i.e.
		\begin{equation*}
			\Lambda_Nx(z):=\frac{1}{\sqrt N}\sum_{i=1}^{N}x_{i}\eins_{[\frac{i-1}{N},\frac{i}{N})}(z),\quad x\in\R^N,\,
			z\in(0,1).
		\end{equation*}
		Still, regarding weak measure convergence on $H$, it holds
		 \begin{equation*}
		 	m_N^\beta\circ\Lambda_N^{-1}\Rightarrow m^\beta,
		 \end{equation*}
		with $m_N^\beta$, $m^\beta$ exactly as in Example \ref{exa:wettingM}, because of Theorem \ref{thm:HeightMaps}.
		\item  In Example \ref{exa:sfPoly}, we can pick the alternative choice $\Xi=\eins_{[-\frac{1}{2},\frac{1}{2})^d}$ instead of the two-dimensional tent function and re-define $\Lambda_N$ according to \eqref{eqn:heightDef2} for $N\in\N$, $N\geq 6$, i.e.
		\begin{gather*}
			\Lambda_Nx(z):=\sum_{i=1}^{k_N}x_i\eins_{[-\frac{1}{2},\frac{1}{2})^d}(Nz-I_N(i)) ,\quad x\in\R^{k_N},\,z\in (0,1)^2,\\
			\text{where}\quad G_N=:(2,N-2)^2\cap\Z^2, \quad k_N:=(N-5)^2,
		\end{gather*}
		for some bijection $I_N:\{1,\dots, k_N\}$ $\to$ $G_N$.
		By Theorem \ref{thm:HeightMaps}, it holds
		\begin{equation*}
			m^\alpha_N\circ\Lambda_N^{-1}\Rightarrow m^\alpha
		\end{equation*}
		regarding weak measure convergence on $H$,
		where $m^\alpha_N$, $m^\alpha$ are exactly as in Example \ref{exa:sfPoly}.
	\end{thm(ii)}
\end{exa}

	\section{The dynamic model}\label{sec:tightness}
	\subsection{Tightness of equilibrium laws}
We continue to work in the setting of Section \ref{sec:hvar}, 
with $(H,\langle\,\cdot\,,\,\cdot\,\rangle)=L^2(D,\de z)$, $\Lambda_N$, $\Xi$, $D_N$, $G_N$, $k_N$ and $p_1,\dots p_{k_N}\in G_N$ as in the previous section. 
In this section, a dynamic version of Theorem \ref{thm:HeightMaps} is derived.
By default, when referring to $C([0,\infty),Y)$ as a topological space for a metric space $Y$, we address the topology corresponding to the locally uniform convergence of functions $[0,\infty)\to Y$.

We make use of the theory of Dirichlet forms for symmetric Markov processes. 
The basic terminology and notions of the standard textbooks \cite{fukushima11,bouleau91,ma92} are used.
From here on, we have the following additional assumptions.
\begin{condition} \label{condi:tighti}
	\begin{thm(ii)}
		\item Let $H_0$ be a separable Hilbert space, densely included in $H$, such that the inclusion $i:H_0\hookrightarrow H$ has a finite Hilbert Schmidt norm
		$\|i\|_{\textnormal{HS}}$. Identifying $H$ with its dual $H'$ we can interpret $H$ as a subset of the dual $H_0'$ via
	\begin{equation*}
		H_0\subset H=H'\subset H_0'.
	\end{equation*}
	\item For each $N\in\N$ let $m_N$ be a Borel probability measure on $\R^{k_N}$  and
	$(\E^N,\dom(\E^N))$ be a symmetric, regular, local and conservative Dirichlet form on $L^2(\R^{k_N},m_N)$, which admits a Carré-du-Champs operator $$\Gamma_N:\dom(\E^N)\times\dom(\E^N)\to L^1(\R^{k_N},m_N).$$
	Moreover, each linear functional $\R^{k_N}\ni x\mapsto a^\textnormal{T}x\in\R$, $a\in\R^{k_N}$, 
	is  a representative of an element $l_a\in\dom(\E^N)$ with
	\begin{equation}\label{eqn:Gamma}
		\Gamma_N(l_a,l_a)(\cdot)\leq  \gamma \,a^\textnormal{T}a\quad m_N\text{-a.e.} ,
	\end{equation}
	where the constant $\gamma\in(0,\infty)$ does not depend on $N$.
	\end{thm(ii)}
\end{condition}

For each $N\in\N$ let 
$$\mathbf M_N=\big(\Omega_N,\mathcal F^N,{(\mathcal F^N_t)}_{t\geq 0},(X^N_t)_{t\geq 0},(P^N_x)_{x\in\R^{k_N}\cup \{\Delta\}} \big)$$ 
denote a special standard process with state space $\R^{k_N}$ and life time $\zeta_N$, which is properly associated with $(\E^N,\dom(\E^N))$.
Such a process exists, because $(\E^N,\dom(\E^N))$ is regular. Then, $\mathbf M_N$ is unique up to equivalence.
It holds
\begin{equation*}
	P^N_x(\{\omega\in\Omega_N\,|\,[0,\zeta_N(\omega))\ni t\mapsto X^N_t(\omega)\text{ is continuous}\})=1\quad\text{for }x\in\R^{k_N},
\end{equation*}
due to the local property of $\E^N$.
The strongly continuous symmetric contraction semigroup on $L^2(\R^{k_N},m_N)$ corresponding to \sloppy $(\E^N,\dom(\E^N))$ is denoted by ${(T^N_t)}_{t\geq 0}$. 
\begin{remark}\label{rem:weed} 
	Let $N\in\N$.
	By conservativeness, it holds $T^N_t\eins_{\R^{k_N}},=\eins_{\R^{k_N}}$ for $t\geq 0$. This means there exists a set $\mathcal N\subset \R^{k_N}$ of zero capacity (referring to the $1$-capacity associated with $\E^N$) such that  $P^N_x(\{\zeta_N=\infty\})=1$ for $x\in\R^{k_N}\setminus \mathcal N$. W.l.o.g.~$\R^{k_N}\setminus \mathcal N$ may be assumed to be $\mathbf M_N$-invariant, as argued in  \cite[Chap.~IV, Cor.~6.5]{ma92}. Considering the restriction $\mathbf M_N|_{\R^{k_N}\setminus \mathcal N}$ (as defined in \cite[Chap.~IV, Rem.~6.2(i)]{ma92}) 
	and then applying 
	the procedure described in \cite[Chapt.IV, Sect.3, pp.~117f.]{ma92}, re-defining $\mathbf M_N$ in such way that each element from $\mathcal N$ is a trap,  we may assume $P_x^N(\{\zeta_N=\infty\})=1$ for all $x\in\R^{k_N}$. 
	Furthermore, after the procedure of weeding (restricting the sample space to a subset of $\Omega_N$ as explained in \cite[Chap.~III, Paragraph 2, pp.~86f.]{dynkin2012markov}), we may w.l.o.g.~assume that $\mathbf M_N$ is non-terminating and continuous, i.e.~$\zeta_N(\omega)=\infty$ and $[0,\infty)\ni t\mapsto X^N_t(\omega)\in \R^{k_N}$ is continuous for every $\omega\in\Omega_N$.
\end{remark}

As a conclusion of Remark \ref{rem:weed}, the following corollary holds. 

\begin{corollary}\label{cor:Mexist}
	There exists a diffusion process, i.e.~a non-terminating, continuous special standard process $$\mathbf M_N=\big(\Omega_N,\mathcal F^N,{(\mathcal F^N_t)}_{t\geq 0},(X^N_t)_{t\geq 0},(P^N_x)_{x\in\R^{k_N}} \big)$$ on $\R^{k_N}$, 
	which is properly associated with $\E^N$. So, for every   $u\in L^\infty(\R^{k_N},m_N)$, the function  
	$$\R^{k_N}\ni x\mapsto \int_{\Omega_N}u\circ X^N_t\de P^N_x$$ 
	is a (quasi-continuous) version of $T^N_t u$ for $t\geq 0$. In particular, the measure $m_N$ is an invariant measure and  $\mathbf M_N$ is $m_N$-symmetric.
\end{corollary}

	We analyse the distribution of 
	${(u^{N}_t)}_{t\geq 0}:\Omega_N\to C([0,\infty), H)$ defined by
	\begin{equation}\label{eqn:dynLambda}
		u^N_t:=\Lambda_N X^{N}_{N^2t},\quad t\geq 0,
	\end{equation}
	for $N\in\N$. We are interested in the weak measure convergence 
	of the equilibrium law under ${(u^{N}_t)}_{t\geq 0}$ as $N\to\infty$. 
	A suitable path space to tackle this problem is $C([0,\infty), H_0')$.
	Denoting the $i$-th coordinate process of ${(X^N_t)}_{t\geq 0}$ with ${(X^{N,i}_t)}_{t\geq 0}$ for $i=1,\dots k_N$,
	we have
	\begin{equation}\label{eqn:dynXi}
		u^N_t=N^{\frac{d}{2}-1}\sum_{i=1}^{k_N}X^{N,i}_{N^2t}\,\Xi(N\,\cdot\,- p_i),\quad t\geq 0.
	\end{equation}
	Let
	\begin{equation*}
		P_N(A):=\int_{\R^{k_N}}P^N_x(A)\de m_N(x),\quad A\in\mathcal F^N.
	\end{equation*}
	The equilibrium law of ${(u^{N}_t)}_{t\geq 0}$ on $C([0,\infty), H_0')$ is denoted by $\mathbb P_N$, i.e.
	\begin{equation}\label{eqn:eqlaw}
		\mathbb P_N(B):=P_{N}\big(\big\{\omega\in\Omega_N\,\big|\,{(u^{N}_t(\omega))}_{t\geq 0}\in B\big\}\big)
	\end{equation}
	for a Borel measurable set $B\subseteq C([0,\infty), H_0')$.
	
\begin{prop}\label{prp:tightness}
	If ${(m_N\circ\Lambda_N^{-1})}_{N\in\N}$ is a tight family of probability measures on $H$, then ${(\mathbb P_N)}_{N\in\N}$ is tight on $C([0,\infty), H_0')$.
\end{prop}

\begin{proof}
	This proof follows a well-known idea for the derivation of a tightness result in the context of symmetric Markov processes, which has been used in \cite[Theorem 6.1]{kondratiev03}, \cite[Theorem 5.1]{kondratiev07} or \cite[Lemma 5.2]{debussche07} among others. 
	As argued in each of the three cited references, the desired tightness property can be concluded from the tightness of the invariant measures ${(m_N\circ\Lambda_N^{-1})}_{N\in\N}$ together with the estimate
	\begin{equation}\label{eqn:dedirest}
		\sup_{N\in\N}\Big(\int_{C([0,\infty), H_0')}|u_t-u_s|_{H'_0}^p\de\mathbb P_N(u)\Big)^{\frac{1}{p}}\leq c(t-s)^{\frac{1}{2}},\quad 0\leq s< t<\infty,
	\end{equation}
	for some constants $c\in(0,\infty)$ and $p\in (1,\infty)$.
	
	The rest of the proof is dedicated to the verification of \eqref{eqn:dedirest}. The strategy is analogous as
	in \cite[Theorem 6.1]{kondratiev03}, \cite[Theorem 5.1]{kondratiev07} or \cite[Lemma 5.2]{debussche07}.
	At first we fix $N\in\N$ and $\varphi\in H_0$.
	Let $r_T$ be the time-reversal operator $\Omega_N\to\Omega_N$ w.r.t.~time $0<T<\infty$.
	There is a $P_N$-martingale ${(M_t)}_{t\geq 0}$ with
    quadratic variation given by
	\begin{equation}\label{eqn:tightOne:QVar}
		{\bracket M}_t
		=2\int_0^t\Gamma_{N}\big(\langle\varphi,\Lambda_N\,\cdot\,\rangle,\langle\varphi,\Lambda_N\,\cdot\,\rangle\big)\big(X_s\big)\de s\quad P_N\text{-a.s.},\quad t\geq 0,
	\end{equation}
	such that
	\begin{equation}\label{eqn:tightOne:FwdBwd}
		\langle\varphi,\Lambda_N X_t\rangle- \langle\varphi,\Lambda_N X_0\rangle=\frac{1}{2}M_t+\frac{1}{2}( M_{T-t}\circ r_T-M_T\circ r_T)\quad P_N\text{-a.s.},\quad 0\leq t\leq T.
	\end{equation}
	The formula of \eqref{eqn:tightOne:FwdBwd} is called Lyons-Zheng decomposition and stated in \cite[Thm.~5.7.1]{fukushima11}. The identification of the quadratic variation in \eqref{eqn:tightOne:QVar} follows from \cite[Thm.~5.1.3.~\&~5.2.3]{fukushima11}. From \eqref{eqn:tightOne:QVar}, \eqref{eqn:heightDef2} and Condition \ref{condi:tighti} (ii) we obtain the estimate
	\begin{equation}\label{eqn:tightOne:1}
		{\bracket M}_t-{\bracket M}_s\,\leq(t-s) 2N^{d-2}\gamma \sum_{i=1}^{k_N}\langle\varphi,\xi_{N,i}\rangle^2\quad P_N\text{-a.s.},\quad 0\leq s\leq t.
	\end{equation}
	Now, we derive an estimate, first using \eqref{eqn:tightOne:FwdBwd} and the Minkowski inequality, then the Burkholder-Davis-Gundy inequality and \eqref{eqn:tightOne:QVar}. It holds
	\begin{align}\label{eqn:tightOne:2}
		&\Big(\int_{\Omega_N}\langle\varphi,u^N_t(\omega)-u^N_s(\omega)\rangle^4\de P_N(\omega)\Big)^{\frac{1}{4}}\nonumber\\
		&\leq\frac{1}{2}\Big(\int_{\Omega_N}(M_{N^2t}-M_{N^2s})^4\de P_N\Big)^{\frac{1}{4}}+\frac{1}{2}\Big(\int_{\Omega_N}(M_{T-N^2s}-M_{T-N^2t})^4\de P_N\Big)^{\frac{1}{4}} \nonumber\\
		&\leq \frac{C}{2}\Big(\int_{\Omega_N} \big({\bracket M}_{N^2t}-{\bracket M}_{N^2s}\big)^2\de P_N\Big)^{\frac{1}{4}}\nonumber\\&\quad\phantom{\leq\;}+\frac{C}{2}\Big(\int_{\Omega_N} \big({\bracket M}_{T-N^2s}-{\bracket M}_{T-N^2t}\big)^2\de P_N\Big)^{\frac{1}{4}},\quad 0\leq s\leq t\leq T/N^2.
	\end{align}
	with a constant $C\in(0,\infty)$ independent of $N$. 
	We continue the estimate of \eqref{eqn:tightOne:2}
	applying \eqref{eqn:tightOne:1}, then the Cauchy-Schwarz inequality together the estimates $\langle\xi_{N,i},\xi_{N,i}\rangle\leq N^{-d}$ for $i=1,\dots,k_N$
	and $\sum_{i=1}^{k_N}1_{\supp{\xi_{N,i}}}(\cdot)\leq 3^d$. It holds
	\begin{align}\label{eqn:tightOne:3}
		&\Big(\int_{\Omega_N}\big\langle\varphi,u^N_t(\omega)-u^N_s(\omega)\big\rangle^4\de P_N(\omega)\Big)^{\frac{1}{2}}\nonumber\\&\leq C^2(t-s)N^d\gamma\sum_{i=1}^{k_N}\langle\varphi,\xi_{N,i}\rangle^2
		\leq (t-s)C^2\gamma\sum_{i=1}^{k_N}\langle1_{\supp{\xi_{N,i}}}\varphi,\varphi\rangle \nonumber\\&\leq (t-s)(2\cdot3^d)C^2\gamma\langle\varphi,\varphi\rangle.
	\end{align}
	Now, we fix an orthonormal basis $\{\varphi_i|$ $i\in\N\}$ of $H_0$. Using the Minkowski inequality and then \eqref{eqn:tightOne:3} we obtain
	\begin{multline*}
		\Big(\int_{\Omega_N}{\big|u^N_t-u^N_s\big|}^4_{H'_0}\de P_N\Big)^{\frac{1}{2}}
		=\Big(\int_{\Omega_N}\Big(\sum_{i\in\N}{\big\langle\varphi_i,u^N_t-u^N_s\big\rangle}^2\Big)^2\de P_N\Big)^{\frac{1}{2}}\\
		\leq \sum_{i\in\N}\Big(\int_{\Omega_N}{\big\langle\varphi_i,u^N_t-u^N_s\big\rangle}^4\de P_N\Big)^{\frac{1}{2}}\leq (t-s)(2\cdot3^d)C^2\gamma\|i\|_{\textnormal{HS}}^2
	\end{multline*}
	or analogously $$\Big(\int_{C([0,\infty), H_0')}|u_t-u_s|_{H'_0}^4\de\mathbb P_N(u)\Big)^{\frac{1}{4}}\leq C{(2\cdot 3^d\gamma )}^\frac{1}{2} \|i\|_{\textnormal{HS}}
	(t-s)^{\frac{1}{2}}.$$ 
	This concludes the proof.
\end{proof}

As done in Theorem \ref{thm:HeightMaps} we now compare
two different choices, $\Xi$ and $\widetilde\Xi$, assuming that each of them meets Condition \ref{condi:Xi}. 
So, for $p\in G_N$ we have
\begin{equation*}
	\xi_N^p(z):=\Xi(Nz-p),\quad z\in\R^d,
\end{equation*}
on the one hand, and
\begin{equation*}
	\widetilde \xi_N^p(z):=\widetilde \Xi(Nz-p),\quad z\in\R^d,
\end{equation*}
on the other. This leads to two different height maps,
\begin{equation*}
	\Lambda_Nx(z)=N^{\frac{d}{2}-1}\sum_{i=1}^{k_N}x_i\xi_{N,i}(z),\quad x\in\R^{k_N},\,z\in D,
\end{equation*}
with $\xi_{N,i}:=\xi_{N}^{p_i}$, as well as
\begin{equation*}
	\widetilde \Lambda_Nx(z)=N^{\frac{d}{2}-1}\sum_{i=1}^{k_N} x_i\widetilde \xi_{N,i}(z),\quad x\in\R^{k_N},\,z\in D,
\end{equation*}
with $\widetilde \xi_{N,i}:=\widetilde \xi_{N}^{p_i}$. We define $\mathbb P_N$, respectively $\widetilde{\mathbb P_N}$, for these two height maps according to in \eqref{eqn:dynLambda}, \eqref{eqn:dynXi} and \eqref{eqn:eqlaw}.

\begin{corollary}\label{thm:Q1dyn}
	Assume that ${(m_N\circ\Lambda_N^{-1})}_{N\in\N}$ is a tight family of probability measures on $H$. Let $P$
	be a Borel probability measure on $C([0,\infty), H_0')$.
	Regarding weak measure convergence on $C([0,\infty), H_0')$ it holds
	\begin{equation*}
		{(\mathbb P_N)}_{N\in\N}\Rightarrow P\quad\text{if and only if}\quad
		{(\widetilde {\mathbb P_N})}_{N\in\N}\Rightarrow P.
	\end{equation*}
\end{corollary}
\begin{proof}
	This corollary is a consequence of Theorem \ref{thm:HeightMaps}, Proposition \ref{prp:tightness} and the Markov property.
	The assumption that ${(m_N\circ\Lambda_N^{-1})}_{N\in\N}$ is tight on $H$ (and so is ${(m_N\circ\widetilde \Lambda_N^{-1})}_{N\in\N}$ by Theorem \ref{thm:HeightMaps}) yields the tightness of ${(\mathbb P_N)}_{N\in\N}$
	and of ${(\widetilde{\mathbb P_N})}_{N\in\N}$ on $H_0'$ due to Proposition \ref{prp:tightness}.
	Therefore, it suffices to show the equivalence of weak convergence w.r.t.~the finite-dimensional distributions. The statement is easily obtained as a consequence of the Markov property and multiple applications of Theorem \ref{thm:HeightMaps}. We demonstrate
	the strategy in the case of a two-dimensional distribution. 
	We want to show that for any Borel probability measure $\mu$ on $H\times H$,
	\begin{multline}\label{eqn:startEQ}
		\lim_{N\to\infty}\int_{\R^{k_N}}g(\Lambda_Nx)T_t^N(f\circ\Lambda_N)(x)\de m_N(x)=
		\int_{H\times H} g(h_1)f(h_1)\de\mu(h_1,h_2) \\\text{for all }f,g\in C_b(H),
	\end{multline}
	is equivalent to
	\begin{multline}\label{eqn:goalEQ}
		\lim_{N\to\infty}\int_{\R^{k_N}}g(\widetilde\Lambda_Nx)T_t^N(f\circ\widetilde \Lambda_N)(x)\de m_N(x)=
		\int_{H\times H} g(h_1)f(h_1)\de\mu(h_1,h_2) \\\text{for all }f,g\in C_b(H).
	\end{multline}
	We remark that naturally Theorem \ref{thm:HeightMaps} applies regarding the weak measure convergence of 
	finite measures (instead of probabilities), as well.
	Now, applying Theorem \ref{thm:HeightMaps} onto \eqref{eqn:startEQ} for a fixed but arbitrary choice of $f\in C_b(H)$, we see that
	\eqref{eqn:startEQ} is equivalent to
	\begin{multline}\label{eqn:nextEQ}
		\lim_{N\to\infty}\int_{\R^{k_N}}g(\widetilde\Lambda_Nx)T_t^N(f\circ\Lambda_N)(x)\de m_N(x)=
		\int_{H\times H} g(h_1)f(h_1)\de\mu(h_1,h_2) \\\text{for all }g,f\in C_b(H).
	\end{multline}
	Next, using the symmetry of $T_t^N$ for $N\in\N$ and applying Theorem \ref{thm:HeightMaps} onto \eqref{eqn:nextEQ}
	for a fixed but arbitrary choice of $g\in C_b(H)$, it follows that \eqref{eqn:nextEQ} is equivalent to
	\begin{multline*}
		\lim_{N\to\infty}\int_{\R^{k_N}}T_t^Ng(\widetilde\Lambda_Nx)(f\circ\widetilde \Lambda_N)(x)\de m_N(x)=
		\int_{H\times H} g(h_1)f(h_1)\de\mu(h_1,h_2) \\\text{for all }g,f\in C_b(H).
	\end{multline*}
	Hence, the equivalence of \eqref{eqn:startEQ} and  \eqref{eqn:goalEQ} is proven.
	This strategy can be generalized to show the  equivalence of weak convergence for any finite-dimensional distributions of ${(\mathbb P_N)}_{N\in\N}$ and ${(\widetilde{\mathbb P_N})}_{N\in\N}$ in a straight forward way by induction.
\end{proof}

	\subsection{Gradient forms and dynamics}
From here on, we make an assumption on the height map $\Lambda_N$ regarding the property of the set $G_N=ND_N\cap \Z^d$, complementing the condition of \eqref{eqn:GEen}.
We assume that there exists $M\in\N$ such that for $N\geq M$:
 $D_N$ is a bounded open set, whose closure
$\overline D_N$ is a strict subset of $D$ with distance from $\R^d\setminus D$ being  larger or equal  $\sqrt{d}/N$, i.e.
\begin{equation*}
	|z-z_0|_\textnormal{euc}\geq \sqrt d/N\quad\text{if}\quad z_0\in D_N,\,z\in\R^{d}\setminus D.
\end{equation*}
We note that this amendment on the definition of $\Lambda_N$ is inline with the height map considered in Example \ref{exa:sfPoly}.
Then, for $N\in\N$ large enough, the support of the function $\xi_{N,i}=\Xi(N\,\cdot\,-p_i)$ is contained in $D$ for $i=1,\dots, k_N$. In particular, for all asymptotic statements, we may assume that
\begin{equation*}
	\int_D\xi_{N,i}(z)\de z=N^{-d},\quad i=1,\dots,k_N.
\end{equation*}

\begin{remark}\label{rem:MplusT}
	\begin{thm(ii)}
	\item Let's assume the weak measure convergence of ${(m_N\circ\Lambda_N^{-1})}_{N\in\N}$ towards a Borel probability measure $m$ on $H$ is given. Then, Mosco convergence of Dirichlet forms in the framework of Kuwae-Shioya provides a method to identify the accumulation points granted by Proposition \ref{prp:tightness} and thus prove the weak convergence of equilibrium laws. This concept is well-known and exploited in many surveys, e.g.~\cite{kondratiev07, toelle06, KOLESNIKOV06, kolesnikov05, Zambotti04}.
	The Mosco convergence of those gradient-type Dirichlet forms, which are relevant to the scaling limit of skew interacting Brownian motions in the context of mesoscopic interface models, is the topic of the subsequent Section \ref{sec:Mokoko}.
	
	\item While the question of  Mosco convergence is more involved,  
	the pointwise convergence of gradient forms on the space 
	\begin{equation*}
	 \mathcal C:=\Big\{H\ni h\mapsto g(\langle\,\cdot\,, f_1\rangle,\dots,\langle\,\cdot\,, f_m\rangle)
	 \,\Big|\,m\in\N,\,g\in C_b^1(\R^m),\, f_1,\dots,f_m\in C(\overline D)\Big\}
	\end{equation*}
	is easily shown, as the next Lemma \ref{lem:strConvF} demonstrates. It is included at this point, 
	because the appearing scaling factor of $N^2$ reflects the choice of time scale in \eqref{eqn:dynLambda},
	and also to further motivate the discussion of Section \ref{sec:Mokoko}.
	\end{thm(ii)}
\end{remark}

The elements in $\mathcal C$ are Fréchet differentiable with gradient
\begin{equation*}
	\grad g(\langle\,\cdot\,, f_1\rangle,\dots,\langle\,\cdot\,, f_m\rangle)(h)
	=\sum_{i=1}^m \partial_ig(\langle h, f_1\rangle,\dots,\langle h, f_m\rangle)f_i
\end{equation*}
at a point $h\in H$, where $m\in\N,\,g\in C_b^1(\R^m),\, f_1,\dots,f_m\in C(\overline D)$.

\begin{lemma}\label{lem:strConvF}
	We assume the weak measure convergence of ${(m_N\circ\Lambda_N^{-1})}_{N\in\N}$ towards a Borel probability measure $m$ on $H$.
	For $F,G\in\mathcal C$ it holds
	\begin{equation*}
	\lim_{N\to\infty}N^2\sum_{i=1}^{k_N}\int_{\R^{k_N}}\frac{\partial (F\circ\Lambda_N)}{\partial x_i}\frac{\partial (G\circ\Lambda_N)}{\partial x_i}\de m_N(x)=\int_H\langle\grad F,\grad G\rangle\de m.
	\end{equation*}
\end{lemma}
\begin{proof}	
    We first show that
    \begin{equation}\label{eqn:fff}
    \lim_{N\to\infty}\sum_{p\in G_N}N^d\langle\xi_N^p,f\rangle\langle\xi_N^p,g\rangle =\langle f,g\rangle
    \end{equation}
    for $f,g\in C(\overline D)$. 
    The statement is equivalent to $\lim_{N\to\infty}\int_Df_N(z)g(z)\de z=\int_Df(z)g(z)\de z$
     with
    \begin{equation*}
    	f_N(z):=\sum_{p\in G_N}N^d\langle\xi_N^p,f\rangle\xi_N^p(z),\quad z\in D.
    \end{equation*}
    It holds $|f_N(z)|\leq \sup_{z'\in\overline D}f(z')$ for $z\in D$. So,  Lebesgue's dominated convergence can be applied.    
     Since the increasing sets ${(D_N)}_{N\in\N}$ exhaust $D$ up to a set of Lebesgue measure zero (see the assumptions in \eqref{eqn:GEen}),
    it suffices to fix an arbitrary choice for $N_0\in\N$ and prove that $f_N$ converges to $f$ pointwisely on $D_{N_0}$ as $N\to\infty$.
    Let $N_0\in\N$ and $K$ be a compact set contained in $D_{N_0}$. Since $D_{N_0}\cap\Z^d\subset D_N\cap\Z^d=G_N$ for $N\geq N_0$, the support of a function $\Xi(Nz-\mathbf k)$, $\mathbf k\in\Z^d\setminus G_N$, does not intersect with $K$, if $N$ is chosen large enough. So, for $z_0\in K$ and $N$ large enough, we have
    \begin{equation*}
    	\sum_{p\in G_N}\xi_N^p(z_0)=\sum_{p\in G_N}\Xi(Nz_0-p)
    	=\sum_{\mathbf k\in\Z^d}\Xi(Nz_0-\mathbf k)=1.
    \end{equation*}
    In particular, for $N$ large enough and $z_0\in K$, using $\int_D\xi_N^p(z)\de z=N^{-d}$ 
    for $p\in G_N$, it holds
	\begin{align*}
		&\Big|\sum_{p\in G_N}N^d\int_D\xi_N^p(z)f(z)\de z\:\xi_N^P(z_0)-f(z_0)\Big|\\
		&=\Big|\sum_{p\in G_N}N^d\int_D\xi_N^p(z)f(z)\de z\:\xi_N^P(z_0)-
		f(z_0)N^d\sum_{p\in G_N}\xi_N^p(z_0)\int_D \xi_N^p(z)\de z\Big|
		\\&\leq N^d\sum_{p\in G_N}\int_{\R^d}\xi_N^p(z)|f(z)-f(z_0)|\de z\,
		\xi_N^p(z_0)
		\leq\max_{\substack{z\in \overline D\\|z_0-z|_\textnormal{euc}\leq {\frac{ 2 \sqrt d}{N}}}}|f(z)-f(z_0)|.
	\end{align*}
    This proves equation \eqref{eqn:fff}.
	
	From \eqref{eqn:fff}, in turn, it follows that
		\begin{equation}\label{eqn:eveneven}
			\lim_{N\to\infty}N^d\sum_{p\in G_N}\int_{\R^{k_N}}\langle \xi_{N}^p,V(\Lambda_Nx)\rangle
			\langle \xi_{N}^p,g\rangle\de m_N(x)=\int_H \langle V(h),g\rangle\de m(h)
		\end{equation}
		 if 
		 \begin{equation}\label{eqn:typespec}
		 	V:H\ni h\;\mapsto\; G_1(h)f_1+\dots G_m(h)f_m\in H,
		 \end{equation}
		 for some $G_1,\dots G_m\in C_b(H)$ and $f_1,\dots, f_m\in C_b(\overline D)$. 
		 By linearity, \eqref{eqn:eveneven} generalizes to
		 \begin{equation}
		 	\lim_{N\to\infty}N^d\sum_{p\in G_N}\int_{\R^{k_N}}\langle \xi_{N}^p,V(\Lambda_Nx)\rangle
		 	\langle \xi_{N}^p,W(\Lambda_Nx)\rangle\de m_N(x)=\int_H \langle V,W\rangle\de m
		 \end{equation}
		 if $W(h)$ is of the same type as $V(h)$, specified in \eqref{eqn:typespec}.
		 This proves the claim of the lemma, because 
		 \begin{equation*}
		 	\frac{\partial (F\circ\Lambda_N)}{\partial x_i}(x)
		 	=N^{\frac{d}{2}-1}\big\langle \xi_{N,i},\grad F(\Lambda_Nx)\big\rangle_H,\quad x\in\R^N,\quad 
		 	i=1,\dots,k_N,
		 \end{equation*}
		 and $\grad F$ is a vector field of the type specified in \eqref{eqn:typespec} for $F\in \mathcal C$.
\end{proof}

\begin{remark}\label{rem:remmi}
	\begin{thm(ii)}
	\item Let $\mu$ be a non-degenerate, centred Gaussian measure on $H$ with covariance operator $Q$
	and let $(A,\dom(A)):=(2Q^{-1},\textnormal{Im}(Q))$ be defined as a positive, self-adjoint, operator on $H$.
	We consider a reference measure $m=\rho\mu$ with a density 
	 $\rho:H\to[c_1,c_2]$, where $0<c_1<c_2<\infty$ and $\int_H\rho\de\mu=1$.
	In this case, 
	$$\E(F,G):=\frac{1}{2}\int_H\langle\grad F,\grad G\rangle\rho\de \mu,\quad F,G\in\mathcal C,$$
	can be extended (in the sense of a minimal closed extension) to a local, quasi-regular and conservative Dirichlet form $(\E,\dom(\E))$ on $L^2(H,\rho\mu)$, for which $\mathcal C$ is a form core (see \cite{Ro2zhu, sheng92, ROCKNER19921, albeverio90}).
	The domain $\dom(\E)$  coincides with the Gaussian $1,2$-Sobolev space on $H$ w.r.t.~$\mu$.
	We assume that
	\begin{equation}\label{eqn:pert}
		\rho(h)=\exp\Big(-\int_D g(h(z))\de z\Big),\quad h\in H,
	\end{equation}
	for some bounded function $g:\R\to\R$ with bounded total variation.	
	There exists a densely and continuously included Hilbert space $H_1\subset H$,
	such that $\rho$ is a function of bounded variation in the Gelfand triple
	\begin{equation*}
		H_1\subset H=H'\subset H_1'.
	\end{equation*}
	Indeed, this is true if $H_1$ can be continuously embedded into the space of continuous functions on $\overline D$. So, we may for example choose $H_1$ as the Sobolev space of square-integrable functions on $D$ with an order which is strictly greater than $d/2$.
	
	Let $(\nabla,\dom(\nabla))$ be the gradient operator $L^2(H,\mu)\to L^2(H,\mu,H)$, whose domain is the Gaussian $1,2$-Sobolev space w.r.t.~$\mu$. We denote the adjoint operator, a closed operator $L^2(H,\mu,H)\to L^2(H,\mu)$, by $(\nabla^*,\dom(\nabla^*))$.
	Then,
	\begin{multline*}
		\Big\{H\ni h\mapsto\sum_{l=1}^n f_l(h)\varphi_l\in H\,\Big|\,n\in\N,\,f_l\in C_b^1(H),\,\varphi_l\in\dom(A)\cap H_1 \Big\}\\=:(C_b^1)_{\dom(A)\cap H_1}\subset
		\dom(\nabla^*).
	\end{multline*}
	Due to \cite[Thm.~3.1]{Ro2zhu}, there exists a 
	a finite positive measure $\|\de\rho\|$ on $H$ and a Borel measurable map $\sigma_\rho:H\to H_1'$ such that $|\sigma_\rho(h)|_{H_1'}=1$ for $\|\de\rho\|$-a.e.~$h\in H$ and	
	\begin{equation*}
		\int_H\nabla^*G(h)\rho(h)\de \mu(h)=\int_H \leftindex_{H_1}\langle G(h),\sigma_\rho(h){\rangle}_{H_1'}\|\de\rho\|(h),\quad G\in (C_b^1)_{\dom(A)\cap H_1}.
	\end{equation*}
\item	Let 
	$$\mathbf M=\big(\Omega,\mathcal F,{(\mathcal F_t)}_{t\geq 0},(X_t)_{t\geq 0},(P_h)_{h\in H} \big)$$ 
	be the diffusion process on $H$, properly associated with $\E$.
	 By virtue of \cite[Thm.~3.2]{Ro2zhu} for quasi almost every $h\in H$ we have
	\begin{multline}\label{eqn:RZZ}
		\langle \varphi_l,X_t-X_0\rangle=\int_0^t\langle \varphi,\de W^h_s\rangle+\frac{1}{2}
		\int_0^t \leftindex_{H_1}\langle\varphi,\sigma_\rho(X_s){\rangle}_{H_1'}\de L_s^{\|\de\rho\|}
		\\-\int_0^t\langle A\varphi,X_s\rangle\de s,\quad t\geq 0,\,P_h\text{-a.s.}
	\end{multline}
	for $\varphi\in \dom(A)\cap H_1$, where ${(W^h_t)}_{t\geq 0}$ is an 
	${(\mathcal F_t)}_t$-cylindrical Wiener process and ${(L_t^{\|\de\rho\|})}_{t\geq 0}$ is the positive continuous additive functional of the process ${(X_t)}_{t\geq 0}$ in Revuz correspondence with $\|\de\rho\|$.
	\end{thm(ii)}
\end{remark}

\begin{exa}\label{exa:moveto}
	In this example, we fix $N\in\N$, a symmetric, positive, linear operator  $A_N:\R^{k_N}\to\R^{k_N}$,
	 a function $g_{N,0}\in C_b^1(\R)$ and real numbers $y_j,\beta_j\in\R$ for $j=1,\dots,M$. 
	 		\begin{thm(ii)}
	 \item  With
		\begin{equation*}
		g_N(y):=g_{N,0}(y)+\sum_{j=1}^M\beta_j\eins_{(-\infty,y_j]}(y),\quad y\in\R,
	\end{equation*}
	 we define the Borel probability measure
	\begin{gather*}
		\de m_{A_N}(x):=\frac{1}{Z_N}\exp\Big(-x^\textnormal{T}A_Nx-\frac{1}{N^d}\sum_{i=1}^{k_N} g_N\big(N^{\frac{d}{2}-1}x_i\big)\Big)\de x,\\
		Z_N:=\int_{\R^{k_N}}\exp\Big(-x^\textnormal{T}A_Nx-\frac{1}{N^d}\sum_{i=1}^{k_N} g_N\big(N^{\frac{d}{2}-1}x_i\big)\Big)\de x,
	\end{gather*}
	on $\R^{k_N}$.	The factor 
$$\exp\Big(-\frac{1}{N^d}\sum_{i=1}^{k_N} g_N\big(N^{\frac{d}{2}-1}x_i\big)\Big),\quad x\in\R^{k_N},$$ should be interpreted as  a discrete version of the perturbing density considered in \eqref{eqn:pert}.
	In the situation of Example \ref{exa:sfPoly}, i.e.~$d=2$ with $N\geq 6$, $D=(0,1)^2$, 
	$G_N=(2, N-2)^2$ and $k_N:=|G_N|$, we might choose $A_N$ such that
	\begin{equation}\label{eqn:Aham}
		x^\textnormal{T}A_Nx=4H_N^{\alpha}(x),\quad x\in\R^{k_N},
	\end{equation}
	with the Hamiltonian function $H_N^{\alpha}$  defined in that example.
	The scaling sequence $\alpha:\N\to(0,\infty)$ determines the limiting distribution of the height
	map $\Lambda_N$ for $N\to\infty$, if there is no perturbing density ($g_N\equiv0$).
	Now, we state the Fukushima decomposition for the dynamic height variables from the microscopic perspective,
	if a perturbing density of the type above is present.

    \item As  $g_N$ is a bounded function, we define the domain of the gradient-type Dirichlet form $(\E^N,\dom(\E^N))$ on $L^2(\R^{k_N},m_{A_N})$
    as the set of all elements in the $1,2$-Sobolev space of the Gaussian measure with density proportional to $x\mapsto \exp(-x^\textnormal{T}A_Nx)$ on $\R^{k_N}$. So, let
	\begin{align*}
		\E^N(u,v):=\frac{1}{2}\sum_{i=1}^{k_N}\int_{\R^{k_N}}\frac{\partial u}{\partial x_i}\frac{\partial v}{\partial x_i}\de m_{A_N},\quad u,v\in\dom(\E^N).
	\end{align*}
	The domain $\dom(\E^N)$ comprises all weakly differentiable elements of $L^2(\R^{k_N},m_{A_N})$ whose derivatives in any direction are again square-integrable w.r.t.~$m_{A_N}$.
	The form $(\E^N,\dom(\E^N))$ obviously meets Condition \ref{condi:tighti} (ii). Let 
	$$\mathbf M_N=\big(\Omega_N,\mathcal F^N,{(\mathcal F^N_t)}_{t\geq 0},(X^N_t)_{t\geq 0},(P^N_x)_{x\in\R^{k_N}} \big)$$ be the associated Markov process as in the beginning of this section. 
	Again, $(X^{N,i}_t)_{t\geq 0}$ for $i=1,\dots k_N$ denotes the $i$-th coordinate process.
    By Lemma \ref{lem:fukudec} of the Appendix it holds
\begin{multline}\label{eqn:skewIB}
	\de X^{N,i}_t=-\big(A_N X^{N}_t\big)_i\de t-\frac{1}{2}N^{-\tfrac{d}{2}-1}\,g_{N,0}'\big(N^{\frac{d}{2}-1}X^{N,i}_t\big)\de t
	\\+\sum_{j=1}^M\tfrac{1-e^{-\beta_j/N^d}}{1+e^{-\beta_j/N^d}}\de l_t^{N,i,\tilde y_j}
	+\de B_t^{N,i},\quad t\geq 0,
\end{multline}
$P_x^N$-a.s.~for each $x\in\R^{k_N}$,
where ${(B_t^{N,i})}_{t\geq 0}$, $i=1,\dots, k_N$, are independent Brownian motions
and $l_t^{N,i,\tilde y_j}$ is the local time of ${(X^{N,i}_t)}_{t\geq 0}$ at $\tilde y_j:=N^{1-\frac{d}{2}}y_j$.
	The linear term of the drift in \eqref{eqn:skewIB}
	for the situation of Example \ref{exa:sfPoly} with \eqref{eqn:Aham} equals
	\begin{equation*}
		-\big(A_N X^{N}_t\big)_i=-2\,\partial_i H_N^\alpha (X^N_t).
	\end{equation*}
\end{thm(ii)}
\end{exa}    
    
	We continue in the setting of Example \ref{exa:moveto} and explain the main motivation 
	for the upcoming Section \ref{sec:Mokoko} on Mosco convergence.
	The following Remark \ref{rem:Argu} shows how the results of Section \ref{sec:Mokoko} (in particular, see Theorem \ref{thm:maiMO} and Corollary \ref{cor:mai}) apply to this case. Let ${(T^N_t)}_{t\geq 0}$ denote the semigroup 
	on $L^2(\R^{k_N},m_{A_N})$ corresponding to $(\E^N,\dom(\E^N))$. 
	We are interested in the weak measure convergence of the equilibrium laws ${(\mathbb P_N)}_{N\in\N}$ for the scaled, interpolated process
    \begin{equation*}
	u^N_t:=\Lambda_N X^{N}_{N^2t},\quad t\geq 0,
    \end{equation*}
	(as in \eqref{eqn:eqlaw},\eqref{eqn:dynXi} and \eqref{eqn:dynLambda}).
	The finite-dimensional distributions of $\mathbb P_N$ are given in terms of the semigroup
	${(T^N_t)}_{t\geq 0}$ for each $N\in\N$. Let $n\in\N$ and $f_1,\dots,f_n$ be bounded measurable functions $H\to\R$.
	Setting $\tilde f_l=f_l\circ\Lambda_N$ for $l=1,\dots, n$ the Markov property of the process $\mathbf M_N$ and the definition of $\mathbb P_N$ imply
	\begin{multline}\label{eqn:finidi}
		\int_{C([0,\infty),H)}f_1(u_{t_1})\cdot f_2(u_{t_1+t_2})\cdot\dots\cdot f_n(u_{t_1+\dots+t_N})\de\mathbb P_N(u)
		\\=\int_{\R^{k_N}}T_{N^2t_1}^N\big(\tilde f_1\cdot T^N_{N^2t_2}\big(\dots T_{N^2t_{k-1}}^N\big(\tilde f_{k-1}\cdot T_{N^2t_k}^N\tilde f_k\big)\dots\big)\big)\de m_{A_N}.
	\end{multline}
	Further, let $\rho$, $\mu$ and $(\E,\dom(\E))$ be as in Remark \ref{rem:remmi}.
	The expected limit of ${(\mathbb P_N)}_{N\in\N}$ is the equilibrium law associated to $\E$,
	if ${(g_N)}_N$ is chosen in a suitable way.
	Arguing as in Remark \ref{rem:MplusT} (i) we observe the following.
	\begin{remark}\label{rem:Argu}
	By virtue of Proposition \ref{prp:tightness} and \eqref{eqn:finidi}, it suffices prove:
	\begin{itemize}
		\item $m_{A_N}\circ\Lambda_N^{-1}\Rightarrow \rho\mu$ weakly on $H$.
		\item The image form of $(N^2\E^N,\dom(\E))$ under the height map $\Lambda_N$ converges to $(\E,\dom(\E))$ in the sense of Mosco-Kuwae-Shioya.
	\end{itemize}
	Thanks to Theorem \ref{thm:HeightMaps} and Corollary \ref{thm:Q1dyn}, if these two statements are true for a particular choice for $\Xi:\R^d\to[0,1]$ in the definition of \ref{eqn:heightDef} for $\Lambda_N$, then they automatically hold for any choice according to Condition \ref{condi:Xi}. This is very useful, because for 
	$\Xi:=\eins_{[-1,0)^d}$ the functions $\xi_{N,i}=\Xi(N\,\cdot\,-p_i)$, $i=1,\dots,k_N$, which are the images of the unit-vectors under $N^{1-\frac{d}{2}}\Lambda_N$, are again orthogonal to each other. Moreover, 
	\begin{equation*}
		\widehat \xi_{N,i}:=N^{\frac{d}{2}}\xi_{N,i},\quad i=1,\dots k_N,
	\end{equation*}
	is normed in $H$, which means that $N\Lambda_N$ is an isometry from $\R^{k_N}$ into $H$.
	Hence, in case $\Xi:=\eins_{[-1,0)^d}$,  the image form of $(N^2\E^N,\dom(\E))$ under the height map $\Lambda_N$ simply equals
	\begin{align*}
		N^2\E^N(u\circ\Lambda_N,v\circ\Lambda_N)&=
		\E^N(u\circ N\Lambda_N,v\circ N\Lambda_N)\\
		&=\frac{1}{2}\sum_{i=1}^{k_N}\int_H \frac{\partial u}{\partial \widehat\xi_{N,i}}
		\frac{\partial v}{\partial \widehat\xi_{N,i}}\de (m_{A_N}\circ\Lambda_N)^{-1},\quad u,v\in\mathcal C,
	\end{align*}
	where $\widehat\xi_{N,1},\dots, \widehat\xi_{N,k_N}$ is an orthonormal basis for the $k_N$-dimensional subspace $(\Lambda_N(\R^{k_N}),\langle\,\cdot\,,\,\cdot\,\rangle)$ of $H$. This reduces the task of the final Section \ref{sec:Mokoko} to the analysis of Mosco-Kuwae-Shioya convergence of gradient forms in the framework of varying $L^2$-spaces over the state space $H$. Moreover, we hint at the fact, that for $\Xi:=\eins_{[-1,0)^d}$ we have $\int_D\xi_{N,i}\de z=N^{-d}$ for $i=1,\dots,k_N$ and therefore
	\begin{align*}
		\exp\Big(\int_D g_N\circ (\Lambda_Nx)\de z\Big)=
		\exp\Big(\frac{1}{N^d}\sum_{i=1}^{k_N}g_N(N^{\frac{d}{2}-1}x_i)\Big),\quad x\in\R^{k_N}.
	\end{align*}
	This means, 
	\begin{equation*}
		\de (m_{A_N}\circ\Lambda_N)^{-1}(h)=\exp\Big(\int_D g_N\circ h\de z\Big)(\de\mu_{A_N}\circ\Lambda_N^{-1})(h),
	\end{equation*}
	where $\mu_{A_N}$ is the Gaussian measure on $\R^{k_N}$ with covariance operator $\frac{1}{2}A_N^{-1}$. 
	\end{remark}
	
	\section{An application of Mosco convergence}\label{sec:Mokoko}
	\subsection{Preliminaries and notation}\label{sec:comp}
As in previous sections, $(H,\langle\cdot,\cdot\rangle):=L^2(D,\de z)$.
Let ${(\mu_N)}_{N\in\N}$ be a sequence of Borel probabilities on $H$ and $\mu$ be their limit in the sense of weak measure convergence on $H$. We assume that all these elements are centred Gaussian and on top of that, the limit $\mu$ has full topological support on $H$. Their covariances are denoted by 
\begin{equation*}
	\cov\mu(h_1,h_2):=\int_H\langle h_1,k\rangle\langle h_2,k\rangle\de\mu(k),\quad h_1,h_2\in H,
\end{equation*}
and
\begin{equation*}
	\cov{\mu_N}(h_1,h_2
	):=\int_H\langle h_1,k\rangle\langle h_2,k\rangle\de\mu_N(k),\quad h_1,h_2\in H,
\end{equation*}
for $N\in\N$.

Throughout Section \ref{sec:comp} we fix a sequence ${(\varphi_N)}_{N\in\N}$ and an element $\varphi$ in $H$, such that:
\begin{itemize}
	\item $\langle\varphi,\varphi\rangle=1$ and there exists $\lambda\in(0,\infty)$ with $\lambda\cov\mu(\varphi,h)= \langle \varphi,h\rangle$ for $h\in H$.
	\item $\langle\varphi_N,\varphi_N\rangle=1$ and there exists $\lambda_N\in(0,\infty)$ with $\lambda_N\cov{\mu_N}(\varphi_N,h)= \langle \varphi_N,h\rangle$ for $h\in H$.
	\item ${|\varphi-\varphi_N|}_H\to0$ as well as $\lambda_N\to\lambda$ for $N\to\infty$.
\end{itemize}

Now we cite some preliminaries on the disintegration of these Gaussian measures.
Under the above conditions $\mu$ is $\varphi$-quasi-invariant, i.e.~the image measure of $\mu$  under the shift $H\ni h$ $\mapsto$ $h+r\varphi$ $\in H$ is absolutely continuous w.r.t.~$\mu$ for every $r\in\R$ (see \cite[Prop.~5.5]{albeverio90}). 
By \cite[Prop.~4.2]{albeverio90}, $\mu$ and the measure
\begin{equation*}
	\sigma_\varphi(B):=\int_\R\mu(B-r\varphi)\de r,\quad B\in\mathcal B(H),
\end{equation*}
are absolutely continuous w.r.t.~each other and setting $\pi_\varphi h:=h-\langle h,\varphi\rangle \varphi$ for $h\in H$ we have 
\begin{equation}\label{eqn:mint}
	\int_Hu\de\mu=\int_{\pi_\varphi(H)}\int_\R u(h+r\varphi)\frac{\de\mu}{\de\sigma_\varphi}(h+r \varphi)\de r\de( \mu\circ\pi_\varphi^{-1})(h)
\end{equation}
for any measurable function $u:H\to[0,\infty)$.
\cite[Prop.~5.5]{albeverio90} provides $\frac{\de\mu}{\de\sigma_\varphi}$ in an explicit form, which in our case yields
\begin{align}\label{eqn:mint2}
	\frac{\de\mu}{\de\sigma_\varphi}(h)
	=\sqrt{\frac{\lambda}{2\pi}}\exp\Big(-\frac{\lambda\langle \varphi,h\rangle^2}{2}\Big),
	\quad h\in H.
\end{align}
From \eqref{eqn:mint} and \eqref{eqn:mint2} we obtain
\begin{equation}\label{eqn:dislimit}
	\int_Hu\de\mu=\sqrt{\frac{\lambda}{2\pi}}\int_{\pi_\varphi(H)}\int_\R u(h+r\varphi)\exp(-\tfrac{\lambda r^2}{2})\de r
	\de( \mu\circ\pi_\varphi^{-1})(h)
\end{equation}
for any measurable function $u:H\to[0,\infty)$.

Let $N\in\N$. We now define a rotation $J_N:H\to H$ which maps $\varphi_N$ onto $\varphi$ and then disintegrate the image of $\mu_N$ under $J_N$ in the same way as we did with $\mu$.
If $\varphi_N=\varphi$, we set $J_N:=$ id.
Otherwise, let $J_N:H\to H$ be the isometric isomorphism which leaves the orthogonal complement of span$(\{\varphi,\varphi_N\})$ invariant, while on span$(\{\varphi,\varphi_N\})$ its action is defined via the basis transformation
\begin{align*}
	\varphi_N&\quad\longmapsto\quad \varphi\\
	\frac{\varphi-\langle\varphi_N,\varphi\rangle\varphi_N}{\sqrt{1-\langle\varphi,\varphi_N\rangle^2}}&\quad\longmapsto\quad
	\frac{\langle\varphi,\varphi_N\rangle\varphi-\varphi_N}{\sqrt{1-\langle\varphi,\varphi_N\rangle^2}}.
\end{align*}
With $\tilde\mu_N:=\mu_N\circ J_N^{-1}$ it holds
\begin{equation*}
	\lambda_N\int_H\langle \varphi,k\rangle\langle h,k\rangle\de\tilde \mu_N(k)
	=\lambda_N\cov{\mu_N}(J_N^{-1}\varphi,J_N^{-1}h)=\langle \varphi_N,J_N^{-1}h\rangle
	=\langle \varphi,h\rangle
\end{equation*}
for $h\in H$. 
So, in analogy to the disintegration for $\mu$, an application of \cite[Prop.~4.2 \& Prop.~5.5]{albeverio90} yields that $\tilde\mu_N$ is $\varphi$-quasi-invariant and 
\begin{equation}\label{eqn:distilde}
	\int_Hu\de\tilde \mu_N=\sqrt{\frac{\lambda_N}{2\pi}}\int_{\pi_\varphi(H)}\int_\R u(h+r\varphi)\exp(-\tfrac{\lambda_N r^2}{2})\de r
	\de( \tilde \mu\circ\pi_\varphi^{-1})(h)
\end{equation}
for any measurable function $u:H\to[0,\infty)$.
\begin{remark}\label{rem:strOpN}
	For $N\in\N$ we set 
	\begin{equation*}
	v_N:=\frac{\varphi-\langle\varphi_N,\varphi\rangle\varphi_N}{\sqrt{1-\langle\varphi,\varphi_N\rangle^2}}\quad
	\text{and}\quad w_N:=\frac{\langle\varphi,\varphi_N\rangle\varphi-\varphi_N}{\sqrt{1-\langle\varphi,\varphi_N\rangle^2}}.
	\end{equation*}
	For $u\in H$ with $|u|_H=1$ it holds
	\begin{align*}
		|J_Nu-u|_H^2&=\big(\langle \varphi,J_Nu\rangle-\langle\varphi,u\rangle\big)^2+\big(\langle w_N,J_Nu\rangle-\langle w_N,u\rangle\big)^2
		\leq|\varphi_N-\varphi|_H^2+|v_N-w_N|_H^2\\
		&=|\varphi-\varphi_N|_H^2+
		\frac{\big|\varphi(1-\langle\varphi_N,\varphi\rangle)+\varphi_N(1-\langle\varphi_N,\varphi\rangle)\big|_H^2}
		{1-\langle\varphi,\varphi_N\rangle^2}\\
		&=|\varphi-\varphi_N|_H^2+
		\frac{2(1+\langle\varphi,\varphi_N\rangle)(1-\langle\varphi,\varphi_N\rangle)^2}{1-\langle\varphi,\varphi_N\rangle^2}\\
		&=|\varphi-\varphi_N|_H^2+2(1-\langle\varphi_N,\varphi\rangle).
	\end{align*}
	The right-hand side of this inequality converges to zero for $N\to\infty$ and is independent from $u$.
	Hence, the identity is the limit of ${(J_N)}_{N\in\N}$ w.r.t.~the operator norm on $H$.
\end{remark}

%

We want to tackle the question of Mosco convergence under the existence of perturbing densities for the Gaussian reference measure $\mu_N$, $N\in\N$, and $\mu$.
For a bounded function $f:\R\to\R$ the total variation is defined as
\begin{equation*}
	\textnormal{TV}(f):=\sup_{\substack{m\in\N\\x_1\leq\dots\leq x_m}}\sum_{i=1}^{m-1}|f(x_{i+1})-f(x_i)|\quad\in[0,\infty].
\end{equation*}
For a function $f:\R\to\R$ with $\textnormal{TV}(f)<\infty$ the Jordan decomposition, as derived in \cite[Chapter 5.2]{royden88}, states the existence of $f_1,f_2:\R\to\R$ such that
\begin{align}\label{eqn:tightOne:Jordan}
	f=f_1-f_2,\quad \|f_1\|_\infty+\|f_2\|_\infty\leq\|f\|_\infty+\|f\|_{\textnormal{TV}},\nonumber\\ f_1,f_2\text{ are monotone increasing.}
\end{align}

Until the end of Section \ref{sec:Mokoko} we fix bounded functions $g_N$ for $N\in\N$ and $g$ from $\R$ to $\R$ with finite total variation. Then we define
\begin{equation*}
	Z:=\int_H\exp\Big(\int_D g(h(z))\de z\Big)\de\mu,\quad\varrho(h):=\exp\Big(\int_D g(h(z))\de z\Big)/Z,\quad h\in H,
\end{equation*}
and 
\begin{equation*}
	Z_N:=\int_H\exp\Big(\int_D g_N(h(z))\de z\Big)\de\mu_N,\quad\varrho_N(h):=\exp\Big(\int_D g_N(h(z))\de z\Big)/Z_N,\quad h\in H. 
\end{equation*}
We need the following condition to be satisfied.
\begin{condition}\label{condi:last}
We assume that 
\begin{equation*}
	g=g^{(1)}-g^{(2)},\quad g_N=g_N^{(1)}-g_N^{(2)},
\end{equation*}
are the Jordan decompositions for $g$ , respectively for $g_N$ for each $N\in\N$, in the sense of \eqref{eqn:tightOne:Jordan}, and that the following conditions hold:
\begin{thm(ii)}
	\item $\displaystyle \sup_N\|g_N\|_\infty+\|g_N\|_{\textnormal{TV}}<\infty$.
	\item $\displaystyle\sup_{x\in\R}\inf_{\substack{r\\|r|\leq\delta}}|g^{(i)}(x)-g^{(i)}_N(x+r)|\overset{N\to\infty}{\longrightarrow}0$\quad for every $\delta\in(0,\infty)$ and $i=1,2$.
\end{thm(ii)}
\end{condition}

Let 
\begin{equation*}
	\mathcal F C_b^\infty:=\Big\{H\ni h\mapsto g(\langle\,\cdot\,, f_1\rangle,\dots,\langle\,\cdot\,, f_m\rangle)
	\,\Big|\,m\in\N,\,g\in C_b^\infty(\R^m),\, f_1,\dots,f_m\in H\Big\}.
\end{equation*}
As stated in \cite[Prop.~5.5]{albeverio90}, $\varphi_N$ is admissible for $\mu_N$ in the sense of \cite[Def.~1.1]{sheng92}. Hence, the partial derivative $u\mapsto\partial u/\partial\varphi_N$
with domain
$\{u\in L^2(H,\varrho_N\mu_N)$ $|$ $u$ has a representative from $\mathcal FC_b^\infty(H)\}$
is a well-defined, closable linear operator on $L^2(H,\varrho_N\mu_N)$.
We are interested in the asymptotic behaviour for $N\to\infty$ of the symmetric Dirichlet form $(\E^N_{\varphi_N},\dom(\E^N_{\varphi_N}))$ on $L^2(H,\varrho_N\mu_N)$ which
is the closure of the pre-Dirichlet form
\begin{equation*}
	\E^N_{\varphi_N}(u,v)=\int_H\frac{\partial u}{\partial \varphi_N}\frac{\partial v}{\partial\varphi_N}\varrho_N\de\mu_N
\end{equation*}
with domain $\{u\in L^2(H,\varrho_N\mu_N)$ $|$ $u$ has a representative from 
$\mathcal FC_b^\infty(H)\}$.

It is asymptotically irrelevant if we consider the image form of $(\E^N_{\varphi_N},\dom(\E^N_{\varphi_N}))$ under $J_N$ instead (see Lemma \ref{lem:tentwelve} (ii) and Proposition \ref{prop:coMo} below). This trick allows us to disintegrate the involved reference measures, which vary with $N$, along the same direction $\varphi$ over the base $\pi_\varphi(H)$.
 Using a disintegration of the measures  analogue to \eqref{eqn:distilde}, where instead of $\tilde\mu_N=\mu_N\circ J_N^{-1}$ we consider the disturbed image measures
\begin{equation*}
	(\varrho_N\mu_N)\circ J_N^{-1}=(\underbrace{\varrho_N\circ J_N^{-1}}_{=:\tilde\varrho_N})\tilde\mu_N,
\end{equation*}
we can apply a result form \cite{me21}. 
The necessary preliminaries are summarized in the next remark. Let
\begin{equation*}
	\mathcal N_N(\de r):=\sqrt{\tfrac{\lambda_N}{2\pi}}\exp(-\tfrac{\lambda_N r^2}{2})\de r
\end{equation*}
be the normal distribution on $\R$ with mean zero and variance $1/\lambda_N$. Further, let
\begin{equation*}
\nu_N:=(\tilde\varrho\tilde\mu_N)\circ {\pi_\varphi}^{-1}.
\end{equation*}
The Radon-Nikodym density $\de\nu_N/(\tilde\mu_N\circ\pi_\varphi^{-1})$ coincides for  $(\tilde\mu_N\circ\pi_\varphi^{-1})$ -a.e.~$h\in \pi_\varphi(H)$ with $\int_\R\tilde\varrho_N(h+t\varphi)\de \mathcal N_N(t)$, as follows from \eqref{eqn:distilde}.
We set
\begin{equation*}
	\rho_N(h,r):=\frac{\tilde\varrho_N(h+r\varphi)}{\int_\R\tilde\varrho_N(h+t\varphi)\de \mathcal N_N (t)}
\end{equation*}
for $h\in\pi_\varphi(H)$ and $r\in\R$.

\begin{remark}
	\begin{thm(vi)}
	\item Using \eqref{eqn:distilde} we obtain
	\begin{align*}
		\int_Hu\tilde\varrho_N\de\tilde \mu_N&=\int_{\pi_\varphi(H)}\int_\R u(h+r\varphi)
		\tilde\varrho_N(h+r\varphi)\de \mathcal N_N(r)\de( \tilde \mu\circ\pi_\varphi^{-1})(h)
		\\&=\int_{\pi_\varphi(H)}\int_\R u(h+r\varphi)
		\rho_N(h,r)\de \mathcal N_N(r)\de\nu_N(h).
	\end{align*}
	Hence,
	\begin{align*}
		 I_N:L^2(H,\tilde\varrho_N\tilde \mu_N)&\overset{\simeq}{\longrightarrow} L^2\big(\pi_\varphi(H)\times\R,\rho_N(\nu_N\times \mathcal N_N)\big)\\
		 I_Nu(h,r)&:=u(h+r\varphi)\quad(\nu_N\times \mathcal N_N)\text{ -a.e.}
	\end{align*}
	is an isometric isomorphism.
	In analogy to \cite[Propo.~1.2]{sheng92} we define
	\begin{align*}
		\dom(\tilde \E_\varphi^N):=\Big\{u\in L^2(H,\tilde\varrho_N\tilde\mu_N)\,\Big|\,
		&\text{for }\nu_N\text{-a.e.~}h\in \pi_\varphi(H)\\
		&I_Nu(h,\cdot)\text{ has an absolutely continuous version }f_h\\
		&\text{and }\int_{\pi_\varphi(H)}\int_\R|f_h'(r)|^2\rho_N(h,r)\de \mathcal N_N(r)\de \nu_N(h)<\infty\Big\}
	\end{align*}
	and in this sense 
	\begin{equation*}
		\tilde \E_\varphi^N(u,v):=\int_{\pi_\varphi(H)}\int_\R  u(h+\,\cdot\;\varphi)'(r)\,
		v'(h+\,\cdot\;\varphi)'(r)\rho_N(h,r)\de \mathcal N_N(r)\de \nu_N(h)
	\end{equation*}
	for $u,v\in \dom(\tilde \E_\varphi^N)$.
	By \cite[Thm.~5.1]{sheng92}, the space $\{u\in L^2(H,\tilde\varrho_N\tilde\mu_N)$ $|$ $u$ has a representative from $\mathcal FC_b^\infty(H)\}$ is a form core for $(\tilde \E_\varphi^N,\dom(\tilde \E_\varphi^N))$.
	\item Analogously, we obtain via \eqref{eqn:dislimit} the according disintegration formula for the disturbed measure $\varrho\mu$. Let
	\begin{equation*}
		\mathcal N(\de r):=\sqrt{\tfrac{\lambda}{2\pi}}\exp(-\tfrac{\lambda r^2}{2})\de r
	\end{equation*}
	be the normal distribution on $\R$ with mean zero and variance $1/\lambda$. We set
		$\nu:=(\varrho\mu)\circ {\pi_\varphi}^{-1}$
	and
	\begin{equation*}
		\rho(h,r):=\frac{\varrho(h+r\varphi)}{\int_\R\varrho(h+t\varphi)\de \mathcal N (t)}
	\end{equation*}
	for $h\in \pi_\varphi(H)$, $r\in\R$. It holds
	\begin{equation*}
		\int_Hu\varrho\de \mu=\int_{\pi_\varphi(H)}\int_\R u(h+r\varphi)
		\rho(h,r)\de \mathcal N(r)\de\nu(h)
	\end{equation*}
	and the map
	\begin{align*}
		I:L^2(H,\varrho \mu)&\overset{\simeq}{\longrightarrow} L^2\big(\pi_\varphi(H)\times\R,\rho(\nu\times \mathcal N)\big)\\
		Iu(h,r)&:=u(h+r\varphi)\quad(\nu\times \mathcal N)\text{ -a.e.}
	\end{align*}
	is an isometric isomorphism.
	\item In analogy to \cite[Propo.~1.2]{sheng92} let
	\begin{align*}
		\dom(\E_\varphi):=\Big\{u\in L^2(H,\varrho\mu)\,\Big|\,
		&\text{for }\nu_N\text{-a.e.~}h\in \pi_\varphi(H)\\
		&Iu(h,\cdot)\text{ has an absolutely continuous version }f_h\\
		&\text{and }\int_{\pi_\varphi(H)}\int_\R|f_h'(r)|^2\rho(h,r)\de \mathcal N(r)\de \nu(h)<\infty\Big\}
	\end{align*}
	and in this sense 
	\begin{equation*}
		\E_\varphi(u,v):=\int_{\pi_\varphi(H)}\int_\R  u(h+\,\cdot\;\varphi)'(r)\,
		v'(h+\,\cdot\;\varphi)'(r)\rho(h,r)\de \mathcal N(r)\de \nu(h)
	\end{equation*}
	for $u,v\in \dom( \E_\varphi)$.
	By \cite[Thm.~5.1]{sheng92}, the space $\{u\in L^2(H,\varrho\mu)$ $|$ $u$ has a representative from $\mathcal FC_b^\infty(H)\}$ is a form core for $( \E_\varphi,\dom( \E_\varphi))$.
\end{thm(vi)}
\end{remark}

\subsection{Convergence of component forms}\label{sec:mococo}

Here we show the Mosco-Kuwae-Shioya convergence of the component forms defined above, with the help of \cite{me21}.
The following Lemma \ref{lem:condcons} refers to the case without perturbing density, i.e.~$g_N(\cdot)=g(\cdot)=0$. 
Together with the subsequent Lemma \ref{lem:tentwelve}, these form the preliminaries for the main result of Section \ref{sec:mococo}, Proposition \ref{prop:coMo} below.

\begin{lemma}\label{lem:condcons}
	On the Borel $\sigma$-algebra of $\pi_\varphi(H)\times\R$ we consider the product measures
	\begin{equation*}
		\big(\tilde \mu_N\circ \pi_\varphi^{-1}\big)\times \mathcal N_N
	\end{equation*}
	for $N\in\N$ and their weak limit $\big(\mu\circ \pi_\varphi^{-1}\big)\times \mathcal N$.
	The criteria of \cite[Cond.~3.8]{me21} are satisfied.
\end{lemma}
\begin{proof}
	Let $f\in C_b(\pi_\varphi(H)\times\R)$. To verify item (i) of \cite[Cond.~3.8]{me21} it suffices to argue that
	\begin{equation*}
		\Big|\int_\R f(h,r)\de \mathcal N_N(r)\Big|
		\overset{N\to\infty}{\longrightarrow}\Big|\int_\R f(h,r)\de \mathcal N(r)\Big|
	\end{equation*}
	uniformly w.r.t.~$h\in\pi_\varphi(H)$. This simply follows from the estimate
	\begin{multline*}
		\Bigg|\sqrt{\tfrac{\lambda_N}{2\pi}}\,\Big|\int_\R f(h,r)\exp(-\tfrac{\lambda_N r^2}{2})\de r\Big|
		-\sqrt{\tfrac{\lambda}{2\pi}}\,\Big|\int_\R f(h,r)\exp(-\tfrac{\lambda r^2}{2})\de r\Big|\,\Bigg|\\
		\leq\|f\|_\infty \int_\R\Big|\sqrt{\tfrac{\lambda_N}{2\pi}}\exp(-\tfrac{\lambda_N r^2}{2})-
		\sqrt{\tfrac{\lambda}{2\pi}}\exp(-\tfrac{\lambda r^2}{2})\Big|\de r
	\end{multline*}
	and an application of Lebesgue's dominated convergence, since $\lambda_N\to\lambda$ for $N\to\infty$.
	A dominating integrable function for the integrand of the right-hand side is given by
		$r\mapsto 2\sqrt{\tfrac{\lambda}{\pi}}\exp(-\tfrac{\lambda r^2}{4})$
	if $N\in\N$ is large enough such that $\lambda/2<\lambda_N<2\lambda$.

We now address (ii) of \cite[Cond.~3.8]{me21}. Let $\kappa:\R\to[0,1]$ be continuous with compact support.
Further, let $R_{1/l}^{\varphi,\eta}(f)$ and
$I_{1/l}^{\varphi,\eta}(f)$, as defined in \cite[Sect.~2.2]{me21}, be the $\varphi,\eta$-residual and the $\varphi,\eta$-perturbation
of a measurable, non-negative function $f:\R\to\R$ for the parameters $l\in\N$ and $\varphi,\eta:\R\to[0,1]$, which are primal functions as defined in the beginning of \cite[Sect.~2.1]{me21}.
In analogy to \cite{me21} we set $\varphi_{1/l}^\alpha(r)=\varphi(lr-l\alpha)$ for $\alpha\in (1/l)\Z$, $r\in\R$.
The Radon-Nikodym density of $\mathcal N_N$ and $\mathcal N$ w.r.t.~the Lebesgue measure on $\R$ is denoted by $\tfrac{\de \mathcal N_N}{\de r}$, respectively $\tfrac{\de \mathcal N}{\de r}$.
	For $\de r$-a.e.~$r\in\R$ it holds
	\begin{align}\label{eqn:pizza}
		R_{1/l}^{\varphi,\eta}\big(\kappa\,\tfrac{\de \mathcal N_N}{\de r}\big)(r)
		&=\sqrt{\tfrac{\lambda_N}{2\pi}}
		\sum_{\alpha\in\frac{1}{l}\Z}l\int_\R\big|\kappa(t)\exp(-\tfrac{\lambda_N t^2}{2})-\kappa(r)\exp(-\tfrac{\lambda_N r^2}{2})\big|\eta_{1/l}^\alpha(t)\de t\varphi_{1/l}^\alpha(r)
		\nonumber\\&\leq\sqrt{\tfrac{\lambda_N}{2\pi}}\sup_{\substack{t\in\R\\|t-r|\leq4/l}}
		\Big|\kappa(r)\exp(-\tfrac{\lambda_N r^2}{2})-\kappa(t)\exp(-\tfrac{\lambda_N t^2}{2})\Big|
		\nonumber\\&\leq 
		\sqrt{\tfrac{\lambda_N}{2\pi}}\sup_{\substack{t\in\R\\|t-r|\leq4/l}}(\kappa(r)+\kappa(t))
		\Big|\exp(-\tfrac{\lambda_N r^2}{2})-\exp(-\tfrac{\lambda_N t^2}{2})\Big|\nonumber\\
		&\phantom{\leq\,}+\sqrt{\tfrac{\lambda_N}{2\pi}}\sup_{\substack{t\in\R\\|t-r|\leq4/l}}|\kappa(r)-\kappa(t)|
		\Big(\exp(-\tfrac{\lambda_N r^2}{2})+\exp(-\tfrac{\lambda_N t^2}{2})\Big)
	\end{align}
	and also 
	\begin{align}\label{eqn:pizza2}
		I_{1/l}^{\varphi,\eta}\big(\kappa\,\tfrac{\de \mathcal N_N}{\de r}\big)(r)
		&=\sqrt{\tfrac{\lambda_N}{2\pi}}
		\sum_{\alpha\in\frac{1}{l}\Z}l\int_\R\kappa(t)\exp(-\tfrac{\lambda_N t^2}{2})\eta_{1/l}^\alpha(t)\de t\varphi_{1/l}^\alpha(r)
		\nonumber\\&\leq\sqrt{\tfrac{\lambda_N}{2\pi}}\,\exp(-\tfrac{\lambda_N r^2}{2})\sup_{\substack{t\in\R\\|t-r|\leq4/l}}
		\Big(\kappa(t)\exp\big(-\tfrac{\lambda_N (t^2-r^2)}{2}\big)\Big).
	\end{align}
	Due to \eqref{eqn:pizza}, for any bounded, measurable function $f:\R\to\R$ it holds
	\begin{align*}
		\frac{\int_\R f(r)R_{1/l}^{\varphi,\eta}\big(\kappa\,\tfrac{\de \mathcal N_N}{\de r}\big)(r)\de r}{\int_\R f^2\de \mathcal N_N}
		&\leq \int_\R\Big(R_{1/l}^{\varphi,\eta}\big(\kappa\,\tfrac{\de \mathcal N_N}{\de r}\big)(r)\Big)^2\big(\tfrac{\de \mathcal N_N}{\de r}(r)\big)^{-1}\de r\nonumber\\
		&\leq \sup_{\substack{t,r\in\R\\|t-r|\leq4/l}}(\kappa(r)+\kappa(t))
		\Big|1-\exp(-\tfrac{\lambda_N (t^2-r^2)}{2})\Big|\nonumber\\
		&\phantom{\leq\,}+\sup_{\substack{t,r\in\R\\|t-s|\leq4/l}}(\kappa(r)-\kappa(t))
		\Big|1+\exp(-\tfrac{\lambda_N (t^2-r^2)}{2})\Big|.
	\end{align*}
    The right-hand side is independent of the choices for $\eta$, $\varphi$ and $f$ and converges to zero for $l\to\infty$ uniformly w.r.t.~$N\in\N$, because ${(\lambda_N)}_{N\in\N}$ is bounded and $\kappa$ continuous with compact support. This means, for $\delta^\kappa_{1/l}(\mathcal N_N)$ defined as in \cite[Sec.~2.2]{me21}, it holds
    $\lim_{l\to\infty}\sup_{N\in\N}\delta^\kappa_{1/l}(\mathcal N_N)=0$.
    
    Due to \eqref{eqn:pizza2}, for any bounded, measurable function $f:\R\to\R$ it holds
	\begin{align*}
		\frac{\int_\R f(r)I_{1/l}^{\varphi,\eta}\big(\kappa\,\tfrac{\de \mathcal N_N}{\de r}\big)(r)\de r}{\int_\R |f|\de \mathcal N_N}
		&\leq \Big\|I_{1/l}^{\varphi,\eta}\big(\kappa\,\tfrac{\de \mathcal N_N}{\de r}\big)\big(\tfrac{\de \mathcal N_N}{\de r}\big)^{-1}\Big\|_{L^\infty(\R,m_N)}\nonumber\\
		&\leq\sup_{\substack{t,r\in\R\\|t-r|\leq4}}
		\Big(\kappa(t)\exp\big(-\tfrac{\lambda_N (t^2-r^2)}{2}\big)\Big).
	\end{align*}
	The value of the right-hand side is independent of $\eta$, $\varphi$, $f$ and $l\in\N$ and bounded over all $N\in\N$. This means, for $C^\kappa_{1/l}(\mathcal N_N)$ defined as in \cite[Sec.~2.2]{me21}, it holds
	$\sup_{l\in \N}\sup_{N\in\N}C^\kappa_{1/l}(\mathcal N_N)<\infty$. This concludes the proof.
\end{proof}

\begin{lemma}\label{lem:tentwelve}
	Under Condition \ref{condi:last}, it holds:
	\begin{thm(vi)}
		\item $\varrho_N\mu_N\Rightarrow \varrho\mu$ for $N\to\infty$ in the sense of weak measure convergence on $H$.
		\item $\tilde\varrho_N\tilde\mu_N\Rightarrow \varrho\mu$ for $N\to\infty$ in the sense of weak measure convergence on $H$.
		\item Let $f\in C_b(\pi_\varphi(H)\times\R)$.
		\begin{equation}\label{eqn:target3}
			\int_\R f(\,\cdot\,,r)\rho_N(\,\cdot\,,r)\de \mathcal N_N(r)
			\overset{N\to\infty}{\longrightarrow}\int_\R f(\,\cdot\,,r)\rho(\,\cdot\,,r)\de \mathcal N(r)
		\end{equation}
		strongly in the sense of Kuwae-Shioya (\cite{kuwae03}), within the framework of the converging Hilbert spaces $L^2(\pi_\varphi(H),\nu_N)$ $\to$ $L^2(\pi_\varphi(H),\nu)$.
	\end{thm(vi)}
\end{lemma}
\begin{proof}
	(i) We first show the claim for the case $g_N^{(2)}(\cdot)=g^{(2)}(\cdot)=0$, i.e.~$g_N=g_N^{(1)}$ and $g=g^{(1)}$.
	The idea is to apply \cite[Lemma 3.5]{me21} proving that $\varrho$ is the strong limit of ${(\varrho_N)}_{N\in\N}$
	in the sense of Kuwae-Shioya within the framework of the converging Hilbert spaces $L^2(H,\mu_N)$ $\to$ $L^2(H,\mu)$. 
	In the following we use Lemma \ref{lem:tightOne:majmin} for the sequence ${(g_N)}_N$ and $g$  adopting its notation. 
	To check the assumptions of \cite[Lemma 3.5]{me21} in this case, for $N,m\in\N$ we choose
	\begin{equation*}
		G_{N,m}^{\textnormal{min}}(h):=\exp\Big({\int_{D}  g_{N,m}^{\textnormal{min}}\circ h \de z}\Big),\quad h\in H,
	\end{equation*}
	as a minorante and 
	\begin{equation*}
		G_{N,m}^{\textnormal{maj}}(h):=\exp\Big({\int_{D}  g_{N,m}^{\textnormal{maj}}\circ h \de z}\Big),\quad h\in H,
	\end{equation*}
	as a majorante for the function $G_N(h):=$ $Z_N\varrho_N$ $=\exp\Big({\int_{D}  g_N\circ h \de z}\Big)$, $h\in H$.
	
	The set of discontinuities of $g$ is a countable subset of $\R$, which we denote by $U_g$. It holds
	\begin{equation}\label{eqn:zerodis}
		\de z(\{z\in D\,|\,h(z)\in U_g\})\leq\sum_{a\in U_g}\de z(\{z\in D\,|\,h(z)=a\})=0\quad\de \mu(h)\text{-a.e.~on }H,
	\end{equation}
	since $\mu$ is a non-degenerate, centred Gaussian measure (see Lemma \ref{lem:LVset}).
	In view of \eqref{eqn:zerodis}, we have
	\begin{equation*}
		\lim_{m\to\infty}\exp\Big({\int_{D}  g_m^{\textnormal{min}}\circ h \de z}\Big)=\exp\Big({\int_{D}  g\circ h \de z}\Big)\quad \de \mu(h)\text{-a.e.~on }H,
	\end{equation*}
	and likewise
	\begin{equation*}
		\lim_{m\to\infty}\exp\Big({\int_{D}  g_m^{\textnormal{maj}}\circ h \de z}\Big)=\exp\Big({\int_{D}  g\circ h \de z}\Big)\quad  \de \mu(h)\text{-a.e.~on }H,
	\end{equation*}
	by Lebesgue's dominated convergence. Another use of Lebesgue's dominated convergence yields
	\begin{equation}\label{eqn:tightOne:13}
		\lim_{m\to\infty}\int_H\Big|\exp\Big({\int_{D}  g_m^{\textnormal{min}}\circ h \de z}\Big)-\exp\Big({\int_{D}  g\circ h \de z}\Big)\Big|^2\de\mu(h)=0,
	\end{equation}
	respectively
	\begin{equation}\label{eqn:tightOne:14}
		\lim_{m\to\infty}\int_H\Big|\exp\Big({\int_{D}  g_m^{\textnormal{maj}}\circ h \de z}\Big)-\exp\Big({\int_{D}  g\circ h \de z}\Big)\Big|^2\de\mu(h)=0.
	\end{equation}
	Moreover, Lemma \ref{lem:tightOne:majmin} provides
	\begin{equation}\label{eqn:kowloon1}
		\sup_{h\in H}\Big|G_{N,m}^{\textnormal{min}}(h)
		-\exp\Big({\int_{D}  g_m^{\textnormal{min}}\circ h \de z}\Big)\Big|\overset{N\to\infty}{\longrightarrow}0,
	\end{equation}
	as well as
	\begin{equation}\label{eqn:kowloon2}
		\sup_{h\in H}\Big|G_{N,m}^{\textnormal{maj}}(h)
		-\exp\Big({\int_{D}  g_m^{\textnormal{maj}}\circ h \de z}\Big)\Big|\overset{N\to\infty}{\longrightarrow}0
	\end{equation}
	for fixed $m\in\N$.
	The convergence of \eqref{eqn:kowloon1} and \eqref{eqn:kowloon2} is stronger than the (strong) convergence w.r.t.~the Kuwae-Shioya topology within the framework of converging Hilbert spaces $L^2(H,\mu_N)$ $\to$ $L^2(H,\mu)$. Therefore, \eqref{eqn:kowloon1}, \eqref{eqn:kowloon2}, \eqref{eqn:tightOne:13} and \eqref{eqn:tightOne:14} together with \cite[Lemma 3.5]{me21} imply that
	$\exp\Big({\int_{D}  g\circ h \de z}\Big)$, $h\in H$, is the strong limit of ${(G_N)}_{N\in \N}$ in sense of Kuwae-Shioya. This in particular implies $Z_N\to Z$ and therefore also $\varrho_N\to\varrho$ in sense of Kuwae-Shioya. 
	
	The general case, in which
	\begin{equation*}
		g=g^{(1)}-g^{(2)},\quad g_N=g_N^{(1)}-g_N^{(2)},
	\end{equation*}
	 can be handled analogously. The assumptions of \cite[Lemma 3.5]{me21} are verified via
	Lemma \ref{lem:tightOne:majmin} through the same arguments as above. In this case, for $N,m\in\N$ we choose
	\begin{equation*}
		G_{N,m}^{\textnormal{min}}(h):=\exp\Big({\int_{D}  g_{N,m}^{(1)\textnormal{min}}\circ h \de z}\Big)
		\exp\Big(-{\int_{D}  g_{N,m}^{(2)\textnormal{max}}\circ h \de z}\Big),\quad h\in H,
	\end{equation*}
	as a minorante and 
	\begin{equation*}
		G_{N,m}^{\textnormal{maj}}(h):=\exp\Big({\int_{D}  g_{N,m}^{(1)\textnormal{max}}\circ h \de z}\Big)
		\exp\Big(-{\int_{D}  g_{N,m}^{(2)\textnormal{min}}\circ h \de z}\Big),\quad h\in H,
	\end{equation*}
	as a majorante for the function $G_N(h):=$ $Z_N\varrho_N$ $=\exp\Big({\int_{D}  g_N \de z}\Big)$, $h\in H$.
	The arguments are analogous. This concludes the proof of (i).
	
	(ii) As $\tilde\varrho_N\tilde\mu_N$ $=$ $(\varrho_N\mu_N)\circ J_N^{-1}$ for $N\in\N$, the statement of (ii) is an immediate consequence of
	\begin{equation*}
		\sup_{\substack{h\in H\\|h|_H=1}}|J_Nh-h|_H\overset{N\to\infty}{\longrightarrow}0,
	\end{equation*}
	which is shown in Remark \ref{rem:strOpN}.
	
	(iii) The proof of (iii) is very similar to that in part (i). 
	Again, we apply \cite[Lemma 3.5]{me21}.
	Here, we additionally assume that $f$ is non-negative in the first step.
	We adopt the  terminology from part (i) of this proof.
	As a suitable minorante to prove the convergence of \eqref{eqn:target3} we choose the functions
	\begin{equation*}
		\pi_\varphi(H)\ni h\;\;\longmapsto\;\;F_{N,m}^\textnormal{min}(h):=\frac{\int_\R f(h+r\varphi)\big(G_{N,m}^\textnormal{min}\circ J_N^{-1}\big){(h+r\varphi)}\de \mathcal N_N(r)}
		{\int_\R\big(G_{N,m}^\textnormal{maj}\circ J_N^{-1}\big){(h+t\varphi)}\de \mathcal N_N(t)}
	\end{equation*} 
	for $N,m\in\N$.
	A majorante is provided by
	\begin{equation*}
		\pi_\varphi(H)\ni h\;\;\longmapsto\;\;F_{N,m}^\textnormal{maj}(h):=\frac{\int_\R f(h+r\varphi)\big(G_{N,m}^\textnormal{maj}\circ J_N^{-1}\big){(h+r\varphi)}\de \mathcal N_N(r)}
		{\int_\R\big(G_{N,m}^\textnormal{min}\circ J_N^{-1}\big){(h+t\varphi)}\de \mathcal N_N(t)}
	\end{equation*}
	for $N,m\in\N$. Indeed, we have
	\begin{equation*}
		F_{N,m}^\textnormal{min}(h)\leq \frac{\int_\R f(h+r\varphi)\big(G_N\circ J_N^{-1}\big){(h+r\varphi)}\de \mathcal N_N(r)}
		{\int_\R\big(G_N\circ J_N^{-1}\big){(h+t\varphi)}\de \mathcal N_N(t)}=\int_\R f(h,r)\rho_N(h,r)\de \mathcal N_N(r)
	\end{equation*}
	and also
	\begin{equation*}
		F_{N,m}^\textnormal{maj}(h)\geq \frac{\int_\R f(h+r\varphi)\big(G_N\circ J_N^{-1}\big){(h+r\varphi)}\de \mathcal N_N(r)}
		{\int_\R\big(G_N\circ J_N^{-1}\big){(h+t\varphi)}\de \mathcal N_N(t)}=\int_\R f(h,r)\rho_N(h,r)\de \mathcal N_N(r)
	\end{equation*}
	for $h\in\pi_\varphi(H)$, $N.m\in\N$. The other assumptions of \cite[Lemma 3.5]{me21} are also easily checked with similar arguments as in the proof for (i). For $m\in\N$ we have the uniform convergence of ${(F_{N,m}^\textnormal{min})}_{N\in\N}$ towards the function
	\begin{equation*}
		\pi_\varphi(H)\ni h\;\;\longmapsto\;\;F_{m}^\textnormal{min}(h):=\frac{\int_\R f(h+r\varphi)\exp\Big(\int_D g_m^\textnormal{min}{(h+r\varphi)}\Big)\de \mathcal N(r)}
		{\int_\R\exp\Big(\int_D g_m^\textnormal{maj}{(h+t\varphi)}\Big)\de \mathcal N(t)}
	\end{equation*}
	as well as the uniform convergence of ${(F_{N,m}^\textnormal{maj})}_{N\in\N}$ towards the function
	\begin{equation*}
		\pi_\varphi(H)\ni h\;\;\longmapsto\;\;F_{m}^\textnormal{min}(h):=\frac{\int_\R f(h+r\varphi)\exp\Big(\int_D g_m^\textnormal{maj}{(h+r\varphi)}\Big)\de \mathcal N(r)}
		{\int_\R\exp\Big(\int_D g_m^\textnormal{min}{(h+t\varphi)}\Big)\de \mathcal N(t)}.
	\end{equation*}
	With Lebesgue's dominated converges we conclude that 
	\begin{equation*}
	\pi_\varphi(H)\ni h\;\;\longmapsto\;\;\int_\R f(h,r)\rho(h,r)\de \mathcal N(r)=
	\frac{\int_\R f(h+r\varphi)\exp\Big(\int_D g{(h+r\varphi)}\Big)\de \mathcal N(r)}
	{\int_\R\exp\Big(\int_D g{(h+t\varphi)}\Big)\de \mathcal N(t)}
	\end{equation*}
	is the strong limit of ${(F_{m}^\textnormal{min})}_{m\in\N}$ and also the strong limit of ${(F_{m}^\textnormal{maj})}_{m\in\N}$ in $L^2(\pi_\varphi(H),\de\nu)$ as $m\to\infty$.
	 By linearity of the claimed statement in (iii), the assumption $f(\,\cdot\,)\geq 0$ can de dropped. This concludes the proof.
\end{proof}

\begin{prop}\label{prop:coMo}
	$(\E_{\varphi},\dom(\E_\varphi))$ is the limit of ${(\E^N_{\varphi_N},\dom(\E^N_{\varphi_N}))}_{N\in\N}$
	in the sense of Mosco-Kuwae-Shioya within the framework of the converging Hilbert spaces $L^2(H,\varrho_N\mu_N)$
	$\to$ $L^2(H,\varrho\mu)$.
\end{prop}
\begin{proof}
	The claim of the proposition is equivalent to the statement that 
	$(\E_{\varphi},\dom(\E_\varphi))$ is the limit of ${(\tilde\E^N_{\varphi},\dom(\tilde\E^N_{\varphi}))}_{N\in\N}$
	in the sense of Mosco-Kuwae-Shioya within the framework of the converging Hilbert spaces $L^2(H,\tilde\varrho_N\tilde\mu_N)$
	$\to$ $L^2(H, \varrho \mu)$. 
	
	The argument for this equivalence is the following. The map $f\mapsto f\circ J_N^{-1}$ is a bijection from $C_b^1(H)$ onto itself with
	\begin{equation*}
		\frac{\partial (f\circ J_N^{-1})}{\partial\varphi}(J_Nh)=\frac{\partial f}{\partial\varphi_N}(h),\quad h\in H.
	\end{equation*}
	 So, on the one hand, under the isometric isomorphism
    \begin{equation*}
    	L^2(H,\varrho\mu_N) \to L^2(H,\tilde\varrho_N\tilde \mu_N),\quad
    	u\mapsto u\circ J_N^{-1},
    \end{equation*}
    we have
    \begin{equation*}
    	u\in \dom (\E_{\varphi_N}^N)\text{ if and only }u\circ J_N^{-1}\in\dom(\tilde\E_\varphi^N),\quad
    	\E_{\varphi_N}^N(u,u)=\tilde\E_\varphi^N(u\circ J_N^{-1},u\circ J_N^{-1}).
    \end{equation*}
    On the other hand, ${(J_N)}_{N\in\N}$ converges to the identity w.r.t.~the operator norm on $H$. For
    $u_N\in L^2(H,\varrho_N\mu_N)$ for $N\in\N$ and $u\in L^2(H,\varrho\mu)$, the strong convergence of ${(u_N)}_N$ towards $u$ in the  framework of $L^2(H,\varrho_N\mu_N)$
    $\to$ $L^2(H,\varrho\mu)$ is equivalent to the strong convergence of ${(u_N\circ J_N^{-1})}_N$ towards $u$ in the framework of $L^2(H,\tilde\varrho_N\tilde\mu_N)$
    $\to$ $L^2(H,\varrho\mu)$. The same holds regarding weak convergence.
	
	To prove Mosco convergence of ${(\tilde\E^N_{\varphi},\dom(\tilde\E^N_{\varphi}))}_{N\in\N}$ towards $(\E_{\varphi},\dom(\E_\varphi))$, we apply \cite[Thm.~3.11]{me21}.
	The only thing left to do is the verification of \cite[Cond.~3.8]{me21}. The case without disturbing density, i.e.~$\varrho_N(\cdot)=\varrho(\cdot)=1$, is settled in Lemma \ref{lem:condcons}. In the general case, 
	\cite[Cond.~3.8 (i)]{me21} is a direct consequence of Lemma \ref{lem:tentwelve} (iii), because the strong convergence in the sense of Kuwae-Shioya implies the convergence of respective $L^2$-norms. \cite[Cond.~3.8 (ii)]{me21} follows from the perturbation result stated in \cite[Lemma 3.12]{me21}.
	This concludes the proof.
\end{proof}
	
	\subsection{Convergence of Gradient forms}
In the final part of this article, we define analyse the Mosco-Kuwae-Shioya convergence of gradient forms on $H$.
We use the results of Sections \ref{sec:comp} \& \ref{sec:mococo}.

The covariance $\cov\mu(h_1,h_1)$, $ h_1,h_2\in H$,
defines an inner product on $H$ which makes $(H,\cov\mu(\cdot,\cdot))$ a pre-Hilbert space. Taking the abstract completion of $H$ w.r.t.~the norm induced by $\cov\mu$ we construct a new Hilbert space $(H_\mu,\langle\cdot,\cdot\rangle_\mu)$, in which $(H,\langle\,\cdot,\,\cdot\,\rangle)$ is densely and continuously included. The self-adjoint generator $(A,\dom(A))$ of the symmetric closed form $\langle\cdot,\cdot\rangle$ with domain $H$ on $(H_\mu,\langle\cdot,\cdot\rangle_\mu)$ is characterized by
\begin{equation*}
	H=\dom(A^{\frac{1}{2}}),\quad \langle Au,v\rangle_\mu=\langle u,v\rangle\quad \text{for }u\in \dom(A),\, v\in H.
\end{equation*}
The spectrum of $A$ is a pure point spectrum which consists of real, positive eigenvalues. We refer to the inverse $A^{-1}$, which in fact can be shown to be a trace class operator from $H$ into $H_\mu$, as the covariance operator of $\mu$.

Analogous structures exist regarding the Gaussian measure $\mu_N$ for $N\in\N$, for which we do not assume the full topological support on $H$.
The topological support $V_N:=\supp{\mu_N}$ is a linear subspace of $H$. It shall be noted at this point, that
all of the three cases are possible in this setting: 
\begin{itemize}
\item $V_N$ may be finite-dimensional subspace.
\item  $V_N$ may be a infinite-dimensional, strict subspace of $H$
\item $V_N=H$.
\end{itemize}
In any of these cases the weak convergence of ${(\mu_N)}_{N\in\N}$ towards $\mu$ implies
\begin{equation}\label{eqn:V2H}
	\lim_{N\to\infty}|p_Nh-h|_H=0,\quad h\in H,
\end{equation}
where $p_N$ denotes the orthogonal projection onto $V_N$.
The covariance of $\mu_N$ defines an inner product on $V_N$  which makes $(V_N,\cov{\mu_N}(\cdot,\cdot))$ a pre-Hilbert space. With the same completion procedure as above we obtain a Hilbert space $(V_{\mu_N},\langle\cdot,\cdot\rangle_{\mu_N})$, in which $(V_N,\langle\,\cdot,\,\cdot\,\rangle)$ is densely and continuously included. 
The self-adjoint generator $(A_N,\dom(A_N))$ of the symmetric closed form $\langle\cdot,\cdot\rangle$ with domain $V_N$ on $(H_{\mu_N},\langle\cdot,\cdot\rangle_{\mu_N})$ is characterized by
\begin{equation*}
	V_N=\dom(A_N^{\frac{1}{2}}),\quad \langle A_Nu,v\rangle_{\mu_N}=\langle u,v\rangle\quad \text{for }u\in \dom(A_N),\, v\in V_N.
\end{equation*}
Of course, if $V_N$ is finite-dimensional, then $V_N=H_{\mu_N}$ and $A_N$ is a continuous symmetric operator on $H$.
The spectrum of $A_N$ is a pure point spectrum which consists of real, positive eigenvalues. 
Until the end of this section, we use the following notation:

$0<\lambda^{(1)}\leq\lambda^{(2)}\leq\dots$ denote the Eigenvalues of $A$.
Moreover, let $\{\lambda_N^{(i)}$ $|$ $i\in I_N\}$ (ordered analogously, i.e.~$\lambda_N^{(i)}\leq \lambda_N^{(j)}$ if $i\leq j$), where either $I_N=\{1,\dots,k_N\}$, for some $k_N\in\N$, or $I_N=\N$, denote the family of Eigenvalues of $A_N$ for $N\in\N$. We fix an orthonormal basis 
$\{\eta_N^{(i)}$ $|$ $i\in I_N\}$ of $(H_{\mu_N},\langle\cdot,\cdot\rangle_{\mu_N})$ such that $\eta_N^{(i)}\in\dom( A_N)$, $A_N\eta_N^{(i)}=\lambda_N^{(i)}\eta_N^{(i)}$ for $i\in I_N$.

In the next lemma, we reformulate a well-known consequence of the weak measure convergence of ${(\mu_N)}_N$ towards $\mu$, the convergence of the spectral structures of their covariance operators. The reader should be aware though, that the exact condition to characterize the weak convergence of Gaussian measures in terms of their covariance operators in a frame as ours, is a stronger notion of convergence: the one induced by the nuclear norm. For further reading on that topic, we refer to \cite{chevet83} and \cite{baushev87}. However, the statement of the next lemma is enough for our purpose and essential to the proof of the subsequent theorem.
We set $\lambda_N^{(i)}:= \infty$ and $\varphi_N^{(i)}:=0$ for $i\in\N\setminus I_N$ and remark
that \eqref{eqn:V2H} implies  $\sup(I_N)\to\infty$ for $N\to\infty$. Hence, for each $i\in\N$ the set $\{N\in\N$ $|$ $i\notin I_N\}$ is finite.
\begin{lemma}\label{lem:baus}
	\begin{thm(ii)}
		\item $\displaystyle \lim_{N\to\infty}\lambda_N^{(i)}=\lambda^{(i)}\quad\text{for  }i\in\N.$
		\item There exists a subsequence ${(\mu_{N_l})}_{l\in\N}$ of ${(\mu_N)}_{N\in\N}$ and an orthonormal basis $\{\eta^{(1)},\eta^{(2)},\dots\}$ of $(H_\mu,\langle\cdot,\cdot\rangle_\mu)$ such that $\eta^{(i)}\in\dom( A)$, $A\eta^{(i)}=\lambda^{(i)}\eta^{(i)}$, and
		\begin{equation*}
			\lim_{l\to\infty}\eta^{N_l}_i=\eta_i\text{ strongly in }H\text{ for } i\in\N.
		\end{equation*}
	\end{thm(ii)}
\end{lemma}
\begin{proof}
	The statement is exactly the content of \cite[Corollary 2.5]{kuwae03}, if we understand $(H_{\mu_N},\langle\cdot,\cdot\rangle_{\mu_N})$, $N\in\N$ as a sequence of converging Hilbert spaces with asymptotic space $(H_\mu,\langle\cdot,\cdot\rangle_\mu)$.
	There is a canonical way to do this, considering the convergence of the second moments
	\begin{equation*}
		\lim_{N\to\infty}\langle p_Nh,p_nh\rangle_{\mu_N}=
		\lim_{N\to\infty}\int_{H}\langle h,k\rangle^2\de\mu_N(k)=\int_H\langle h,k\rangle^2\de\mu(k)
		=\langle h,h\rangle_\mu
	\end{equation*}
	for $h\in H$.
	Moreover, by \eqref{eqn:V2H} states
	\begin{equation}
		\langle h,h\rangle=\lim_{N\to\infty}\langle p_Nh,p_Nh\rangle
	\end{equation}
	for $h\in H$.
	So, the second criterion of \cite[Def.~2.8]{kuwae03} is fulfilled considering the quadratic form $|\cdot|_H^2$ with domain $V_N$ on the Hilbert space $(H_{\mu_N},\langle\cdot,\cdot\rangle_{\mu_N})$ for $N\in\N$
	and the asymptotic form $|\cdot|_H^2$ with domain $H$ on the limiting Hilbert space $(H_\mu,\langle\cdot,\cdot\rangle_\mu)$.
	The assumptions of \cite[Corollary 2.5]{kuwae03} require the compact convergence of the spectral structure of $A_N$ towards the one of $A$ as $N\to\infty$. We claim the following: Given $u_N\in V_N$ for $N\in\N$ and $u\in H$, then the weak convergence of ${(u_N)}_N$ towards $u$ as a sequence in $(H,\langle\cdot,\cdot\rangle)$
	implies the strong convergence of ${(u_N)}_N$ towards $u$ in the topology of Kuwae-Shioya, \cite[Def.~2.4]{kuwae03}.
	This claim would imply the first condition for Mosco convergence, as specified in \cite[Def.~2.11]{kuwae03},
	and also verify asymptotic compactness as defined in \cite[Def.~2.12]{kuwae03}.
	So, let $u_N\in V_N$ for $N\in\N$ and $u\in H$  with $u_N\rightharpoonup u$ weakly in $H$. 
	We want to prove the strong convergence of ${(u_N)}_N$ towards $u$ in the strong topology of Kuwae-Shioya within the framework of converging Hilbert spaces $(H_{\mu_N},\langle\cdot,\cdot\rangle_{\mu_N})$ $\to$ $(H_\mu,\langle\cdot,\cdot\rangle_\mu)$.
	In view of \cite[Lem.~2.3]{kuwae03} it is enough to show
	\begin{equation}\label{eqn:firstIE}
		\lim_{N\to\infty}\int_{H}\langle u_N,k\rangle^2\de\mu_N(k)=\int_H\langle u,k\rangle^2\de\mu(k)
	\end{equation}
	(convergence of norms) and 
	\begin{equation}\label{eqn:secondIE}
		\lim_{N\to\infty}\int_{H}\langle u_N,k\rangle\langle h,k\rangle\de\mu_N(k)=\int_H\langle u,k\rangle\langle h,k\rangle\de\mu(k)
	\end{equation}
	for $h\in H$, as \eqref{eqn:secondIE} implies convergence w.r.t.~the weak topology of Kuwae-Shioya (see \cite[Def.~2.5]{kuwae03} in combination with \cite[Lem.~2.3]{kuwae03}).
	By \cite[Theorem 1]{chevet83} there exists $a\in(0,\infty)$ such that
	\begin{equation}\label{eqn:chevet}
		\sup_{N\in\N}\int_H\exp(a|k|_H^2)\de\mu_N(k)<\infty
	\end{equation}
	Let $\varepsilon>0$ and $h\in H$ be fixed.
	Because of \eqref{eqn:chevet} and the boundedness of ${(u_N)}_N$ in $H$, there exists $R\in(0,\infty)$ with
	\begin{equation*}
		\sup_{N\in\N}\int_{\{k\in H\,|\,|k|_H>R\}}|k|_H^2\de\mu_N(k)<\frac{\varepsilon}{\sup_{N\in\N}(|u_N|_H^2+|u_N|_H|h|_H)}.
	\end{equation*}
	In particular,
	\begin{equation*}
		\sup_{N\in\N}\int_{\{k\in H\,|\,|k|_H>R\}}\langle u_N,k\rangle^2\de\mu_N(k)<\varepsilon
	\end{equation*}
	as well as
	\begin{equation*}
			\sup_{N\in\N}\int_{\{k\in H\,|\,|k|_H>R\}}\big|\langle u_N,k\rangle\langle h,k\rangle\big|\de\mu_N(k)<\varepsilon.
	\end{equation*}
	Now, due to the weak measure convergence $\mu_N\to\mu$ and a 
	 $\varepsilon/3$-argument (see \cite[Theorem 2.2]{yang11}),
	 the integrals in \eqref{eqn:firstIE} and \eqref{eqn:secondIE} converge as desired if
	 the integrands converge uniformly on compact sets. 
	 This, however, is immediate from the weak convergence of ${(u_N)}_N$ towards $u$ in $H$.
	 Indeed, for an arbitrary compact  set $K\subset H$, the family $\{K\ni k\mapsto \langle u_N,k\rangle$ $|$ $N\in\N\}$ is equicontinuous and converges pointwisely to 
	 $K\ni k\mapsto \langle u,k\rangle$. Hence, the Theorem of Arzelà-Ascoli shows
	 \begin{equation*}
	 	\sup_{k\in K}\big|\langle u_N,k\rangle-\langle u,k\rangle\big|\overset{N\to\infty}{\longrightarrow}0.
	 \end{equation*}
     This concludes the proof.
\end{proof}

\begin{remark}\label{rem:remarki}
	\begin{thm(ii)}
    \item The only assumption of Lemma \ref{lem:baus} is the weak measure convergence of Gaussian measures on $H$. So, Item (ii) of that Lemma can also be read in the sense that for every subsequence of ${(\mu_N)}_{N\in\N}$ there exists a sub-subsequence with the stated properties.
	\item Let $i\in\N$.  We define 
	\begin{equation*}
	\varphi^{(i)}:=\frac{\eta^{(i)}}{\sqrt{\lambda^{(i)}}}\quad\text{and}\quad
	\varphi^{(i)}_N:=\frac{\eta_N^{(i)}}{\sqrt{\lambda_N^{(i)}}}
	\end{equation*}
	for all $N\in\N$ large enough such that $i\in I_N$.
	Those elements are normalized in $(H,\langle\cdot,\cdot\rangle)$, $|\varphi^{(i)}_N-\varphi^{(i)}|_H\to$ as $N\to\infty$ by Lemma \ref{lem:baus} and moreover 
	\begin{equation*}
	\lambda^{(i)}\cov\mu(\varphi^{(i)},h)=\langle A\varphi^{(i)},h\rangle_\mu=\langle\varphi^{(i)},h\rangle
	\end{equation*}
	for $h\in H$. Similarly, for $N\in\N$ large enough such that $i\in I_N$, we have
	\begin{align*}
		\lambda_N^{(i)}\cov{\mu_N} (\varphi_N^{(i)},h)&=\lambda_N^{(i)}\cov{\mu_N} (\varphi_N^{(i)},p_Nh)
		=\langle A_N\varphi_N^{(i)},p_Nh\rangle_{\mu_N}\\&=\langle \varphi_N,v\rangle=\langle \varphi_N,v+w\rangle
		=\langle \varphi_N^{(i)},p_nh\rangle=\langle \varphi_N^{(i)},h\rangle
	\end{align*}	
	for $h\in H$.
	Hence, Proposition \ref{prop:coMo} can be applied and 
	\begin{equation*}
	\big(\E^N_{\varphi^{(i)}_N},\dom\big(\E^N_{\varphi^{(i)}_N}\big)\big)\longrightarrow(\E_{\varphi^{(i)}},\dom(\E_{\varphi^{(i)}}))
	\end{equation*}
	in the sense of Mosco-Kuwae-Shioya.
	\end{thm(ii)}
\end{remark}

The infinite sum
\begin{align*}
	\E(u,v)&:=\sum_{i=1}^\infty\E_{\varphi^{(i)}}(u,v)\quad\text{for }u,v\text{ from the domain }\\
	\dom(\E)&:=\Big\{w\in \bigcap_{i\in\N}\dom(\E_{\varphi^{(i)}})\,\Big|\,\sum_{i=1}^\infty\E_{\varphi^{(i)}}(w,w)<\infty\Big\}
\end{align*}
defines a symmetric Dirichlet form on $L^2(H,\varrho\mu)$.

\begin{lemma}
		For  $u\in\dom(\E)$ there exists a weak gradient $\grad u\in L^2(H,H,\mu)$ which is characterized by
		\begin{equation*}
			\langle \eta,\grad u(h)\rangle=\lim_{t\downarrow 0}t^{-1}(u(h+t\eta)-u(h))
		\end{equation*}
		for $\mu$-a-e-~$h\in H$ and $\eta\in\dom(A)$. The limit on the right-hand side exists in $\mu$-a.e.~sense. It holds
		\begin{equation*}
			\E(u,v)=\int_H\langle \grad u,\grad v\rangle\varrho\de\mu
		\end{equation*}
		Moreover, $\mathcal F C_b^\infty(H)$ is dense in $(\dom(\E),\E_1)$.
\end{lemma}
\begin{proof}
	  The statement is a consequence of \cite[Thm.~1.9]{sheng92} together with \cite[Chap.~II, Thm.~4.3.2.]{bouleau91} an the fact that $\varrho$ is bounded from below and above by positive constants.
\end{proof}

For $N\in\N$ let $(\E^N,\dom(\E^N))$ denote the symmetric Dirichlet form on $L^2(H,\varrho_N\mu_N)$ which is the closure of
\begin{equation*}
	\E^N(u,v):=\int_H\langle p_N\grad u,p_N\grad v\rangle\de \mu_N\quad
	 \text{for }u,v\text{ from the pre-domain } \mathcal FC_b^\infty(H).
\end{equation*}

\begin{thm}\label{thm:maiMO}
	If Condition \ref{condi:last} holds, then
	$(\E,\dom(\E))$ is the limit of ${(\E^N,\dom(\E^N))}_{N\in\N}$
	in the sense of Mosco-Kuwae-Shioya within the framework of the converging Hilbert spaces $L^2(H,\varrho_N\mu_N)$
	$\to$ $L^2(H,\varrho\mu)$.
\end{thm}
\begin{proof}
	Let $u\in L^2(H,\varrho\mu)$ and $u_N\in L^2(H,\varrho_N\mu_N)$ for $N\in\N$ such that $u$ is the weak limit of $u$ in the sense of Kuwae-Shioya. We have to show
	\begin{equation*}
		\E(u,u)\leq\liminf_{N\to\infty}\E^N(u_N,u_N),
	\end{equation*}
	where $\E$ and $\E^N$ are extended as quadratic forms, assigning the value $\infty$ outside the domain of the respective Dirichlet form.
	It holds
	\begin{equation*}
		\E^{N}(u_N,u_N)=\sum_{i\in I_N}\E^N_{\varphi_N^{(i)}}(u_N,u_N)
	\end{equation*}
	extending the quadratic form associated with \vspace{2pt}$\big(\E^N_{\varphi_N^{(i)}},\dom(\E^N_{\varphi_N^{(i)}}\big)\big)$ in the same way, since $\varphi_N^{(i)}$, $i\in I_N$, is an orthonormal basis of $V_N$.
	We formally set \vspace{2pt}$\E^N_{\varphi^{(i)}_N}(\cdot,\cdot):=\infty$ on $L^2(H,\varrho_N\mu_N)$ if $i\in\N\setminus I_N$ and remark that for every $i\in\N$ the set $\{N\in\N$ $|$ $i\in\N\setminus I_N\}$ is finite.
	As a consequence of Proposition \ref{prop:coMo}, respectively Remark \ref{rem:remarki} (ii), we have
	\begin{align*}
		\E(u,u)&=\lim_{N'\to\infty}\sum_{i\in I_{N'}}\E_{\varphi^{(i)}}(u,u)\leq
		\liminf_{N'\to\infty}\Big(\sum_{i\in I_{N'}}\liminf_{N\to\infty}\E^N_{\varphi_N^{(i)}}(u_N,u_N)\Big)\\
		&\leq\liminf_{N'\to\infty}\liminf_{N\to\infty}\Big(\sum_{i\in I_{N'}}
		\E^N_{\varphi_N^{(i)}}(u_N,u_N)\Big)\\
		&\leq \liminf_{N\to\infty}\Big(\sum_{i\in I_{N}}\E^N_{\varphi_N^{(i)}}(u_N,u_N)\Big)
		=\liminf_{N\to\infty}\E^N(u_N,u_N)
	\end{align*}
	as desired.
	
	In view of \cite[Lem.~2.1(7)]{kuwae03}, the claim of this theorem follows if 
	$(\dom(\E^N),\E^N_1)$, $N\in\N$, form a sequence of converging Hilbert spaces with asymptotic element 
	$(\dom(\E),\E_1)$.
	Since
	\begin{equation*}
		\E(u,u)=\lim_{N\to\infty}\int_H\langle \grad u,\grad v\rangle\de \mu_N\geq\limsup_{N\to\infty}\E^N(u,u)
	\end{equation*}
	for $u\in\mathcal F C_b^\infty(H)$, we must have 
	\begin{equation*}
		\E(u,u)=\lim_{N\to\infty}\E^N(u,u).
	\end{equation*}
	This concludes the proof.
\end{proof}

\begin{corollary}\label{cor:mai}
	Let $g_N, g$ be functions of bounded variation, as in Condition \ref{condi:last}. For $N\in\N$ let ${(X^N_t)}_{t\geq 0}$ be the finite-dimensional system of skew interacting Brownian motions \eqref{eqn:skewIB}, defined in Example \ref{exa:moveto},
	and $({X_t)}_{t\geq 0}$ be the diffusion process of Remark \ref{rem:remmi} with Röckner-Zhu-Zhu decomposition \eqref{eqn:RZZ}. Through the argument of Remark \ref{rem:Argu}, we have shown the weak measure convergence of the equilibrium law of 
	\begin{equation*}
		u^N_t:=\Lambda_N X^{N}_{N^2t},\quad t\geq 0,
	\end{equation*}
	towards the equilibrium law of $({X_t)}_{t\geq 0}$ as $N\to\infty$. The topological space to which the term of weak measure convergence refers in this context is $C([0,\infty),H_0')$ (see Condition \ref{condi:tighti}). The statement  holds for any height map $\Lambda_N$ as in \eqref{eqn:heightDef}, if $\Xi$ meets Condition \ref{condi:Xi}.
\end{corollary}

	\section*{Appendix}

\begin{test}\label{lem:fukudec}
	Let $N\in\N$.
	We fix
	 a symmetric, positive, linear operator  $A_N:\R^{k_N}\to\R^{k_N}$,
	a function $g_{N,0}\in C_b^1(\R)$ and real numbers $y_j,\beta_j\in\R$ for $j=1,\dots,m$. 
 With
		\begin{equation*}
			g_N(y):=g_{N,0}(y)+\sum_{j=1}^m\beta_j\eins_{(-\infty,y_j]}(y),\quad y\in\R,
		\end{equation*}
		we define the Borel probability measure
		\begin{gather*}
			\de m_{A_N}(x):=\frac{1}{Z_N}\exp\Big(-x^\textnormal{T}A_Nx-\frac{1}{N^d}\sum_{i=1}^{k_N} g_N\big(N^{\frac{d}{2}-1}x_i\big)\Big)\de x,\\
			Z_N:=\int_{\R^{k_N}}\exp\Big(-x^\textnormal{T}A_Nx-\frac{1}{N^d}\sum_{i=1}^{k_N} g_N\big(N^{\frac{d}{2}-1}x_i\big)\Big)\de x,
		\end{gather*}
		on $\R^{k_N}$.
	Then, the Process ${(X^N_t)}_{t\geq 0}$ defined in Example \ref{exa:moveto} satisfies 
	\begin{equation*}
		\de X^{N,i}_t=-\big(A_N X^{N}_t\big)_i\de t-\frac{1}{2}N^{-\tfrac{d}{2}-1}\,g_{N,0}'\big(N^{\frac{d}{2}-1}X^{N,i}_t\big)\de t
		\\+\sum_{j=1}^m\tfrac{1-e^{-\beta_j/N^d}}{1+e^{-\beta_j/N^d}}\de l_t^{N,i,\tilde y_j}
		+\de B_t^{N,i}
	\end{equation*}
	$P_x^N$-a.s.~for each $x\in\R^{k_N}$,
	where ${(B_t^{N,i})}_{t\geq 0}$, $i=1,\dots, k_N$, are independent Brownian motions
	and $l_t^{N,i,\tilde y_j}$ is the local time of ${(X^{N,i}_t)}_{t\geq 0}$ at $\tilde y_j:=N^{1-\frac{d}{2}}y_j$.
\end{test}
\begin{proof} 
	The proof is analogous to \cite[Thm.~5.2]{bounebache14}. Our situation only is different from the cited theorem concerning scaling constants, the dimensions $k_N$, $d\in\N$, and the choice of $A_N$, which in both situations represents a positive, symmetric operator. These amendments do not touch the validity of the arguments from the proof of \cite[Thm.~5.2]{bounebache14}. Therefore, the statement of the Lemma follows directly. For the reader's convenience, we sketch the major steps of the proof.
	
	Step (1): For $y\in\R$ and $i\in\{1,\dots,k_N\}$ we define a finite Borel measure
	\begin{equation*}
		\tilde\delta_y^i(B):=\lim_{\varepsilon\to 0}\,\frac{1}{2\varepsilon }m_{A_N}\big(B\cap\big\{x\in\R^{k_N}\,\big|\,|x_i-y|\leq\varepsilon\big\}\big),\quad B\in\mathcal B(\R^{k_N}).
	\end{equation*}
    Then, analogously to \cite[Equation (5.6)]{bounebache14}, we obtain the following integration	by parts formula for $m_{A_N}$. For $i=1,\dots, k_N$ and $f\in C^1(\R^{k_N})$ with finite Lipschitz constant, it holds
	\begin{align}\label{eqn:ibpf}
		\int_{\R^{k_N}}\frac{\partial f}{\partial x_i}\de m_{A_N}=&-
		2\int_{\R^{k_N}}f(x) (A_Nx)_i\de m_{A_N}(x)\nonumber\\
		&+N^{-\frac{d}{2}-1}\int_{\R^{k_N}}f(x)
		g_{N,0}'\big(N^{\frac{d}{2}-1}x_i\big)\de m_{A_N}(x)\nonumber\\
		&+2\sum_{j=1}^m\frac{1-e^{-\beta_j/N^d}}{1+e^{-\beta_j/N^d}}
		\int_{\R^{k_N}}f(x)\de \tilde \delta_{N^{1-\frac{d}{2}}y_j}^i(x).
	\end{align}
	
	Step (2): Analogously as in the proof of \cite[Thm.~5.2]{bounebache14}, the positive, continuous additive functional  of ${(X^N_t)}_{t\geq 0}$, which satisfies the Revuz correspondence w.r.t.~$\tilde\delta_y^i$ for given $i=1,\dots,k_N$ and $y\in\R$ coincides with the local time ${(l_t^{N,i,y})}_{t\geq 0}$ of the component process ${(X^{N,i}_t)}_{t\geq 0}$ at $y$.
	
	Step (3): Let $(\E^N\dom(\E^N))$ be the Dirichlet form of Example \ref{exa:moveto} and   $u_i:\R^{k_N}\to\R$ denote the $i$-th coordinate projection for $i\in\{1,\dots,k_N\}$. Then, 
	$$\E^N(u_i,f)=\frac{1}{2}\int_{\R^{k_N}}\partial_i f\de m_{A_N}$$ for $f$ as in step (1). Using \eqref{eqn:ibpf} and (2), the Fukushima decomposition yields
	\begin{equation}\label{eqn:fuku}
		X_t^{N,i}-X_0^{N,i}=M_t^{u_i}+N_t^{u_i},\quad t\geq 0,\quad i=1,\dots,k_N,
	\end{equation}
	$P_x^N$-a.s.~for quasi-every $x\in\R^{k_N}$, where
	\begin{equation*}
		N_t^{u_i}=-\big(A_N X^{N}_t\big)_i\de t-\frac{1}{2}N^{-\tfrac{d}{2}-1}\,g_{N,0}'\big(N^{\frac{d}{2}-1}X^{N,i}_t\big)\de t
		\\+\sum_{j=1}^m\tfrac{1-e^{-\beta_j/N^d}}{1+e^{-\beta_j/N^d}}\de l_t^{N,i,\tilde y_j},
	\end{equation*} 
	$\tilde y_j:= N^{1-\frac{d}{2}}y_j$, and ${(M_t^{u_i})}_{t\geq 0}$ is the martingale additive functional with 
	\begin{equation}
	{\bracket {M^{u_i},M^{u_j}}_t}
	=2\int_0^t\Gamma_{N}(u_i,u_j)\big(X_s\big)\de s=t\delta_{ij},\quad t\geq 0,\quad i,j\in\{1,\dots,k_N\}.
\end{equation}
	
	 Step (4): Via compactness of the embedding $\dom(\E)\hookrightarrow L^2(\R^{k_N},m_{A_N})$ the absolute continuity criterion for the corresponding semigroup can be shown. Therefore,
	 we have the stronger statement that \eqref{eqn:fuku} indeed holds $P_x^N$-a.s.~for every $x\in\R^{k_N}$
\end{proof}

\begin{test}\label{lem:tightOne:majmin}
	Let ${(g_N)}_{N\in\N}$ either be a sequence of bounded, monotone decreasing functions, or a sequence of bounded, monotone increasing functions on $\R$ such that there exists a bounded function $g:\R\to\R$ with
	\begin{equation*}
		\sup_{x\in\R}\inf_{\substack{r\\|r|\leq\delta}}|g(x)-g_N(x+r)|\overset{N\to\infty}{\longrightarrow}0\quad\text{for every }\delta\in(0,\infty).
	\end{equation*}
	Let $m\in\N$. There exists a minorante $g_{N,m}^{\text{min}}\in C_b(\R)$ and a majorante $g_{N,m}^{\text{maj}}\in C_b(\R)$ for $N\in\N$ with
	\begin{equation*}
		-\|g_{N}\|_\infty\leq g_{N,m}^{\text{min}}(x)\leq g_N(x)\leq g_{N,m}^{\text{maj}}\leq\|g_N\|_\infty,\quad x\in\R,
	\end{equation*}
	which meet the following properties:
	
	There are asymptotic functions $g_m^{\text{min}}$, $g_m^{\text{maj}}\in C_b(\R)$ such that
	\begin{gather*}
		\lim_{N\to\infty}\|g_{N,m}^{\text{min}}-g_m^{\text{min}}\|_\infty=0,\quad \lim_{N\to\infty}\|g_{N,m}^{\text{maj}}-g_m^{\text{maj}}\|_\infty=0.
	\end{gather*}
	Moreover,
	\begin{gather*}
		\lim_{m\to\infty}g_m^{\text{min}}(x)=g(x),\quad \lim_{m\to\infty}g_m^{\text{maj}}(x)=g(x),\quad \text{if }g\text{ is continuous at the point }x\in\R.
	\end{gather*}
\end{test}
\begin{proof}
	We will construct the desired minorante and majorante for $g_N$ via a mollifying technique with uses a partition of unity. W.l.o.g.~let's assume we are in the monotone increasing case. Let $m\in\N$. Let $\tau_{m,i}$, $i\in\Z$, be a partition of union, subordinate to $(\frac{i-1}{m},\frac{i+1}{m})$, $i\in\Z$. Choose $-\frac{1}{m}<r(N,m,i)<\frac{1}{m}$ such that
	\begin{equation}\label{eqn:tightOne:11}
		\big|g(\frac{i}{m})-g_N\big(\frac{i}{m}+r(N,m,i)\big)\big|\leq\frac{1}{2N}+\inf_{\substack{r\\|r|\leq1/m}}\big|g\big(\frac{i}{m}\big)-g_N\big(\frac{i}{m}+r\big)\big|,\quad i,N\in\N.
	\end{equation}
	We define
	\begin{align*}
		g_{N,m}^{\text{min}}:=\sum_{i\in\Z}g_N\big(\frac{i}{m}+r(N,m,i)\big)\tau_{m,i+2},\quad g_{N,m}^{\text{maj}}:=\sum_{i\in\Z}g_N\big(\frac{i}{m}+r(N,m,i)\big)\tau_{m,i-2}
	\end{align*}
	for $N\in\N$ and 
	\begin{align*}
		g_{m}^{\text{min}}:=\sum_{i\in\Z}g\big(\frac{i}{m}\big)\tau_{m,i+2},\quad g_{m}^{\text{maj}}:=\sum_{i\in\Z}g\big(\frac{i}{m}\big)\tau_{m,i-2}.
	\end{align*}
	We remark that $\tau_{m,i}(x)\neq 0$ necessitates $i\in\{\lfloor mx\rfloor,\lfloor mx\rfloor+1\}$ for $i\in\Z$, $x\in\R$, and in particular, $\tau_{m,\lfloor mx\rfloor}(x)+\tau_{m,\lfloor mx\rfloor+1}(x)=1$. The functions defined above have the desired properties indeed. 
	It holds
	\begin{align*}
		g_{N,m}^{\text{min}}(x)&=g_N\Big(\frac{\lfloor mx\rfloor-2}{m}+r(N,m,\lfloor mx\rfloor-2)\Big)\tau_{m,\lfloor mx\rfloor}(x)\\&\phantom{=}+g_N\Big(\frac{\lfloor mx\rfloor-1}{m}+r(N,m,\lfloor mx\rfloor-1)\Big)\tau_{m,\lfloor mx\rfloor+1}(x)\\
		&\leq g_N\Big(\frac{\lfloor mx\rfloor-1}{m}\Big)\tau_{m,\lfloor mx\rfloor}(x)+g_N\Big(\frac{\lfloor mx\rfloor}{m}\Big)\tau_{m,\lfloor mx\rfloor+1}(x)\\
		&\leq g_N(x)(\tau_{m,\lfloor mx\rfloor}(x)+\tau_{m,\lfloor mx\rfloor+1}(x))= g_N(x),\quad x\in\R,N\in\N,
	\end{align*}
	and also
	\begin{align*}
		g_{N,m}^{\text{maj}}(x)&=g_N\Big(\frac{\lfloor mx\rfloor+2}{m}+r(N,m,\lfloor mx\rfloor+2)\Big)\tau_{m,\lfloor mx\rfloor}(x)\\&\phantom{=}+g_N\Big(\frac{\lfloor mx\rfloor+3}{m}+r(N,m,\lfloor mx\rfloor+3)\Big)\tau_{m,\lfloor mx\rfloor+1}(x)\\
		&\geq g_N\Big(\frac{\lfloor mx\rfloor+1}{m}\Big)\tau_{m,\lfloor mx\rfloor}(x)+g_N\Big(\frac{\lfloor mx\rfloor+2}{m}\Big)\tau_{m,\lfloor mx\rfloor+1}(x)\\
		&\geq g_N(x)(\tau_{m,\lfloor mx\rfloor}(x)+\tau_{m,\lfloor mx\rfloor+1}(x))= g_N(x),\quad x\in\R,N\in\N.
	\end{align*}
	Let $\varepsilon>0$. We choose $N_0$ is large enough such that $\inf_{r:|r|\leq 1/m}|g(x)-g_N(x+r)|<\frac{\varepsilon}{2}$ for every $x\in\R$ and $N\geq N_0$. For $N\geq \max\{N_0,\frac{1}{\varepsilon}\}$, using \eqref{eqn:tightOne:11}, it holds
	\begin{align*}
		|g_{N,m}^{\text{min}}(x)-g_{m}^{\text{min}}(x)|&\leq \Big|g_N\Big(\frac{\lfloor mx\rfloor-2}{m}+r(N,m,\lfloor mx\rfloor-2)\Big)-g\Big(\frac{\lfloor mx\rfloor-2}{m}\Big)\Big|\tau_{m,\lfloor mx\rfloor}(x)\\&\phantom{=}+\Big|g_N\Big(\frac{\lfloor mx\rfloor-1}{m}+r(N,m,\lfloor mx\rfloor-1)\Big)-g\Big(\frac{\lfloor mx\rfloor-1}{m}\Big)\Big|\tau_{m,\lfloor mx\rfloor+1}(x)\\
		&\leq\big(\frac{1}{2N}+\frac{\varepsilon}{2}\big)\tau_{m,\lfloor mx\rfloor}(x)+\big(\frac{1}{2N}+\frac{\varepsilon}{2}\big)\tau_{m,\lfloor mx\rfloor+1}(x)\\
		&\leq\varepsilon \tau_{m,\lfloor mx\rfloor}(x)+\varepsilon\tau_{m,\lfloor mx\rfloor+1}(x)=\varepsilon.
	\end{align*}
	Analogously, we get $|g_{N,m}^{\text{maj}}(x)-g_{m}^{\text{maj}}(x)|\leq\varepsilon$ for $x\in\R$ and $N\geq \max\{N_0,\frac{1}{\varepsilon}\}$.
	
	Moreover, let $g$ be continuous at $x\in\R$ and $m\in\N$ be large enough such that $|g(x)-g(y)|\leq\varepsilon $ if $y\in\R$, $|y-x|<\frac{3}{m}$, then
	\begin{align*}
		|g_{N}^{\text{min}}(x)-g(x)|&\leq \Big|g\Big(\frac{\lfloor mx\rfloor-2}{m}\Big)-g(x)\Big|\tau_{m,\lfloor mx\rfloor}(x)+\Big|g\Big(\frac{\lfloor mx\rfloor-1}{m}\Big)-g(x)\Big|\tau_{m,\lfloor mx\rfloor+1}(x)\\
		&\leq \varepsilon\tau_{m,\lfloor mx\rfloor}(x)+\varepsilon\tau_{m,\lfloor mx\rfloor+1}(x)=\varepsilon
	\end{align*}
	and also
	\begin{align*}
		|g_{N}^{\text{maj}}(x)-g(x)|&\leq \Big|g\Big(\frac{\lfloor mx\rfloor+2}{m}\Big)-g(x)\Big|\tau_{m,\lfloor mx\rfloor}(x)+\Big|g\Big(\frac{\lfloor mx\rfloor+3}{m}\Big)-g(x)\Big|\tau_{m,\lfloor mx\rfloor+1}(x)\\
		&\leq \varepsilon\tau_{m,\lfloor mx\rfloor}(x)+\varepsilon\tau_{m,\lfloor mx\rfloor+1}(x)=\varepsilon.
	\end{align*}

    In the case, where $g_N$, $N\in\N$, are monotone decreasing, we define
    \begin{align*}
    	g_{N,m}^{\text{min}}:=\sum_{i\in\Z}g_N\big(\frac{i}{m}+r(N,m,i)\big)\tau_{m,i-2},\quad g_{N,m}^{\text{maj}}:=\sum_{i\in\Z}g_N\big(\frac{i}{m}+r(N,m,i)\big)\tau_{m,i+2}
    \end{align*}
    for $N\in\N$ and 
    \begin{align*}
    	g_{m}^{\text{min}}:=\sum_{i\in\Z}g\big(\frac{i}{m}\big)\tau_{m,i-2},\quad g_{m}^{\text{maj}}:=\sum_{i\in\Z}g\big(\frac{i}{m}\big)\tau_{m,i+2},
    \end{align*}
    instead. The values of $r(N,m,i)$ for $N,m,i\in\N$ are chosen exactly as in the first case. From here, the argumentation is analogous. This concludes the proof.
\end{proof}

\begin{test}\label{lem:LVset}
	Let $\mu$ be a non-degenerate centred Gaussian measure on $H:=L^2(D,\de z)$ for a domain $D\subseteq\R^d$.
	For each value $a\in\R$ the level set $\{z\in D$ $|$ $h(z)=a\}$ has Lebesgue measure zero for $\mu$-a.e.~$h\in H$.
\end{test}
\begin{proof}
	The claim is that the $\de z$-class defined by the composition $\eins_{\{a\}}\circ h$
	vanishes in $\de z$-a.e.~sense for $\mu$-a.e.~$h\in H$. This is shown as follows.
	
	Let $\varphi\in H$ such that the image measure of $\mu$ under any shift $\tau_{s\varphi}h:=h+s\varphi$, $h\in H$, $s\in\R$, is absolutely continuous w.r.t.~$\mu$ itself, i.e.~$\mu\circ\tau_{s\varphi}^{-1}\ll\mu$. In that case, we immediately have $\mu\ll\mu\circ\tau_{s\varphi}^{-1}$, hence the equivalence (mutual absolute continuity) $\mu\sim\mu\circ\tau_{s\varphi}^{-1}$, for $s\in\R$. Moreover, defining the measure
	\begin{equation*}
		\sigma_\varphi(B):=\int_\R\mu\circ\tau_{s\varphi}^{-1}(B)\de s,\quad B\subseteq H\text{ Borel measurable},
	\end{equation*}
	it holds $\sigma_\varphi\sim\mu$. Moreover,
	\begin{equation*}
		\int_H u\de\mu=\int_H\int_{\R} u(h+s\varphi)\frac{\de\mu}{\de\sigma_\varphi}(h+s\varphi)\de s\de\mu(h)
	\end{equation*}
	for any Borel measurable function $u:H\to[0,\infty)$.
	Now, let 
	$\tilde\varphi:D\to\R$ be a representative for $\varphi$ and $A\subseteq D$ be a Borel measurable set such that $\tilde\varphi(z)\neq 0$ for $z\in A$. Since 
	any singleton is negligible w.r.t.~the Lebesgue measure, Fubini's theorem yields
	\begin{multline}\label{eqn:tightOne:19}
		\int_H\int_A\eins_{a}(h(z))\de z\de\mu(h)=\int_H\int_{\R}\int_A\eins_{a}\big(h(z)+s\tilde \varphi(z)\big)\de z\frac{\de\mu}{\de\sigma_\varphi}(h+s\varphi) \de s\de\mu(h)\\
		=\int_H\int_A\int_\R\eins_{\tilde \varphi(z)^{-1}(a-h(z))}(s)\frac{\de\mu}{\de\sigma_\varphi}(h+s\varphi)\de s\de z\de\mu(h)=0.
	\end{multline}
	Due to the Cameron-Martin formula, the space $V:=\{\varphi\in H$ $|$ $\mu\circ\tau_{s\varphi}^{-1}\ll\mu$ for $s\in\R\}$ is dense in $H$, as is shown in 
	\cite[Theorems 3.1.2 \& 3.1.3  of Chapter II]{bouleau91}. Hence, we can find an orthonormal basis $\varphi_1,\varphi_2,\dots$ of elements from $V$ and define Borel measurable subsets $A_1,A_2,\dots$ of $D$ as $A_i:=\{z\in D\,|\,\tilde\varphi_i(z)\neq 0\}$ for some $\de z$-version $\tilde\varphi_i$ of $\varphi_i$. The right-hand-side of \eqref{eqn:tightOne:19} is independent of $\varphi$ and in this case yields
	\begin{equation*}
		\int_H\int_{A_i}\eins_{a}(h(z))\de z\de\mu(h)=0,\quad i\in\N.
	\end{equation*}
	Let $Q:=D\setminus \big(\bigcup_{i\in\N}A_i\big)$. Since $\langle\eins_Q,\varphi_i\rangle=0$ for every $i\in\N$, the set $Q$ has Lebesgue measure zero. Therefore,
	\begin{align*}
		\int_H\int_{D}\eins_{a}(h(z))\de z\de\mu(h)&=\int_H\int_{\bigcup_{i\in\N}A_i}\eins_{a}(h(z))\de z\de\mu(h)
		\\&\leq\sum_{i=1}^\infty\int_H\int_{A_i}\eins_{a}(h(z))\de z\de\mu(h)=0.
	\end{align*}
	This concludes the proof.
\end{proof}
	\printbibliography

@book{dynkin2012markov,
	title={Markov Processes: Volume I},
	author={Dynkin, E.B.},
	isbn={9783662000311},
	lccn={65001759},
	series={Die Grundlehren der mathematischen Wissenschaften},
		volume = {121},
	% url={https://books.google.com.hk/books?id=vHrpCAAAQBAJ},
	year={1965},
	publisher={Springer Verlag}
}

@article{ROCKNER19921,
	title = {Tightness of general C1, p capacities on Banach space},
	journal = {Journal of Functional Analysis},
	volume = {108},
	number = {1},
	pages = {1-12},
	year = {1992},
	issn = {0022-1236},
	%doi = {https://doi.org/10.1016/0022-1236(92)90144-8},
	%url = {https://www.sciencedirect.com/science/article/pii/0022123692901448},
	author = {Michael R\"{o}ckner and Byron Schmuland}
}

@article{Ro2zhu,
	%ISSN = {00911798},
	%URL = {http://www.jstor.org/stable/23248488},
	author = {Michael Röckner and Rong-Chan Zhu and Xiang-Chan Zhu},
	journal = {The Annals of Probability},
	number = {4},
	pages = {1759--1794},
	publisher = {Institute of Mathematical Statistics},
	title = {The stochastic reflection problem on an infinite dimensional convex set and BV functions in a Gelfand triple},
	urldate = {2023-12-04},
	volume = {40},
	year = {2012}
}

@article{KOLESNIKOV06,
	title = {Mosco convergence of Dirichlet forms in infinite dimensions with changing reference measures},
	journal = {Journal of Functional Analysis},
	volume = {230},
	number = {2},
	pages = {382-418},
	year = {2006},
	issn = {0022-1236},
	%doi = {https://doi.org/10.1016/j.jfa.2005.06.002},
	%url = {https://www.sciencedirect.com/science/article/pii/S0022123605002168},
	author = {Alexander V. Kolesnikov}
}

@phdthesis{Wit23,
	author      = {Simon Wittmann},
	title       = {Large scale asymptotics for Markov processes in the analytic framework of Mosco-Kuwae-Shioya},
	type        = {Doctoral Thesis},
%	pages       = {IV, 117},
	school      = {Rheinland-Pf{\"a}lzische Technische Universit{\"a}t Kaiserslautern-Landau},
%	doi       = {10.26204/KLUEDO/7190},
%	url  = {https://nbn-resolving.de/urn:nbn:de:hbz:386-kluedo-71900},
	year        = {2023},
}

@book {fukushima11,
	AUTHOR = {Fukushima, Masatoshi and Oshima, Yoichi and Takeda, Masayoshi},
	TITLE = {Dirichlet Forms and Symmetric Markov Processes},
	SERIES = {De Gruyter Studies in Mathematics},
	VOLUME = {19},
	EDITION = {2nd revised and extended edition},
	PUBLISHER = {Walter de Gruyter \& Co., Berlin},
	YEAR = {2011}
}

@book {kurtz86,
	AUTHOR = {Ethier, Stewart N. and Kurtz, Thomas G.},
	TITLE = {Markov Processes},
	SERIES = {Wiley Series in Probability and Mathematical Statistics},
	TITLEADDON = {Characterization and convergence},
	PUBLISHER = {John Wiley \& Sons, Inc., New York},
	YEAR = {1986}
}

@book {bouleau91,
	AUTHOR = {Bouleau, Nicolas and Hirsch, Francis},
	TITLE = {Dirichlet Forms and Analysis on Wiener Space},
	SERIES = {De Gruyter Studies in Mathematics},
	VOLUME = {14},
	PUBLISHER = {Walter de Gruyter \& Co., Berlin},
	YEAR = {1991}
}

@online{me21,
	author = {Grothaus, Martin and Wittmann, Simon},
	title = {Mosco convergence of gradient forms with non-convex interaction potential},
	archivePrefix = {arXiv},
	eprint = {2105.05140},
	primaryClass = {math.PR},
	Year = {2021}
}

@article {zambotti05,
	AUTHOR = {Deuschel, Jean-Dominique and Giacomin, Giambattista and
		Zambotti, Lorenzo},
	TITLE = {Scaling limits of equilibrium wetting models in
		{$(1+1)$}-dimension},
	JOURNAL = {Probab. Theory Relat. Fields},
	FJOURNAL = {Probability Theory and Related Fields},
	VOLUME = {132},
	NUMBER = {4},
	PAGES = {471--500},
	YEAR = {2005},
	%DOI = {10.1007/s00440-004-0401-8},
	%URL = {https://doi.org/10.1007/s00440-004-0401-8},
}

@book {royden88,
	AUTHOR = {Royden, H. L.},
	TITLE = {Real Analysis},
	EDITION = {3},
	PUBLISHER = {Macmillan Publishing Company, New York},
	YEAR = {1988}
}

@book {ma92,
	AUTHOR = {Ma, Zhi Ming and R\"{o}ckner, Michael},
	TITLE = {Introduction to the Theory of (Non-Symmetric) Dirichlet
		Forms},
	SERIES = {Universitext},
	PUBLISHER = {Springer-Verlag, Berlin},
	YEAR = {1992},
	%       DOI = {10.1007/978-3-642-77739-4},
	%       URL = {https://doi.org/10.1007/978-3-642-77739-4},
}

@article {albeverio90,
	AUTHOR = {Albeverio, Sergio and R\"{o}ckner, Michael},
	TITLE = {Classical {D}irichlet forms on topological vector
		spaces---closability and a {C}ameron-{M}artin formula},
	JOURNAL = {J. Funct. Anal.},
	FJOURNAL = {Journal of Functional Analysis},
	VOLUME = {88},
	YEAR = {1990},
	NUMBER = {2},
	%DOI = {10.1016/0022-1236(90)90113-Y},
	%URL = {https://doi.org/10.1016/0022-1236(90)90113-Y},
}

@article {debussche07,
	AUTHOR = {Debussche, Arnaud and Zambotti, Lorenzo},
	TITLE = {Conservative stochastic {C}ahn-{H}illiard equation with
		reflection},
	JOURNAL = {Ann.~Probab.},
	FJOURNAL = {The Annals of Probability},
	VOLUME = {35},
	YEAR = {2007},
	NUMBER = {5},
	PAGES = {1706--1739},
	%       DOI = {10.1214/009117906000000773},
	%       URL = {https://doi.org/10.1214/009117906000000773},
}

@article {kondratiev03,
	AUTHOR = {Grothaus, Martin and Kondratiev, Yuri G. and Lytvynov, Eugene
		and R\"{o}ckner, Michael},
	TITLE = {Scaling limit of stochastic dynamics in classical continuous
		systems},
	JOURNAL = {Ann.~Probab.},
	FJOURNAL = {The Annals of Probability},
	VOLUME = {31},
	YEAR = {2003},
	NUMBER = {3},
	PAGES = {1494--1532},
	%       DOI = {10.1214/aop/1055425788},
	%       URL = {https://doi.org/10.1214/aop/1055425788},
}

@article {kondratiev07,
	AUTHOR = {Grothaus, Martin and Kondratiev, Yuri G. and R\"{o}ckner, Michael},
	TITLE = {N/V-limit for stochastic dynamics in continuous particle
		systems},
	JOURNAL = {Probab. Theory Relat. Fields}, 
	FJOURNAL = {Probability Theory and Related Fields},
	VOLUME = {137},
	YEAR = {2007},
	NUMBER = {1-2},
	PAGES = {121--160},
	%      DOI = {10.1007/s00440-006-0499-y},
	%      URL = {https://doi.org/10.1007/s00440-006-0499-y},
}

@article {chevet83,
	AUTHOR = {Chevet, Simone},
	TITLE = {Sur les suites de mesures gaussiennes \'{e}troitement
		convergentes},
	JOURNAL = {C. R. Acad. Sci. Paris S\'{e}r. I Math.},
	FJOURNAL = {Comptes Rendus des S\'{e}ances de l'Acad\'{e}mie des Sciences. S\'{e}rie
		I. Math\'{e}matique},
	VOLUME = {296},
	YEAR = {1983},
	NUMBER = {4},
	PAGES = {227--230}
}

@article {baushev87,
	AUTHOR = {Baushev, A. N.},
	TITLE = {On the weak convergence of {G}aussian measures},
	JOURNAL = {Teor. Veroyatnost. i Primenen.},
	FJOURNAL = {Akademiya Nauk SSSR. Teoriya Veroyatnoste\u{\i} i ee Primeneniya},
	VOLUME = {32},
	YEAR = {1987},
	NUMBER = {4},
	PAGES = {734--742}
}

@article {kuwae03,
	AUTHOR = {Kuwae, Kazuhiro and Shioya, Takashi},
	TITLE = {Convergence of spectral structures: a functional analytic
		theory and its applications to spectral geometry},
	JOURNAL = {Comm. Anal. Geom.},
	FJOURNAL = {Communications in Analysis and Geometry},
	VOLUME = {11},
	YEAR = {2003},
	NUMBER = {4},
	PAGES = {599--673}
	%       DOI = {10.4310/CAG.2003.v11.n4.a1},
	%       URL = {https://doi.org/10.4310/CAG.2003.v11.n4.a1},
}

@article {kolesnikov05,
	AUTHOR = {Kolesnikov, A. V.},
	TITLE = {Convergence of {D}irichlet forms with changing speed measures
		on {$\mathbb R^d$}},
	JOURNAL = {Forum Mathematicum},
	VOLUME = {17},
	YEAR = {2005},
	NUMBER = {2},
	PAGES = {225-259}
	%       DOI = {10.1515/form.2005.17.2.225},
	%       URL = {http://dx.doi.org/10.1515/form.2005.17.2.225},
}

@Mastersthesis{toelle06,
	author = {T{\"o}lle, J. M.},
	title = {Convergence of non-symmetric forms with changing reference measures},
	school = {University of Bielefeld},
	year = {2006},
	type = {Diploma thesis}
}

@article {yang11,
	AUTHOR = {Yang, Xiangfeng},
	TITLE = {Integral convergence related to weak convergence of measures},
	JOURNAL = {Appl. Math. Sci. (Ruse)},
	FJOURNAL = {Applied Mathematical Sciences},
	VOLUME = {5},
	YEAR = {2011},
	NUMBER = {53-56},
	PAGES = {2775--2779}
}

@article {sheng92,
	AUTHOR = {R\"{o}ckner, Michael and Zhang, Tu Sheng},
	TITLE = {Uniqueness of generalized {S}chr\"{o}dinger operators and
		applications},
	JOURNAL = {J. Funct. Anal.},
	FJOURNAL = {Journal of Functional Analysis},
	VOLUME = {105},
	YEAR = {1992},
	NUMBER = {1},
	PAGES = {187--231}
	%       DOI = {10.1016/0022-1236(92)90078-W},
	%       URL = {https://doi.org/10.1016/0022-1236(92)90078-W},
}

@article{bounebache14,
	title = {A Skew Stochastic Heat Equation},
	journal = {Journal of Theoretical Probability},
	volume = {27},
	number = {1},
	pages = {168–201},
	year = {2014},
	author = {S. K. Bounebache and L. Zambotti}
}

@article {cipriani20,
	AUTHOR = {Cipriani, Alessandra and Dan, Biltu and Hazra, Rajat Subhra},
	TITLE = {Scaling limit of semiflexible polymers: a phase transition},
	JOURNAL = {Comm. Math. Phys.},
	FJOURNAL = {Communications in Mathematical Physics},
	VOLUME = {377},
	YEAR = {2020},
	NUMBER = {2},
	PAGES = {1505--1544}
	%     DOI = {10.1007/s00220-020-03762-9},
	%      URL = {https://doi.org/10.1007/s00220-020-03762-9},
}

@Inbook{Dembo2005,
	author="Dembo, Amir
	and Funaki, Tadahisa",
	editor="Picard, Jean",
	title="Stochastic Interface Models",
	bookTitle="Lectures on Probability Theory and Statistics: Ecole d'Et{\'e} de Probabilit{\'e}s de Saint-Flour XXXIII - 2003",
	year="2005",
	publisher="Springer Berlin Heidelberg",
	address="Berlin, Heidelberg",
	pages="103--274",
	isbn="978-3-540-31537-7",
%	doi="10.1007/11429579_2",
%	url="https://doi.org/10.1007/11429579_2"
}

@article {Zambotti04,
	AUTHOR = {Zambotti, Lorenzo},
    TITLE = {Fluctuations for a $\nabla\Phi$ interface model with repulsion from a wall},
	JOURNAL = {Probab.~Theory Relat.~Fields},
	FJOURNAL = {Communications in Mathematical Physics},
	VOLUME = {129},
	YEAR = {2004},
	NUMBER = {3},
	PAGES = {315--129}
	%     DOI = {10.1007/s00440-004-0335-1},
	%      URL = {https://doi.org/10.1007/s00440-004-0335-1},
}

@article {mosco94,
	author = {U. Mosco},
	title = {Composite media and asymptotic {D}irichlet forms},
	journal = {Journal of Functional Analysis},
	year = {1994},
	volume = {123},
	number = {2},
	pages = {368--421},
%	doi = {10.1006/jfan.1994.1093},
%	url = {http://dx.doi.org/10.1006/jfan.1994.1093},
}

\end{document}